%% file: main.tex
\documentclass[11pt]{amsart}

\makeatletter
\def\@tocline#1#2#3#4#5#6#7{\relax
    \ifnum #1>\c@tocdepth 
    \else
    \par \addpenalty\@secpenalty\addvspace{#2}%
    \begingroup \hyphenpenalty\@M
    \@ifempty{#4}{%
        \@tempdima\csname r@tocindent\number#1\endcsname\relax
    }{%
        \@tempdima#4\relax
    }%
    \parindent\z@ \leftskip#3\relax \advance\leftskip\@tempdima\relax
    \rightskip\@pnumwidth plus4em \parfillskip-\@pnumwidth
    #5\leavevmode\hskip-\@tempdima
    \ifcase #1
    \or\or \hskip 1em \or \hskip 2em \else \hskip 3em \fi%
    #6\nobreak\relax
    \dotfill\hbox to\@pnumwidth{\@tocpagenum{#7}}\par
    \nobreak
    \endgroup
    \fi}
\makeatother

\usepackage[colorlinks=true]{hyperref}
\hypersetup{urlcolor=blue, citecolor=red}

\usepackage{amssymb}
\usepackage{amsmath}
\usepackage{amscd}
\usepackage{curves}
\usepackage{epsfig}
\usepackage{bbm}
\usepackage{geometry}
\usepackage{graphicx}
\usepackage{pst-plot}
\usepackage{amsfonts}
\usepackage{enumerate}
\usepackage{mathtools}

\numberwithin{equation}{section}

\newtheorem{thm}{Theorem}[section]

\newtheorem{lm}{Lemma}[section]
\newtheorem{pp}{Proposition}[section]
\newtheorem{df}{Definition}[section]
\newtheorem{theorem}{Theorem}[section]
\newtheorem{lemma}[theorem]{Lemma}
\newtheorem{prop}[theorem]{Proposition}

\newtheorem{definition}[theorem]{Definition}

\newtheorem{remark}[theorem]{Remark}
\newtheorem{assumption}[theorem]{Assumption}

\input{abbreviations.tex}

\input{abrv2.tex}

\begin{document}
\title[Inverse Nonlinear Scattering] {Inverse Nonlinear Scattering by a Metric}
\author[Hintz]{Peter Hintz}
\address{Peter Hintz \newline
    \indent ETH Zurich, Department of Mathematics, Group 5 \newline
    \indent HG G 62.2 R\"amistrasse 101 8092 \newline
    \indent Zurich, Switzerland}
\email{peter.hintz@math.ethz.ch}
\author[S\'a Barreto]{Ant\^onio S\'a Barreto}
\address{Ant\^onio S\'a Barreto \newline
    \indent Department of Mathematics, Purdue University \newline
    \indent 150 North University Street, West Lafayette IN 47907}
\email{sabarre@purdue.edu}
\author[Uhlmann]{Gunther Uhlmann}
\address{Gunther Uhlmann \newline\indent  Department of Mathematics, University of Washington, Seattle, WA 98195}
\email{gunther@math.washington.edu}
\author[Zhang] {Yang Zhang}
\address{Yang Zhang \newline\indent  Department of Mathematics, University of California, Irvine, CA 92697}
\email{yangz79@uci.edu}

\begin{abstract} We study the inverse problem of determining a time-dependent globally hyperbolic Lorentzian metric from the scattering operator for semilinear wave equations.
\end{abstract}
\maketitle

\tableofcontents


\input{intro}
\input{forwardproblem}

\input{preliminary}

\input{constructwaves}
\input{layerstripping}

\input{simple_scattering_light}

\input{reconstructionscheme}





\input{asymptotic}

\subsection*{Acknowledgment}
ASB was partly supported by the Simons Foundation grant.
GU was partly supported by NSF.
\begin{footnotesize}
    \bibliographystyle{plain}
    \bibliography{microlocal_analysis}
\end{footnotesize}

\end{document}

%% file: abbreviations.tex
\newcommand{\tUmi}{{U}_{-}^{\mathrm{in}}}
\newcommand{\tUpl}{{U}_{+}^{\mathrm{in}}}

\newcommand{\vjone}{v_{j,1}}
\newcommand{\vjtwo}{v_{j,2}}

\newcommand{\mupl}{\mu^+}
\newcommand{\mumn}{\mu^-}
\newcommand{\CP}{CP}

\newcommand{\LUpl}{L(U_+)}
\newcommand{\LUmi}{L(U_-)}

\newcommand{\Vmi}{V_-}

\newcommand{\Vmione}{V_{-,1}}

\newcommand{\lN}{\mathcal{N}}

\newcommand{\lVsmooth}{\mathcal{V}^\mathrm{\infty}}

\newcommand{\Jpl}{J^+}
\newcommand{\Jmi}{J^-}
\newcommand{\vout}{v_\mathrm{out}}
\newcommand{\vin}{v_\mathrm{in}}
\newcommand{\Coutpl}{C_{\mathrm{out},+}}
\newcommand{\Cinpl}{C_{\mathrm{in},+}}
\newcommand{\Coutmi}{C_{\mathrm{out},-}}
\newcommand{\Cinmi}{C_{\mathrm{in},-}}
\newcommand{\Ppl}{P_{+}}
\newcommand{\Pmi}{P_{-}}

\newcommand{\tU}{\tilde{U}}
\newcommand{\lV}{\mathcal{V}}

\newcommand{\teal}{\color{teal}}

\newcommand{\tp}{\tilde{p}}
\newcommand{\tv}{\tilde{v}}

\newcommand{\lL}{\mathcal{L}}
\newcommand{\lM}{\mathcal{M}}

\newcommand{\lR}{\mathcal{R}}
\newcommand{\lC}{\mathcal{C}}
\newcommand{\lE}{\mathcal{E}}
\newcommand{\lB}{\mathcal{B}}

\newcommand{\tups}{\tilde{\ups}}

\newcommand{\lCear}{\mathcal{C}^{\mathrm{ear}}}
\newcommand{\lCsmooth}{\mathcal{C}^{\infty}}
\newcommand{\lCreg}{\mathcal{C}^{\mathrm{reg}}}

\newcommand{\Cear}{{C}^{\mathrm{ear}}}

\newcommand{\Creg}{{C}^{\mathrm{reg}}}

\newcommand{\Spl}{S_+}
\newcommand{\Smi}{S_-}

\newcommand{\CUjreg}{{C}^{\mathrm{reg}}_{U^{(j)}}}

\newcommand{\pmi}{p_-}
\newcommand{\ppl}{p_+}

\newcommand{\imi}{i_-}
\newcommand{\ipl}{i_+}

\newcommand{\yzeropl}{y_{+,0}}
\newcommand{\yzeromi}{y_{-,0}}

\newcommand{\M}{M}

\newcommand{\mR}{\mathbb{R}}
\newcommand{\mS}{\mathbb{S}}

\newcommand{\sq}{\square}
\newcommand{\tghat}{\hat{g}_\mathrm{e}}

\newcommand{\Omgin}{\Omega_{\mathrm{in}}}

\newcommand{\lA}{\mathcal{A}}

\newcommand{\varsm}{\varsigma}
\newcommand*\dif{\mathop{}\!\mathrm{d}}

\newcommand{\ep}{\epsilon}

\newcommand{\zi}{{\zeta^{i}}}
\newcommand{\zj}{{\zeta^{j}}}

\newcommand{\nxxi}{{\mathcal{N}(\vec{p}, \vec{w})}}

\newcommand{\capgammathree}{{\cap_{j=1}^{3}\gamma_{\ups_j}(\mathbb{R}_+)}}

\newcommand{\xxii}{\ups_i}

\newcommand{\xxiip}{\ups'_i}

\newcommand{\CUreg}{C^\mathrm{reg}_U}
\newcommand{\CUear}{C^\mathrm{ear}_U}
\newcommand{\LUreg}{L^\mathrm{reg}_U}
\newcommand{\LUear}{L^\mathrm{ear}_U}
\newcommand{\LUsmooth}{L^\infty_U}
\newcommand{\SUreg}{\mathcal{C}^\mathrm{reg}_U}
\newcommand{\SUear}{\mathcal{C}^\mathrm{ear}_U}
\newcommand{\SUsmooth}{\mathcal{L}^\infty_U}

\newcommand{\SUone}{\mathcal{C}^\mathrm{ear}_{U^{(1)}}}
\newcommand{\SUtwo}{\mathcal{C}^\mathrm{ear}_{U^{(2)}}}

\newcommand{\lU}{\mathcal{U}}

\newcommand{\tw}{{\widetilde{w}}}

\newcommand{\wfset}{{\text{\( \WF \)}}}

\newcommand{\intM}{M}
\newcommand{\Char}{\mathrm{Char}}

\newcommand{\Ical}{{I}}

\newcommand{\Ir}{\mathring{\mathcal{I}}}

\newcommand{\sigmp}{{ \sigma_{p}}}

\newcommand{\pM}{\partial M}

\newcommand{\tM}{M_{\text{e}}}

\newcommand{\tN}{\tilde{N}}

\newcommand{\yb}{y_|}
\newcommand{\etab}{\eta_|}

\newcommand{\Gammaset}{(\cap_{j=1}^{3}\gamma_{\ups_j}(\mathbb{R}_+)) \cap \gamma_{\ups_0}(\mathbb{R}_-)}

\newcommand{\RN}[1]{%
    \textup{\uppercase\expandafter{\romannumeral#1}}%
}

\DeclareMathOperator{\WF}{WF}
\DeclareMathOperator{\supp}{supp}

\newcommand{\rv}{}

\newcommand{\ba}{\[\begin{aligned}}
    \newcommand{\ea}{\end{aligned}\]}



%% file: abrv2.tex
\def \scp {S_+}
\def \scm {S_-}
\def \scmrho {S_{-,\rho}}
\def \scprho {S_{+,\rho}}
\def \bpf {\begin{proof}}
    \def \epf {\end{proof}}
\def \beq {\begin{equation}}
    \def \eeq {\end{equation}}
\def \bsp{\begin{split}}
    \def \esp{\end{split}}

\def \wt {\widetilde}

\def \mce {{\mathcal E}}

\def \mcf {{\mathcal F}}

\def \mch {{\mathcal H}}

\def \mcl {{\mathcal L}}

\def \mcv {{\mathcal V}}

\def \mr {{\mathbb R}}
\def \mn {{\mathbb N}}

\def \mfp {{\mathfrak p}}

\def \mfU {{\mathfrak U}}

\def\ha {\frac{1}{2}}

\def \ka {\kappa}

\def \loc {\operatorname{loc}}

\def \WF {\operatorname{WF}}

\def \supp {\text{supp }}

\def \mcn {\mathcal{N}}

\def \ga {{\gamma}}
\def \eps {\varepsilon}
\def \ups {\upsilon}
\def \Ups {\Upsilon}

\def \La {\Lambda}

\def \p {\partial}

\def \beqq {\begin{equation}}
    
    \def \eeqq {\end{equation}}

\def \Diff {\operatorname{Diff}}
\numberwithin{equation}{section}

%% file: intro.tex
\section{Introduction}

Scattering in physics and mathematics is the comparison of the behavior of waves before and after they interact  with a medium.  There is finite time scattering in which one  compares properties of the wave for $t<-T$ and for $t>T,$ for some finite time $T>0$.  The global scattering experiment compares the asymptotic behavior of waves as time  $t\rightarrow -\infty$  and $t \rightarrow +\infty.$  The inverse scattering problem consists  obtaining information about the medium with which the waves interact from such experiments.
Both scattering and inverse scattering have been very well studied for a variety of evolution equations, including Schr\"odinger and wave equations, and we refer the reader to \cite{IsoKurLas,Uhl,Uhl1} for an account.

We are concerned with the problem of determining the time-dependent coefficients of nonlinear wave equations from scattering data.
In more precise terms, one wants to recover a globally hyperbolic Lorentzian metric from the scattering operator, modulo natural obstructions.
For such problems,
perhaps the most basic model is to determine a time-independent Riemannian metric $g$  in $\mr^n,$ with suitable rate of decay at infinity, from the global scattering operator for the wave equation, $\partial_t^2-\Delta_g,$ again modulo some natural obstructions.
Here $\Delta_g$ is the Laplace-Beltrami operator induced by the Riemannian metric.
This scattering operator can be precisely defined in several different ways,
and the most closely related to what we define here was introduced by Friedlander \cite{Fri,Fri1}, using the radiation fields,  following the work of Penrose \cite{Pen1,Pen2}, see also \cite{SaB, SaBwun}.
If the metric $g$  is a compactly supported perturbation of the Euclidean metric, this problem can be reduced to studying the Dirichlet-to-Neumann map, and it has been solved by Belishev and Kurylev \cite{BelKur} using the boundary control method, which relies on a theorem of Tataru \cite{Tat,Tat1}, see also \cite{HorUC}.  However, for arbitrary non-compactly supported perturbations, as far as we know, this problem remains unsolved, even if the metric decays very fast at infinity.

In this work, we study a more general form of this problem, but for a semilinear wave equation.
Although nonlinear equations are typically more challenging to analyze,
the interaction of waves due to the nonlinearity carry additional information about the medium.
The central question is how to measure and interpret this information and how to use it to determine properties about the medium.
Kurylev, Lassas and Uhlmann \cite{Kurylev2018} were the first to observe that this nonlinearity can be used to obtain information about the medium.   This was done for a finite time scattering experiment called the source-to-solution map.
They proved that this source-to-solution map determines a globally Lorenztian metric, modulo conformal factors and diffeomorphisms, in a Lorentzian manifold without boundary.
The main idea is to use a multi-fold linearization and the nonlinear interaction of waves.
This comes from the observation that the transversal interaction of three or four conormal waves for a nonlinear wave equation produces new singularities from the set of interaction.
Such results go back to the work  Beals \cite{Beals}, Bony \cite{Bony4},
Melrose and Ritter \cite{MelRit} and Rauch and Reed \cite{RauRee}.
Specifically, by choosing specially designed sources,
one can produce new singularities caused by the interaction of linear waves, and then detect them from the measurements.
Information about the metric and nonlinearity is encoded in these new singularities.
One can extract such information from the principal symbol of the new singularities, using the calculus of conormal distributions and paired Lagrangian distributions.

Starting with \cite{Kurylev2018, Kurylev2014a},
there are many works studying inverse problems for nonlinear hyperbolic equations, see
\cite{Barreto2021,Barreto2020, Chen2019,Chen2020,Hoop2019,Hoop2019a,Uhlig2020,Feizmohammadi2019,Hintz2020, Kurylev2014,Balehowsky2020,Lai2021,Lassas2017,Tzou2021,Uhlmann2020,Uhlmann2019}.
For an overview of the recent progress, see \cite{lassas2018inverse,Uhlmann2021}.
In particular, in \cite{sa2022inverse} the inverse scattering problem for defocusing energy critical semilinear wave equations in Minkowski space is studied.
The inverse problem of recovering the metric from a source-to-solution map with different sources and receivers is studied in \cite{feizmohammadi2021inverse}.
Inverse boundary value problems of recovering the metric or nonlinearity are considered for semilinear or quasilinear wave equations in \cite{Hoop2019,Uhlmann2019, Hintz2017, Hintz2020, hintz2022dirichlet, Uhlmann2021a, acosta2022nonlinear, UZ_acoustic,zhang2023nonlinear}.

More explicitly, we consider a smooth manifold  $\tM = (-1, 1)_T \times \mathbb{R}_X^3,$ which is equipped with a globally hyperbolic Lorentzian metric $g$ given by
\[
g = -\beta(T,X) \dif T^2 + \kappa(T,X),
\]
where $\beta>0$ is smooth and $\kappa$ is a family of Riemannian metrics on $\mR^3$ depending on $T$ smoothly.
The region of interest is an open diamond set given by
\[
M = \{(T, X) \in \tM: T + |X| < 1, \ T - |X|> -1\}.
\]
With
$\pM = \bar{S}_+ \cup \bar{S}_-$,
we shall denote
\beq\label{def_Spm}
S_\pm  = \{(T,X)\in \tM: \pm T + |X| = 1, \ 0 < |X| < 1\}
\eeq
to be respectively the future or past null infinity of $M$.
We denote by $i_\pm = (\pm 1, 0)$ respectively the future and past timelike infinity
and by $R = \{(T, X): T = 0, \ |X| = 1\}$ the spacelike infinity.
In addition, we make the following assumptions on the structure of the spacetime.
\begin{assumption}\label{assump_Mg}
    Let $(M,g)$ be a globally hyperbolic subset defined as above. We assume
    \begin{itemize}
        \item the metric $g$ is smooth up to $i_\pm$, and
        \item the future and past null infinity satisfy \[
        \Spl = \partial J^-(\ipl)\cap I^+(\imi) \setminus \{\ipl\}, \quad
        \Smi = \partial J^+(\imi)\cap I^-(\ipl) \setminus \{\imi\},
        \]
        such that $i_\pm$ have no cut points there and $S_\pm$ are simple characteristic hypersurfaces with respect to the wave operator $\square_{g}$, and
        \item $(M,g)$ is nontrapping, in the sense that the projections of all null bicharacteristics tend to $S_\pm$ as their parameters tend to $\pm \infty$.
    \end{itemize}
\end{assumption}
Here the assumption that $i_\pm$ has no cut points on $S_\pm$ implies a smooth parameterization of $S_\pm$ by null geodesics, see Section \ref{subsec_familycurves}.
The assumption that $S_\pm$ are characteristic hypersurfaces allows us to choose local coordinates on null infinity where the metric and the Laplace-Beltrami operator $\sq_g$ are given by certain forms, see Lemma \ref{norm-form}.
The nontrapping assumption can also be found in the so-called asymptotically Minkowski spaces, see \cite{baskin2015asymptotics,hintz2015semilinear}.
One example of such a globally hyperbolic subset is the conformal compactification of the Minkowski metric.

We consider the semilinear wave equation
\begin{equation}\label{eq_problem}
    \begin{aligned}
        \sq_g u  + F(T,X,u) &= 0 & \  & \mbox{in } M,\\
        u &= u_- & \ &\mbox{on } \Smi,
    \end{aligned}
\end{equation}
where $\sq_g$ is the Laplace-Beltrami operator induced by the Lorentzian metric $g$, $F \in C^\infty(M \times \mR)$,
and $u_-$ is the scattering data posed on the past null infinity.
We solve the Cauchy-Goursat problem with properly chosen data $u_-$ supported on the past null infinity $S_-$ and define the scattering operator as the value of the solution of \eqref{eq_problem} restricted to a subset of the future null infinity $S_+$.
This defines the nonlinear scattering operator, see
Section \ref{sec_forward} for more details.
More precisely, for fixed $\rho \in(0,1)$, we consider consider the set
\[
\scmrho=\{ (T,X) \in \scm: -\rho < T+|X| < 1\}.
\]
By Theorem \ref{NLGPT},
the nonlinear problem (\ref{eq_problem}) has a unique solution $u \in \mch^k(\Omega_\rho)$ for $k \geq 3$, as long as the scattering data
\[
u_- \in D(\rho) := \{ u_- \in H^k(\scmrho) : \|u_-\|_{H^k} < \epsilon(\rho)\},
\]
with $\ep(\rho)$ satisfying Theorem \ref{NLGPT}.
The nonlinear scattering map is thus a map (or a collection of maps ) given by
\beq\label{def_N}
\begin{split}
    \mcn: u_- \mapsto u|_{\scprho}, \quad \text{where } \scprho =\{(T,X) \in \scp:  T-|X| < \rho\},
\end{split}
\eeq
for any $u_- \in D(\rho)$.
We consider the inverse problem of recovering $g$ from this scattering map.
Moreover, we assume the nonlinearity $F$ is analytic in $u$ and thus can be written as a power series
\beq\label{def_F}
F(T,X, u) = \sum_{m=2}^{+\infty}\beta_m(T,X) u^m, \quad \beta_m \in C^\infty(M),
\eeq
and additionally for each $q \in M$, there exists $m \geq 2$ such that $\beta_m(q) \neq 0$.
Our main theorem is the following:
\begin{theorem}\label{mainthm}
    Let $(M^{(j)}, g^{(j)}), j = 1,2$ be globally hyperbolic Lorentzian subsets satisfying Assumption \ref{assump_Mg}.
    Consider the semilinear wave equation (\ref{eq_problem}) with nonlinearity $F^{(j)}(T, X, u)$ satisfying (\ref{def_F}) and the assumption below, for $j = 1,2$.
    If the nonlinear scattering operators defined by (\ref{def_N}) satisfy
    \[
    \lN^{(1)}(u_-) = \lN^{(2)}(u_-)
    \] 
    for each $u_- \in D(\rho)$,
    then there exists a smooth diffeomorphism
    $\Psi: M \rightarrow M$ and a function $\gamma\in C^\infty(M)$ such that
    for any $q \in M$, we have
    \[
    \Psi^*\bigl(g^{(1)})=e^{2\gamma}g^{(2)}.
    \]
\end{theorem}
To better illustrate the ideas used in proving Theorem \ref{mainthm},
we first consider the reconstruction for a cubic wave equation.
\begin{theorem}\label{thm_cubic}
    Let $(M^{(j)}, g^{(j)}), j = 1,2$ be globally hyperbolic Lorentzian subsets satisfying Assumption \ref{assump_Mg}.
    Consider the cubic wave equation
    \[
    \sq_g u^{(j)}  + \beta^{(j)}(T,X) (u^{(j)})^3 = 0, \quad  j=1,2,
    \]
    where $\beta^{(j)} \in C^\infty(M^{(j)})$ are nonvanishing on  $M$.
    If the nonlinear scattering operators satisfy
    \[
    \lN^{(1)}(u_-) = \lN^{(2)}(u_-)
    \] 
    for each $u_- \in D(\rho)$,
    then there exists a smooth diffeomorphism
    $\Psi: M \rightarrow M$ and a function $\gamma\in C^\infty(M)$ such that
    for any $q \in M$, we have
    \[
    \Psi^*\bigl(g^{(1)})=e^{2\gamma}g^{(2)}.
    \]
\end{theorem}

In this work, following the approach in \cite{Kurylev2018},
we start with the multi-fold linearization
and derive the asymptotic expansion of the measurements, using solutions to linear wave equations.
Inspired by Friedlander \cite{Fri,Fri1}, we construct receding waves as special solutions to linear wave equations.
The interaction of these waves produces new singularities carrying information to the future null infinity.
Rather than using an observation set outside $M$ to recover the metric, we reconstruct ${g}$ within $M$, up to conformal diffeomorphisms, using the so-called earliest (or regular) scattering light observation sets, see Section \ref{sec_SLOS}.
Essentially, in the absence of caustics (such as conjugate points or cut points), this earliest (or regular) scattering light observation set is simply the intersection of the future light cone surface from a point in $M$ restricted to the future null infinity.
This is an analog of the reconstruction from the earliest light observation set in \cite{Kurylev2018} or the boundary light observation sets in \cite{Hintz2017}.
As caustics naturally exists in Lorentzian geometry, to deal with them, we use a layer stripping procedure for the reconstruction.
This approach allows us to recover the metric by sending sources and detecting new singularities in small pieces of $S_-$ and $S_+$ respectively.

%% file: forwardproblem.tex
\section{The forward problem and the nonlinear scattering operator}\label{sec_forward}
We define the scattering operator for the nonlinear equation
\[
(\square_g+L) u= F(T,X,u),
\]
where $L$ is a first order differential operator with smooth coefficients
and $F\in C^\infty(M \times \mR)$ satisfying $F(T,X,0)= \p_u F(T,X,0)=0$.
Recall the definition of $S_\pm$ in (\ref{def_Spm}).
%
We begin by analyzing the linear Goursat problem
\beq\label{LGP}
\begin{split}
    (\square_g+ L)u & = f(T,X), \\
    u\bigr|_{\scm} &= u_- \in H^1(\scm),
\end{split}
\eeq
which has been studied by several people including Baez, Segal and Zhou \cite{BaeSegZho}, H\"ormander \cite{HorCP} and followed by \cite{BarWaf,Nic} and several others.  They prove the following global result
\begin{theorem}\label{GP1} Let $f \in L_{\loc}^1(\mr, L^2(\mr^3))$ and let $u_-\in H^1(\scm)$ then there exists $u\in \mce $ which satisfies \eqref{LGP} on $M,$ where  $\mce$ is the space of functions with finite energy
    \[
    E(u,T)=\ha \int_{\mr^3} \bigl[ (\p_Tu)^2+ \sum g^{jk} \p_j u \p_k u+ u^2\bigr] d\nu.
    \]
    where $\nu$ is a density on $\mr^3,$ for $|T| \leq R,$ and
    \beq\label{GPE}
    \sup_{|T|\leq R} E(u,T) \leq C(R) \bigl[ ||u_0||_{H^1(\scm)}+ ||f||_{L^1(-R,R); L^2(\mr^3)}\bigr].
    \eeq
    Moreover, $u$ is unique on $J^+(\scm).$
\end{theorem}

H\"ormander works on Lorentzian manifolds $\mr\times N,$ where $N$ is compact, but   by finite speed of propagation we can use his result to prove a local theorem, which is what we need. One can also refer to Theorem 23 or \cite{BarWaf}, which does not assume $N$ to be compact.  This is not an issue for the existence of solutions to \eqref{LGP}, however to guarantee  uniqueness one needs the fact that $J^+(\scm)$ is past compact, which means that $J^-((T,X)) \cap J^+(\scm)$ is compact for all $(T,X)$ in $J^+(\scm),$  to guarantee the uniqueness of the solution on $J^+(\scm).$ They give a counterexample which shows that uniqueness is not guaranteed otherwise.

We may assume that the initial data is smoother. Suppose that  $\mcv(\scm)$ denotes the Lie algebra of $C^\infty$ vector fields that are tangent to $\scm,$ and assume that
\[
\bigl[\mcv(\scm)\bigr|_{\scm}\bigr]^k u_0 \in L^2 (\scm), \;\ 1\leq k \leq K \in \mn.
\]
In the case $k>1$ we also obtain more regularity for the solution.  We know from Proposition 4.10 of \cite{MelRit} that  the commutator of $\square_g$ and elements of $\mcv(\scm)$ satisfy
\[
[\square_g, \mcv(\scm)] \in h \square_g +  \Diff^1\cdot \mcv(\scm)+ \Diff^1,
\]
where $h\in C^\infty,$ and $\Diff^1$ denotes the $C^\infty$ first order differential operators.  We also know from  Proposition 2.4 of \cite{MelRit} that $\mcv(\scm)$ is finitely generated over $C^\infty(M)$ by $C^\infty$ vector fields. By differentiating the equation \eqref{LGP} we obtain
a system
\[
\bigl( \square_g I+ \mcl) U=  \wt F,
\]
where $U^T=(u, V_1 u, \ldots, V_k u),$  $F^T=(f, V_1 f, \ldots, V_k f),$ and $V_k$ are the generators of $\mcv(\scm),$ $I$ is the $k\times k$ identity matrix and $\mcl$ is a matrix of differential operators of order one.  Theorem \ref{GP1} applied to this system shows that $\mcv(\scm)u$ has finite energy and it satisfies estimate \eqref{GPE}.  Repeating this process for any integer $k\geq 1,$ we obtain
\beq\label{GPEk}
\begin{split}
    \sup_{|T|\leq R} & E(\mcv(\scm)^k u,T) \leq  C(R) \bigl[ ||\mcv(\scm)^k u_0||_{H^1(\scm)}+ || \mcv(\scm)^k f||_{L^1(-R,R); L^2(\mr^3)}\bigr] \leq \\
    & C_1(R) \bigl[ ||\mcv(\scm)^k u_0||_{H^1(\scm)}+ || \mcv(\scm)^k f||_{L^2([-R,R]\times \mr^3)}^2 \bigr].
\end{split}
\eeq

Since away from $\scm,$ $\mcv(\scm)$  includes all vector fields, this shows that $u \in H^{k+1}(M\setminus \scm).$ In particular $u\in L_{\loc}^\infty(M\setminus \scm),$ provided $k\geq 2.$  However we need to control the $L^\infty$ norm of $u$  in terms of the norm of the initial data up to $\scm.$    For $\rho\in (-1,1),$ let
\[
S_{-,\rho}=\{ (T,X) \in \scm: -1<-\rho < T+|X| < 1\}.
\]
One can find coordinates $(x,y)$ such that
\[
\Smi=\{x=0\},
\]
and the fact that a solution $u$ of \eqref{LGP} with $u_0$ supported on $S_{-,\rho}$  we can say that
$u=0$ for $y_j< y_{0j},$ provided $0<x<\eps.$
The $C^\infty$ vector fields tangent to $\scm$ are spanned by $\{\p_{y_j},\ j=1,2,3.\}$

This is done in the following

\begin{lemma}\label{LEM1}  Let $u(x,y)\in C^k\bigl((0,2)\times [0,1]^3\bigr),$ $k\geq 4,$ be such that
    \[
    u(x,y_1,y_2,0)= u(x,y_1,0,y_3)=u(x,0,y_2,y_3)=0,
    \]
    then for $(t,x)\in (0,1)\times (0,1)$
    \[
    |u(x,y_1,y_2,y_3)| \leq |u(1, y_1,y_2,y_3)|+ C ||\p_x\p_{y_1}\p_{y_2}\p_{y_3} u||_{L^2((0,1)^4)}.
    \]
\end{lemma}
\begin{proof}  This just follows from the identity
    \[
    u(1,y_1,y_2,y_3)-u(x,y_1,y_3,y_3)= \int_x^1 \int_0^{y_3}\int_0^{y_2} \int_0^{y_1} \p_s \p_{\mu_1}\p_{\mu_2}\p_{\mu_3} u(s,\mu_1,\mu_2,\mu_3) \ ds d\mu.
    \]
\end{proof}

\begin{theorem}\label{NLGPT}
    Let $F \in C^\infty(M \times \mR)$.
    Then for any $\rho\in (-1,1),$  there exists $\eps=\eps(\rho,G,L)$ such that if $u_- \in H^k(S_{-,\rho}),$ $k\geq 3,$  and $||u_-||_{H^k(\scm)}<\eps,$ there exists a unique $u$ which satisfies
    \beq\label{NLGPE}
    \begin{split}
        (\square_g+ L)u &=  F(T,X,u) \text{ in }  \Omega_\rho,\\
        u\bigr|_{\scm} &= u_-,
    \end{split}
    \eeq
    where we define $\Omega_\rho = \{ (T,X): T-|X| < \rho, \;\ T+|X|>-\rho \}$.
\end{theorem}
\begin{proof}  Let $u_0$ satisfy
    \beq\label{GP0}
    \begin{split}
        & (\square_g + L)u_0=0, \\
        & u_0\bigr|_{\scm}= u_- \in H^k(S_{-,\rho}),
    \end{split}
    \eeq
    It follows from \eqref{GPEk} that $\mcv(\scm)^k u_0\in \mce$ and that for fixed $\rho,$
    \beq\label{norm0}
    ||\mcv(\scm)^k u_0||_{H_1(\Omega_\rho)} \leq C_1(\rho)  ||\mcv(\scm)^k u_0||_{H^1(\scm)}\leq C_1(\rho) \eps.
    \eeq
    Also, by finite speed of propagation
    \[
    u_0=0 \text{ outside } J^+(S_{-,\rho}).
    \]

    This means that away from $\scm,$ $u_0\in H^{k+1},$ and hence bounded if $k$ is large. But according to Lemma \ref{LEM1},
    \beq\label{bdu0}
    \|u_0\|_{L^\infty} \leq C \| u_0\|_{\mch^{k+1}(\Omega_\rho)}.
    \eeq
    By proposition 2.1.2 of \cite{SaBSWE}, if $\Omega_\rho$ is an open set, the space
    \beq\label{GN}
    \begin{split}
        & u \in \mch^k(\Omega_\rho)  := L_{\loc}^\infty(\Omega_\rho)\cap \{ w: \mcv(\scm)^k w \in L^2\} \text{ is a } C^\infty \text{ algebra and moreover} \\
        & || \mcv(\scm)^k F(T,X,u)||_{L^2(\Omega_\rho)} \leq  K(||u||_{L^\infty}) ||u||_{\mch^k(\Omega)},  \text{ where } K \text{ is continuous. }
    \end{split}
    \eeq
    Let
    \[ B(0, C_1(\rho) \eps)=  \bigl\{ w\in \mch^k(\Omega_\rho), \; w=0 \text{ outside } J^+(\scmrho)  : \|w\|_{\mch^k} \leq C_1(\rho)\eps \bigr\},
    \]
    Then, according to \eqref{GPEk}, for $w\in B(0, C_1(\rho) \eps),$  the solution $v$ of
    \beq\label{MAP}
    \begin{split}
        & \bigl(\square+L\bigr) v= F(T,X,u_0+w), \\
        & v\bigr|_{\scm}= 0,
    \end{split}
    \eeq
    satisfies
    \[
    \begin{split}
        & ||v||_{\mch^{k}(\Omega_\rho)} \leq C_1(\rho) ||F(T,X,u_0+w)||_{\mch^k(\Omega_\rho)}^2 \leq \\
        & C_1(\rho) \bigl(K(||u_0+w||_{L^\infty}) \bigr) ||u_0+ w||_{\mch^k(\Omega_\rho} \bigr)^2,
    \end{split}
    \]
    but according to Lemma \ref{LEM1}, $||w||_{L^\infty} \leq C ||w||_{\mch^{k}},$ and so we find that
    \[
    ||v||_{\mch^{k}(\Omega_\rho)} \leq  C_1(\rho) \bigl(K(C||u_0+w||_{\mch^k} ||u_0+ w||_{\mch^k(\Omega_\rho} \bigr)^2\leq 4 C_1(\rho)^3 K(2C\eps)^2 \eps^2.
    \]
    If we pick $\eps$ such that $4C_1(\rho)^2K(2C\eps)^2 \eps<1,$ this gives a bounded nonlinear map
    \[
    \begin{split}
        B(0, C_1(\rho) & \eps) \longrightarrow B(0, C_1(\rho) \eps), \\
        & w \longmapsto v
    \end{split}
    \]

    Next we show that by further restricting $\eps$ this map is a contraction. For $w_1, w_2\in B(0, C_1(\rho) \eps),$  let $\mcf(T,X,s,t)\in C^\infty$ such that
    \[
    F(T,X,u_0+w_1)- F(T,X,u_0+w_2)= \mcf(T,X,u_0,w_1,w_2)(w_1-w_2)^2.
    \]
    We appeal to the Gagliardo-Nirenberg estimates proved in Lemma 5.2 of \cite{MelRit} which states that if $V_1,V_2, \ldots V_p\in \mcv(\scm),$ then for $0<|\alpha|< k,$
    \beq\label{GNE}
    \bigl|\bigl| \bigl(V_1, V_2, \ldots V_p\bigr)^\alpha u\bigr|\bigr|_{L^{\frac{2k}{|\alpha|}}} \leq C\bigl( \|u\|_{L^\infty} + \sum_{|\beta|\leq k} \bigl\|(V_1, \ldots V_p\bigr)^\beta \bigr\|_{L^2}\bigr),
    \eeq
    where $C$ only depends on the vector fields and the domain.
    A combination of the product rule and  H\"older's inequality gives
    \[
    \begin{split}
        & \bigl\| \bigl( V_1 \ldots V_p\bigr)^\alpha \bigl(\mcf(T,X,u_0,w_1,w_2) (w_1-w_2)\bigr)\bigr\|_{L^2} \leq  \\
        & \sum_{\ga\leq \alpha } \bigl\|(V_1\ldots V_p\bigr)^\ga \mcf(T,X,u_0,w_1,w_2)\bigr\|_{L^{\frac{2|\ga|}{|\alpha|}}} \bigl\|(V_1\ldots V_p\bigr)^{\alpha-\ga} (w_1-w_2)^2\bigr\|_{L^{\frac{2|\alpha-\ga|}{|\alpha|}}},
    \end{split}
    \]
    and  \eqref{GNE} gives that

    \[
    \begin{split}
        & \bigl\| \bigl( V_1 \ldots V_p\bigr)^\alpha \bigl(\mcf(T,X,u_0,w_1,w_2) (w_1-w_2)^2\bigr)\bigr\|_{L^2} \leq \\
        & C\biggl( \|\mcf\|_{L^\infty} + \sum_{|\beta|\leq |\alpha|} \bigl\|(V_1, \ldots V_p\bigr)^\beta \mcf \bigr\|_{L^2}\biggr)
        \biggl( \|w_1-w_2\|_{L^\infty}^2 + \sum_{|\beta|\leq |\alpha|} \bigl\|(V_1, \ldots V_p\bigr)^\beta (w_1-w_2)^2 \bigr\|_{L^2}\biggr)
    \end{split}
    \]
    In view of Lemma \ref{LEM1}, since $u_0=w_1=w_2=0$ outside $J^+(\scmrho),$ we obtain
    \[
    ||w_j||_{L^\infty} \leq C ||w_j||_{\mch^k}, j=1,2, \text{ and } ||u_0||_{L^\infty} \leq C ||u_0||_{\mch^k},  \text{ provided } k\geq 2.
    \]
    We also know from \eqref{GN} that
    \[
    \begin{split}
        & \bigl\|(V_1, \ldots V_p\bigr)^\beta \mcf(T,X,u_0,w_1,w_2) \bigr\|_{L^2}\leq \\
        & K\bigl(|\|u_0\|_{L^{\infty}},|\|w_1\|_{L^{\infty}},|\|w_2\|_{L^{\infty}}\bigr)\bigl(|\|u_0\|_{\mch^k}+|\|w_1\|_{\mch^k}+|\|w_2\|_{\mch^k}).
    \end{split}
    \]

    This implies that
    \beq\label{aux3}
    \|f(T,X,u_0+w_1)-f(T,X,u_0+w_2)\|_{\mch^k} \leq K(\rho,\eps) \eps^2,
    \eeq
    and therefore if $v_1$ and $v_2$ are the solutions of \eqref{MAP} corresponding to $w_1$ and $w_2,$ it follows from \eqref{GPEk} and \eqref{aux3} that
    \[
    ||v_1-v_2||_{\mch^k} \leq C_1(\rho) (K(\rho,\eps) \eps^2)^2.
    \]
    We now just pick $\eps$ such that $C_1(\rho) (K(\rho,\eps) \eps^2)^2<1.$ This gives a contraction mapping and therefore we obtain a function $w\in \mch^k(\Omega_\rho)$ such that $u=w+u_0$ satisfies \eqref{NLGPE}.
\end{proof}
Then the scattering map is defined on the set
\[
D = \bigcup_{\rho \in (-1,1)} D(\rho),  \quad \text{ with }  D(\rho) := \{ u_- \in H^k(\scmrho) : \|u_-\|_{H^k} < \epsilon(\rho) \} \text{ for } k\geq 3,
\]
where $\epsilon(\rho)>0$ is chosen so that  Theorem \ref{NLGPT} holds and the nonlinear wave equation, with data $u_-$, admits a unique solution which is defined in the region $t-|X| < \rho$. The scattering map is thus a map (or collection of maps)
\beq\label{def-scat}
\begin{split}
    D(\rho) \ni u_- \mapsto u\bigr|_{S_{+,\rho}},
\end{split}
\eeq
where $u$  is the solution of  (\eqref{NLGPE}) and  $S_{+,\rho} =\{(T,X) \in \scp:  T-|X| < \rho\}$.
Observe that $u\in C^k$ and there is no problem defining the restriction of $u$ to $\scp.$

%% file: preliminary.tex
\section{Preliminaries}\label{sec_prelim}
\subsection{Lorentzian manifolds}

We recall some notations and preliminaries in \cite{Kurylev2018}.
Let $(\tM, g)$ be a globally hyperbolic Lorentzian manifold.
For $\eta \in T_p^*\tM$, the corresponding vector of $\eta$  is denoted by $ \eta^\# \in T_p \tM$.
The corresponding covector of a vector $\xi \in T_p \tM$ is denoted by $ \xi^\flat \in T^*_p \tM$.
We denote by
\[
L_p \tM = \{\zeta \in T_p \tM \setminus 0: \  g(\zeta, \zeta) = 0\}
\]
the set of light-like vectors at $p \in \tM$ and similarly by $L^*_p \tM$ the set of light-like covectors.
The sets of future (or past) light-like vectors are denoted by $L^+_p \tM$ (or $L^-_p \tM$), and those of future (or past) light-like covectors are denoted by $L^{*,+}_p \tM$ (or $L^{*,-}_p \tM$).

The characteristic set $\Char(\sq_g)$ is the set $b^{-1}(0) \subset T^*\tM$, where
$
b(x, \zeta) =g^{ij}\zeta_i \zeta_j
$
is the principal symbol.
It is also the set of light-like covectors with respect to $g$.
We denote by $\Theta_{x, \zeta}$ the null bicharacteristic of $\sq_g$ that contains $(x, \zeta) \in L^*\tM$,
which is defined as the integral curve of the Hamiltonian vector field $H_b$. 
Then a covector $(y, \eta) \in \Theta_{x, \zeta}$ if and only if there is a light-like geodesic $\gamma_{x, \zeta^\#}$ such that
\[
(y, \eta) = (\gamma_{x, \zeta^\#}(s), (\dot{\gamma}_{x, \zeta^\#}(s))^\flat), \ \text{ for } s \in \mathbb{R}.
\]
Here we denote by $\gamma_{x, \zeta^\#}$ the unique null geodesic starting from $x$ in the direction $\zeta^\#$.

The time separation function $\tau(x,y) \in [0, \infty)$ between two points $x < y$ in  $\tM$
is the supremum of the lengths \[
L(\alpha) =  \int_0^1 \sqrt{-g(\dot{\alpha}(s), \dot{\alpha}(s))} ds
\] of
the piecewise smooth causal paths $\alpha: [0,1] \rightarrow \tM$ from $x$ to $y$.
If $x<y$ is not true, we define $\tau(x,y) = 0$.
Note that $\tau(x,y)$ satisfies the reverse triangle inequality
\[
\tau(x,y) +\tau(y,z) \leq \tau(x,z), \text{ where } x \leq y \leq z.
\]
For $(x,v) \in L^+\tM$, recall the cut locus function
\[
\rho(x,v) = \sup \{ s\in [0, \mathcal{T}(x,v)]:\ \tau(x, \gamma_{x,v}(s)) = 0 \},
\]
where $\mathcal{T}(x,v)$ is the maximal time such that $\gamma_{x,v}(s)$ is defined.
The cut locus function for past lightlike vector $(x,w) \in L^-\tM$ is defined dually with opposite time orientation, i.e.,
\[
\rho(x,w) = \inf \{ s\in [\mathcal{T}(x,w),0]:\ \tau( \gamma_{x,w}(s), x) = 0 \}.
\]
For convenience, we abuse the notation $\rho(x, \zeta)$ to denote $\rho(x, \zeta^\#)$ if $\zeta \in L^{*,\pm}\tM$.
By \cite[Theorem 9.15]{Beem2017}, the first cut point $\gamma_{x,v}(\rho(x,v))$ is either the first conjugate point or the first point on $\gamma_{x,v}$ where there is another different geodesic segment connecting $x$ and  $\gamma_{x,v}(\rho(x,v))$.

\subsection{Lagrangian distributions}
Suppose $\Lambda$ is a conic Lagrangian submanifold in $T^*\tM$ away from the zero section.
We denote by $\Ical^\mu(\Lambda)$ the set of Lagrangian distributions in $\tM$ associated with $\Lambda$ of order $\mu$.
In local coordinates, a Lagrangian distribution can be written as an oscillatory integral and we regard its principal symbol,
which is invariantly defined on $\Lambda$ with values in the half density bundle tensored with the Maslov bundle, as a function in the cotangent bundle.
If $\Lambda$ is a conormal bundle of a submanifold $K$ of $\tM$, i.e. $\Lambda = N^*K$, then such distributions are also called conormal distributions.
The space of distributions in $\tM$ associated with two cleanly intersecting conic Lagrangian manifolds $\Lambda_0, \Lambda_1 \subset T^*\tM \setminus 0$ is denoted by $\Ical^{p,l}(\Lambda_0, \Lambda_1)$.
If $u \in \Ical^{p,l}(\Lambda_0, \Lambda_1)$, then one has $\wfset{(u)} \subset \Lambda_0 \cup \Lambda_1$ and
\[
u \in \Ical^{p+l}(\Lambda_0 \setminus \Lambda_1), \quad  u \in \Ical^{p}(\Lambda_1 \setminus \Lambda_0)
\]
away from their intersection $\Lambda_0 \cap \Lambda_1$. The principal symbol of $u$ on $\Lambda_0$  and  $\Lambda_1$ can be defined accordingly and they satisfy some compatible conditions on the intersection.

For more detailed introduction to Lagrangian distributions and paired Lagrangian distributions, see \cite[Section 3.2]{Kurylev2018} and \cite[Section 2.2]{Lassas2018}.
The main reference are \cite{MR2304165, Hoermander2009} for conormal and Lagrangian distributions and
{\rv
    \cite{Hoop2015,Greenleaf1990,Greenleaf1993,Melrose1979,Guillemin1981}
    for paired Lagrangian distributions.
}

{
    \subsection{Inverses of linear wave equations}\label{subsec_causalinverse}
    On a globally hyperbolic Lorentzian manifold $(\tM,g)$, the wave operator $\sq_g$ with the principal symbol $b(x, \zeta) =g^{ij}\zeta_i \zeta_j$ is normally hyperbolic, see \cite[Section 1.5]{Baer2007}.
    It has a unique casual inverse $\sq_g^{-1}$ according to \cite[Theorem 3.3.1]{Baer2007}.
    By \cite[Proposition 6.6]{Melrose1979}, one can symbolically construct a parametrix $Q_g$, which is the solution operator to the wave equation
    \[
    \begin{split}
        \sq_g v &= f, \quad \text{ on } \tM,\\
        v & = 0, \quad \text{ on } \tM \setminus J^+(\supp(f)),
    \end{split}
    \]
    in the microlocal sense.
    It follows that $Q_g \equiv \sq^{-1}_g$ up to a smoothing operator.
    We denote the kernel of $Q_g$ by $q(x, \tilde{x})$ and
    it is a paired Lagrangian distribution
    in $\Ical^{-\frac{3}{2}, -\frac{1}{2}} (N^*\text{Diag}, \Lambda_g)$,
    where $\text{Diag}$ denotes the diagonal in $M \times M$, $N^*\text{Diag}$ is its conormal bundle, and $\Lambda_g$ is the flow-out of
    $N^*\text{Diag} \cap \Char(\sq_g)$ under the Hamiltonian vector field $H_b$.
    Here we construct the microlocal solution to the equation
    \[
    \sq_g q(x, \tilde{x}) = \delta(x, \tilde{x}) \mod C^\infty(M \times M),
    \]
    using the proof of \cite[Proposition 6.6]{Melrose1979}.
    The symbol of $Q_g$ can be found during the construction there.
    In particular, the principal symbol of $Q_g$ along $N^*\text{Diag}$  satisfying
    $
    \sigma_p(\delta) = \sigma_p(\sq) \sigma_p(Q_g)
    $
    is nonvanishing.
    The one along $\Lambda_g \setminus N^*\text{Diag}$ solves the transport equation
    \[
    \mathcal{L}_{H_b}\sigma_p(Q_g) + i c\sigma_p(Q_g) = 0,
    \]
    where $\mathcal{L}_{H_b}$ is the Lie action of the Hamiltonian vector field $H_b$ and $c$ is the subprincipal symbol of $\sq$.  
    The initial condition 
    is given by restricting $\sigma_p(Q_g)|_{N^*\text{Diag}}$ to $\partial \Lambda_g$,
    see (6.7)  and Section 4 in \cite{Melrose1979}.
    Then one can solve the transport equation by integrating along the bicharacteristics. 
    This implies 
    the solution to the transport equation is nonzero and therefore 
    $\sigma_p(Q_g)|_{\Lambda_g}$ is nonvanishing.
    See also \cite{Hoop2015, Baer2007, Greenleaf1993} for more references.

    We have the following proposition according to \cite[Proposition 2.1]{Greenleaf1993}, see also \cite[Proposition 2.1]{Lassas2018}.
    \begin{prop}
        Let $\Lambda$ be a conic Lagrangian submanifold in $T^*M \setminus 0$.
        Suppose $\Lambda$ intersects $\Char(\sq_g)$ transversally, such that its intersection with each bicharacteristics has finite many times.
        Then
        \[
        Q_g: \Ical^\mu(\Lambda) \rightarrow \Ical^{\mu- \frac{3}{2},-\frac{1}{2}}(\Lambda, \Lambda^g),
        \]
        where $ \Lambda^g$ is the flow-out of $\Lambda \cap \Char(\sq_g)$ under the Hamilton flow.
        Moreover, for $(x, \xi) \in \Lambda^g \setminus \Lambda$, we have
        \[
        \sigma_p(Q_g u)(x, \xi) = \sum \sigma_p(Q_g)(x, \xi, y_j, \eta_j)\sigma_p(u)(y_j, \eta_j),
        \]
        where the summation is over the points $(y_j, \eta_j) \in \Lambda$ that lie on the bicharacteristics from $(x, \xi)$.
    \end{prop}
}
Later we would like to consider the solution to
\beq\label{def_Qs}
\sq_g v = f \text{ in } M, \quad \text{ with }\lR_-[v] = 0,
\eeq
where $\lR_-[v]$ is the restriction of $v$ to the smooth null hypersurface $S_-$.
For convenience, we define the solution operator $Q_s$ in the sense that $v = Q_s(f)$ solves the linear problem above.
In particular, for $f \in \lE'(M)$, this solution operator coincides with the causal inverse $Q_g$ up to smoothing operators.
\subsection{Receding waves with conormal singularities}\label{sec_receding}
Our purpose is to generate solutions of the Goursat problem \eqref{GP0} with scattering data $u_0$ supported in $S_{-,\rho}$.
In addition, these solutions are supposed to have conormal singularities along a $C^\infty$ hypersurface transversal to $S_-$, in a small neighborhood of a given point $p\in S_{-,\rho}$.
We name these receding waves, following Friedlander \cite{Fri,Fri1}. This is a local result and we  work in suitable local coordinates given by the following lemma.
\begin{lemma}\label{norm-form} One can choose local coordinates $(s,x)$,
with $x=( x_1, x_2,x_3)$, valid near $p\in S_-$ such that, modulo lower order terms, we have $\Smi = \{x_1 = 0\}$ and
\beq\label{NF1}
    \begin{split}
        \square_{g}\equiv a(x,s) x_1 \p_{x_1}^2+  b_0(x,s)\p_{x_1} \p_s+&  \sum_{j=2,3} (\alpha_{j} \p_s + x_1 \beta_j\p_{x_1})\p_{x_j}+
        \sum_{j,k=2,3}^3 \alpha_{jk} \p_{x_j} \p_{x_k},
    \end{split}
    \eeq
    with $b_0(x,s)\not=0$.

\end{lemma}
\begin{proof} We may choose local coordinates $y=(y_1,y_2,y_3,y_4)$ near $p \in S_-$ such that
    \[
    S_-=\{y_1=0\} \text{ and } p=(0,0,0,0).
    \]
    Since $S_-$ is characteristic for $\square_g$, the principal  symbol $p(y,\eta)=\sigma_2(\square_g)$ satisfies
    \[
    p(y,\eta)= a_{11}(y) y_1 \eta_1^2 + \sum_{j=2}^4 a_{1j}(y) \eta_1 \eta_j+ \sum_{j,k=2}^4 a_{jk}(y) \eta_j\eta_k.
    \]
    Moreover, since  $S_-$ is simply characteristic, $dp\not=0$ on $N^*S_-\setminus 0,$ and hence
    \[
    a_{1j}(0)\not=0,  \;\ j=2,3,4.
    \]

    Now we  pick local coordinates $(s,x_2,x_3)$ defined in a neighborhood of $p$, but inside the surface $S_-,$ such that
    \[
    \begin{split}
        \p_s &= a_{12}(0,y')\p_{y_2}+ a_{13}(0,y')\p_{y_3}+ a_{14}(0,y')\p_{y_4}, \\
         \text{ and } \p_s x_j &= 0, \text{ for } j=2,3, \qquad  \p_{x_2}x_3=\p_{x_3}x_2=0.
    \end{split}
    \]
    Now we extend these coordinates to a neighborhood of $q$ in $\tM$ such that
    \[
    \p_{y_1} s=0, \quad \p_{y_1} x_2=0, \quad \p_{y_1} x_3=0.
    \]
    It follows that in coordinates  $(s,x_1,x_2,x_3)$, where $x_1=y_1$, the principal part of $\square_{g}$ satisfies
    \[
    \begin{split}
        \square{g}=  a_{11} x_1\p_{x_1}^2+ & \bigl(\sum_{j=1}^3 a_{1j} \frac{\p s}{\p y_j}\bigr) \p_s \p_{x_1} + \sum_{j=2}^4 \biggl( a_{12} \frac{\p x_j}{\p y_2}+ a_{13} \frac{\p x_j}{\p y_3}+ a_{14} \frac{\p x_j}{\p y_4} \biggr)\p_{x_j} \p_{x_1}+\\
        & \quad \quad \quad \quad \quad \quad \quad\quad \quad \quad\sum_{j,k=2}^4 a_{jk}\bigl( \frac{\p s}{\p y_j}\p_s+  \frac{\p x}{\p y_j}\cdot \p_x\bigr) \bigl( \frac{\p s}{\p y_k}\p_s+  \frac{\p x}{\p y_k}\cdot \p_x\bigr),
    \end{split}
    \]
    where
    \[
    \frac{\p x}{\p y_k}\cdot \p_x= \frac{\p x_1}{\p y_k} \p_{x_1}+ \frac{\p x_2}{\p y_k} \p_{x_2}+\frac{\p x_3}{\p y_k} \p_{x_3}.
    \]
    Since $\p_s x_j=0$ on $\{x_1=0\}$, this proves \eqref{NF1}.  The fact $b_0\not=0$ is due to the fact that $S_-$ is simply characteristic.
\end{proof}

Given a point $p \in \Smi$ and let $\tU \subseteq \tM$ be a neighborhood of $p$ in which coordinates \eqref{NF1} are valid.
Let $\phi(x_2,x_3)\in C^\infty(\{x_1=0\})$  and let
\[
\Sigma=\{(s,x_2,x_3)\in U: s-\phi(x_2,x_3)=0\}
\]
be a $C^\infty$ surface in
$U = \tU \cap \{x_1 = 0\}$.
Its conormal bundle is parameterized by the phase function $\Phi(s,x_2,x_2,\alpha)=\alpha (s-\phi(x_2,x_3))$ with $\alpha \in \mr \setminus 0,$  in the sense that
\[
\begin{split}
N^*\Sigma =\{ (s,x_2, x_3, \sigma, \xi_2, \xi_3): &\ \p_\alpha \Phi=s-\phi(x_2,x_3)=0, \;\ \sigma=\p_s \Phi=\alpha,\\
&\quad \quad \quad
\xi_j=\p_{x_j}\Phi=-\alpha \p_{x_j} \phi(x_2,x_3), \text{ for } j=2,3\}.
\end{split}
\]
Since here $N=1$, let
\[
\begin{split}
    u(s,x_2,x_3)= \frac{1}{(2\pi)^n} \int_{\mr} & e^{i \alpha(s-\phi(x_2,x_3)} a(s,x_2,x_3,\alpha) \ d\alpha, \quad \text{where  $a\in S^{\frac{m+2n-1}{4}} (U \times (\mr_{\alpha}\setminus 0))$},
\end{split}
\]
with
\[
a(s,x_2,x_3,\alpha)\sim \sum_{j=j_0}^\infty a_j(s,x_2,x_3) \alpha^{j_0-j}.
\]
According to \cite[Chapter 25]{Hor4}, this is a conormal distribution on $\Sigma.$
Now we extend $N^*\Sigma$ to a submanifold of $\mcv \subseteq T_{\{x=0\}}^* \tM,$ which is characteristic for the operator $\square_{g},$ or in other words, that $ \mcv \subseteq p^{-1}(0),$ where $p= \sigma_2(\square_{g}).$ But it follows from \eqref{NF1} that at $\{x_1=0\},$ we have
\[
p= b_0\sigma\xi_1+ h(s,x_2,x_3,\sigma,\xi_2,\xi_3),
\]
where $h$ is a homogeneous polynomial of degree two in $(\sigma,\xi_2,\xi_3)$, with smooth coefficients depending on $(s,x_2,x_3).$
Thus, we set
\[
\begin{split}
     \mcv =\bigl\{ (s,0, x_2,x_3, &\sigma,\xi_1, \xi_2,\xi_3): \ \sigma= \alpha, \;\ \xi_j= -\p_{x_j}\phi(x_2,x_3)\alpha, \text{ for }j=2,3, \\
    & \xi_1= - \frac{1}{b_0(0,x_2,x_3,s) \sigma}h(s, x_2,x_3,\sigma,\xi_2,\xi_3)= -\frac{\alpha}{b_0(0,x_2,x_3,s)} h(s,x_2,x_3,1, \p_{x_2} \phi, \p_{x_3} \phi)\}.
\end{split}
\]
Since $H_p$ is transversal to $\{x_1=0\}$ when $\sigma\not=0,$ we define
\[
\mfU= \exp(-\mu H_{\mfp}) \mcv,
\]
to be the manifold obtained by flowing-out of $\mcv$ to $M,$ where the negative sign indicates that the flow goes towards the region $\{x_1>0\}=M$  is a $C^\infty$ conic Lagrangian submanifold on the entire $T^*\tM\setminus 0.$

Since by construction, the projection
\[
\begin{split}
    & \Pi: \mfU \longrightarrow \tM, \\
    & (s,x, \sigma,\xi) \longmapsto (s,x)
\end{split}
\]
has maximal rank at $p$, then it must have maximal rank near $a$ and so $\mfU$ is the conormal bundle of a $C^\infty$ surface near
$\{x_1=0\}.$  Therefore there exists a $C^\infty$ surface
\[
\wt  \Sigma\subset \tM, \text{ such that } \wt \Sigma\cap \{x_1=0\}= \Sigma= \{\phi(s,x_2,x_3)=0\}
\]
and so there exists a $C^\infty$ function $\psi(s,x_1,x_3,x_3)$ in a neighborhood $\wt U\subset \tM$ of $p$ such that
$\psi(s,0,x_2,x_3)= \phi(s,x_2,x_3)$ and $\p_s \psi\not=0  \text{ in } \wt U,$  and therefore
\[
\mfU=N^*\{(s,x) \in \wt U:  \psi(s,x)=0, \;\ \p_s \psi\not=0 \text{ in } \wt U\}.
\]
The phase function $\psi(s,x)\alpha$ with $\alpha\in \mr\setminus 0$ parameterizes  $N^*\wt \Sigma.$
Next we want to find $v(s,x)$ such that $Pv=0$ and
\[
\begin{split}
    v(s,0,x_2,x_3)= u_0(s,x_2,x_3)=  \frac{1}{(2\pi)^n} \int_{\mr} e^{i\alpha(s- \phi(x_2,x_3))} a(s,x_2,x_3,\alpha) \ d\alpha,
\end{split}
\]
where we write $P=\square_{g}+L$ and $a\in S^{\frac{m+2n-1}{4}} (U \times \mr_{\alpha})$.
We shall denote $p=\sigma_2(P).$
We take
\[
\begin{split}
    v_1(s,x)= \frac{1}{(2\pi)^n} \int_{\mr} & e^{i\alpha \psi(s,x) } \beta(s,x,\alpha) \ d\alpha,
\end{split}
\]
where $\beta \in S^{\frac{m+2n-1}{4}} (\wt U \times( \mr_{\alpha}\setminus 0))$ such that $\beta(s,x,\theta,\alpha)\sim \sum_{j=j_0}^\infty \beta_j(s,x,\theta) \alpha^{j_0-j}.$
A standard computation gives
\[
P v_1=  \frac{1}{(2\pi)^n} \int_{\mr} e^{i\alpha\psi(s, x) } \bigl( - p(\nabla \psi)\alpha^2 \beta +  P\beta + \alpha ((P\psi) \beta+ H_{p}\bigr|_{\mfU} \beta) \bigr) \ d\alpha,
\]
where $H_{p}\bigr|_{\mfU}$ is the restriction of the Hamilton vector field $H_p$ to $\mfU.$   We know that $\mfU$ is characteristic, so $ p(\nabla \psi)=0.$
We obtain the following transport equations for $\beta_j$ in terms of $\beta_k$, for $k<j$.
Since $\p_s\psi(s,0,x_2,x_3)=\p_s \phi(s,x_2,x_3)\not=0$ near $p$, then $\p_s\psi(s,x)\not=0$ near $\{x_1=0\}$.
In particular, the vector filed $H_{p}\bigr|_{\mfU}$ is transversal to $\{x_1=0\}$ and these equations can be solved in a neighborhood of $\{x_1=0\}.$  This construction gives a function $v_1(s,x)$ such that
\[
\begin{split}
    Pv_1 =f \in C^\infty(M), \quad \text{ with }
    v(s,0,x_2,x_3) = u(s,x_2,x_3).
\end{split}
\]
We know from \eqref{GPEk} that the solution $w$ of
\[
\begin{split}
    P w = f \quad \text{in } M,\quad \text{ with } w =0 \text{ on } \Smi,
\end{split}
\]
is a $C^\infty$ function and therefore so $v= v_1-w$ is the desired solution.

Moreover, the proof of Lemma \ref{norm-form} and the analysis above indicates the following proposition.
\begin{prop}\label{pp_RFFIO}
Let $\Lambda$ be a conic Lagrangian submanifold in $T^*M \setminus 0$.
Suppose $\Lambda$ intersects $\Char(\sq_g)^\pm$ transversally finite many times.
Let $\Lambda^{g,\pm}$ be the flow-out of $\Lambda \cap \Char(\sq_g)^\pm$ under the Hamiltonian flow.
Then for $\eps>0$ small enough, $\La^{g,\pm}$ extend to $C^\infty$ Lagrangian submanifolds to $T^*\tM$.
Moreover, it intersects $S_\pm$ cleanly with
\[
 \Lambda_\infty^\pm= \La^{g,\pm} \cap S_\pm
\]
as $C^\infty$ Lagrangian submanifolds of $T^*S_\pm \setminus 0$.
In addition, if $u \in I^\mu(\tM, \Lambda^{g,\bullet})$ is a Lagrangian distribution, then its restriction
\[
    u|_{S_\bullet} \in I^{\mu+ 1/4}(S_\bullet, \Lambda_\infty^\bullet)
    \] and the symbol is given by
    $\sigma(u)$ restricted to $\Lambda_\infty^\bullet$, where we write $\bullet = \pm$.
    %
\end{prop}

%% file: constructwaves.tex
\section{Interaction of nonlinear waves}\label{subsec_waves}
\subsection{Asymptotic expansions}
In the following, let $\ep_j > 0$ be small parameters and
let $\Upsilon_j \in D(\rho)$ for $j = 1,2, 3$.
Let $u$ solve the nonlinear problem
\[
\sq_g u  + \beta u^3 = 0 \quad \text{ in }M, \quad \text{ with } \lR_-[u] = \ep_1\Upsilon_1 + \ep_2\Upsilon_2+ \ep_3\Upsilon_3,
\]
where $\lR_-[u]$ is the restriction of $u$ to $\Smi$.
We consider the nonlinear scattering map
\[
\mcn(\ep_1\Upsilon_1 + \ep_2\Upsilon_2+ \ep_3\Upsilon_3) = \lR_+[u],
\]
where $\lR_+[u]$ is the restriction of $u$ to $\Spl$.
We derive the asymptotic expansion of $u$ with respect to these small parameters.
Indeed, if we write
\[
u = \sum_j \ep_j v_j + \sum_{i,j} \ep_i \ep_j A_2^{ij}
+ \sum_{i,j, k} \ep_i\ep_j\ep_k  A_3^{ijk}
+ R_3
\]
where $v_j$ are receding waves constructed before solving
\beq\begin{split}\label{eq_vj}
    \square_g v_j = 0 \text{ in } M, \quad \text{ with }
    \lR_- [v_j]  = \Upsilon_j,
\end{split}\eeq
and $R_3$ is the remainder containing $\ep$-terms higher than $|\ep|^3$.
By plugging the formula of $u$ into the nonlinear equation and equating $\ep$-terms,
for $1 \leq i,j,k \leq 3$, we have
\[\begin{split}
    \square_g A_2^{ij}  = 0, &\quad \text{ with }
    \mathcal{R}_-[A_2^{ij}] = 0,\\
    \square_g A_3^{ijk} = -\beta v_iv_jv_k, &\quad \text{ with }
    \mathcal{R}_-[A_3^{ijk}] = 0.
\end{split}\]
Note that $A_2^{ij} \equiv 0$ and therefore
the remainder $R_3$ solves
\[\begin{split}
    \square_g R_3  = -\beta(v+A_3)^3 - 3\beta(v + A_3)^2R_3 - 3\beta(v + A_3)R_3^2 - \beta R_3^3,\qquad
    \mathcal{R}_-[R_3] = 0,
\end{split}\]
where we write $v = \sum_j \ep_j v_j$ and $A_3 =  \sum_{i,j, k} \ep_i\ep_j\ep_k A_3^{ijk}$ for simplification.
One can verify that $R$ is relatively smaller by energy estimates derived in Section \ref{sec_forward} to conclude that
\beq\label{eq_asymptotic}
\begin{split}
    \mcn(\ep_1\Upsilon_1 + \ep_2\Upsilon_2+ \ep_3\Upsilon_3)
    = \sum_{1 \leq i,j,k \leq 3} \ep_i\ep_j \ep_k \lR_+[A_3^{ijk}] + O_{L^2(\mathbb{R}\times \mathbb{S}^2)}(|\vec{\ep}|^4).
\end{split}
\eeq
In particular, we have
\[
\partial_{\ep_1}\partial_{\ep_2}\partial_{\ep_3} \lN(\ep_1\Upsilon_1 + \ep_2\Upsilon_2+ \ep_3\Upsilon_3)|_{\ep_1 = \ep_2 = \ep_3 = 0}
= \lR_+[\lU_3],
\]
where we write
\[
\lU_3  = \sum_{(i,j,k) \in \Sigma(3)} A_3^{ijk}
\]
solving linear problem
\[
\square_g \lU_3 -  6\beta v_{1}v_{2}v_{3} = 0 \text{ in } M, \quad \text{ with }
\mathcal{R}_-[\lU_3] = 0.
\]
\subsection{Propagation of receding waves}\label{subsec_constructwaves}
Recall the construction of receding waves in Section \ref{sec_receding}.
For fixed $p_0 \in \Smi$, there exists
a small neighborhood $\tU \subseteq \tM$ of $p_0$ such that
locally
\[
\tU \cap S_- = \{(s,x_1, x_2, x_3): x_1 = 0\}.
\]
We write $U = \tU \cap S_-$.
Let $w_0 \in L^+_{p_0}M \cap T^+_{p_0} M$, where $T^+_{p_0} M_-$ is the outward vector space defined in (\ref{def_outwardTM}).
The analysis in Section \ref{sec_receding} shows we can find $\varphi \in C^\infty(U)$ such that
\[
\partial_s \varphi \neq 0,
\quad \varphi(p_0) = 0,
\quad
\alpha(1, -\partial_{x_2} \varphi(p_0), -\partial_{x_3} \varphi(p_0)) = w_0,
\quad \text{for some }\alpha \neq 0.
\]
One example is given by $\varphi(s, x_2, x_3) = s - \phi(x_2, x_3)$ for some $\phi \in C^\infty(U)$ in Section \ref{sec_receding}.

Now for a small parameter $\kappa_0>0$, we define
\[
\begin{split}
    \Sigma(p_0, w_0, \kappa_0) =\{p \in U: \varphi(p) = 0 \text{ with } d_{g^+}(p, p_0) < \kappa_0 \}.
\end{split}
\]
The condition $\partial_s \varphi \neq 0$ guarantees that each covector in $N^*\Sigma(p_0, w_0, \kappa_0)$ can be uniquely mapped to an outward future pointing lightlike vector.
Thus, we define
\[
\begin{split}
    {W}({p_0, w_0, \kappa_0}) &= \{(p, w) \in L^+_{\Smi}M \cap T^+_{\Smi} M: (p, \pi(w)^\flat)\in N^*\Sigma(p_0, w_0, \kappa_0)\}.
\end{split}
\]
We denote by $\gamma_{p_0, w_0}(\mR)$ the unique null geodesic starting from $p_0$ with direction $w_0$, and we define
\[
K({p_0, w_0, \kappa_0}) = \{\gamma_{p, w}(\varsigma) \in M: (p, w)\in {W}({p_0, w_0, s_0}), \ \varsigma\in (0, \infty) \}
\]
as the subset of the light cone emanating near $(p_0, w_0)$ by light-like vectors in ${W}({p_0, w_0, \kappa_0})$.
As $\kappa$ goes to zero, the surface $K({p_0, w_0, \kappa_0})$ tends to the geodesic $\gamma_{p_0, w_0}(\mathbb{R}_+)$.
We define
\[
\begin{split}
    \Lambda({p_0, w_0, \kappa_0})
    = &\{(\gamma_{p,w}(\varsigma), r\dot{\gamma}_{p,w}(\varsigma)^\flat )\in T^*M: \\
    & \quad \quad \quad \quad \quad
    (p,w) \in {W}({p_0, w_0, \kappa_0}),\  s\in (0, \infty),\ r \in \mathbb{R}\setminus \{0 \} \}
\end{split}
\]
as the flow-out from $\Char(\sq_g) \cap \Sigma({y_0, w_0, \kappa_0})$ by the Hamiltonian vector field of $\sq_g$ in the future direction.
Note that $\Lambda({p_0, w_0, \kappa_0})$ is the conormal bundle of $K({p_0, w_0, \kappa_0})$ near $\gamma_{p_0, w_0}(\mathbb{R}_+)$ before the first cut point of $p_0$.

Let $(p_j, w_j) \in L^+_{\Smi}M \cap T^+_{\Smi} M$ for $j = 1,2,3$
and we construct $\Upsilon(p_j, w_j, \kappa_0)$ using Section \ref{sec_receding} as conormal distributions supported in $N^*\Sigma(p_j, w_j, \kappa_0)$.
Then $v_j$ are constructed as receding waves using Section \ref{sec_receding} to satisfy (\ref{eq_vj}).
Note such $v_j$ are Lagrangian distributions microlocally supported in $\Lambda({p_j, w_j, \kappa_0})$ and they are conormal distributions before the first cut point of $ \gamma_{p_j, w_j}(\mR_+)$.
As in \cite{Kurylev2018}, we consider the interaction of waves in the open set
\beq\label{def_nxxi}
\nxxi  = M \setminus \bigcup_{j=1}^3 J^+(\gamma_{p_j, w_j}(\rho(p_j, w_j))),
\eeq
which is the complement of the causal future of the first cut points.
In $\nxxi$, the receding waves that we constructed are conormal distributions and any two of the null geodesics $\gamma_{p_j, w_j}(\mathbb{R}_+)$ intersect at most once, by \cite[Lemma 9.13]{Beem2017}.

We introduce the
definition of the regular intersection of three null geodesics at a point $q$, as in \cite[Definition 3.2]{Kurylev2018}.
\begin{df}\label{def_inter}
    We say the geodesics corresponding to
    $(p_j, w_j)_{j=1}^3$
    intersect regularly at a point $q$,  if  one has
    \begin{enumerate}[(1)]
        \item there are $0 < \varsm_j < \rho(p_j, w_j)$ such that $q = \gamma_{p_j, w_j}(s_j)$, for $j= 1, 2,3$,
        \item the vectors $\dot{\gamma}_{p_j, w_j}(s_j), j= 1, 2,3$ are linearly independent.
    \end{enumerate}
\end{df}
In addition, we introduce the following definition on the intersection of three submanifolds as in \cite[Definition 3.1]{Lassas2018}.
\begin{df}\label{df_intersect}
    We say three $1$-codimensional submanifolds $K_1, K_2, K_3$ intersect 3-transversally if
    \begin{enumerate}[(1)]
        \item $K_i, K_j$ intersect transversally at a codimension $2$ manifold $K_{ij}$, for $i < j$;
        \item $K_1, K_2, K_3$ intersect at a codimension $3$ submanifold $K_{123}$, for $i < j < k$;
        \item for any two disjoint index subsets $I, J \subset \{1, 2, 3\}$,
        the intersection of $\cap_{i \in I} K_i$ and $\cap_{j \in J} K_j$ is transversal if not empty.
    \end{enumerate}
\end{df}
By \cite{Lassas2018}, such $K_1, K_2, K_3$ intersect  with linearly independent normal covectors $\zeta^j \in N_q^*K_j$, $j = 1,2,3$.
If three null geodesics $\gamma_{p_j, w_j}, j = 1,2,3,4$ intersect regularly at $q$, then we can always construct $K_j$ with small enough $\kappa_0$ such that
they intersect 3-transversally at $q$.

For convenience, in  some cases we denote the triplet by $(\vec{p}, \vec{w}) = (p_j, w_j)^3_{j=1}$.
We  omit the parameters $p_j, w_j, s_0$ and use the following notations
\[
\gamma_j= \gamma_{p_j, w_j},
\quad \Upsilon_j = \Upsilon(p_j, w_j, \kappa_0),
\quad \Sigma_j=\Sigma({p_j, w_j, s_0}),
\quad K_j = K({p_j, w_j, s_0}), \quad \Lambda_j = \Lambda({p_j, w_j, s_0}),
\]
and
\[
\Lambda_{ij} = N^*(K_i \cap K_j), \quad \Lambda_{123} =  N^*(K_1 \cap K_2 \cap K_3),
\]
when the null geodesics $\gamma_{x_j, \xi_j}$ intersects regularly at $q$.
We define
\[
\Lambda^{(1)} = \cup_{j=1}^3 \Lambda_j, \quad \Lambda^{(2)} = \cup_{i<j} \Lambda_{ij}
\]
and denote by $\Lambda_{123}^g$ the flow-out of $\Lambda_{123} \cap \Char(\sq_g)$ under the null bicharacteristics in $T^*\tM$.

In the following, we consider distinct lightlike vectors $(p_j, w_j)_{j=1}^3$.
Note that $\Upsilon_j$ are chosen to be supported near $p_j$.
Then this condition allows us to choose sources with disjoint supports, i.e.,
\beq\label{assump_Upsj}
\supp(\Upsilon_j) \cap \supp(\Upsilon_k) = \emptyset,
\quad \text{ for } 1\leq j \neq k \leq 3.
\eeq
\subsection{Singularities in nonlinear interaction}
Recall $v_j \in I^\mu(\Lambda_j)$ for $j =1, 2, 3$ by our construction and $\lU_3$ is the solution to
\[
\square_g \lU_3 -  6\beta v_{1}v_{2}v_{3} = 0, \quad \text{ with }
\mathcal{R}_-[\lU_3] = 0.
\]
Using the solution operator $Q_s$ defined in Section \ref{subsec_causalinverse},
we can write $\lU_3 = -6Q_s(\beta v_{1}v_{2}v_{3})$.
By \cite[Lemma 3.6]{Lassas2018} and \cite[Propostion 3.7]{Lassas2018}, we have the following proposition.
\begin{prop}\label{pp_v1v2v3}
    Suppose the submanifolds $K_1, K_2, K_3$ intersect 3-transversally at $K_{123}$.
    Then in $\nxxi$ (see (\ref{def_nxxi}) for the definition), there is a decomposition $v_1v_2v_3 = w_0 + w_1 + w_2$ with
    \beq\begin{split}\label{eq_v1v2v3}
            &w_0 \in I^{3\mu + 2}(\Lambda_{123}),\\
            &\wfset(w_1) \subset \Lambda^{(1)} \cup (\Lambda^{(1)}(\epsilon) \cap \Lambda_{123}) ,
            \quad  \wfset(w_2)\subset \Lambda^{(1)} \cup \Lambda^{(2)}.
    \end{split}
    \eeq
    In particular, for $(q, \zeta) \in \Lambda_{123}$ the leading term $w_0$ has the principal symbol
    \[
    \sigma_p(w_0)(q, \zeta) =
    6(2\pi)^{-1}
    \prod_{m=1}^3\sigma_p(v_m) (q, \zeta^m), \quad \text{ where } \zeta = \zeta^1 + \zeta^2 + \zeta^3.
    \]
\end{prop}
Recall $Q_s(f)$ coincides with $Q_g(f)$ for $f \in \lE'(M)$.
Thus, we have the following proposition.
\begin{prop}\label{pp_U3}
    Suppose $K_1, K_2, K_3$ intersect 3-transversally at $K_{123}$.
    Let $\Lambda_{123}^g$ and $\Lambda^{(1)}$ be defined as in Section \ref{subsec_constructwaves}.
    In $\nxxi$ away from $\Lambda^{(1)}$, we have
    \[
    \lU_3 \in I^{3\mu+ \frac{1}{2}, -\frac{1}{2}}(\Lambda_{123}, \Lambda_{123}^g).
    \]
    In particular, let $(y, \eta) \in L^{+,*}M$ lie along a future pointing null bicharactersitic of $\sq_g$ starting from $(q, \zeta) \in \Lambda_{123}$.
    Suppose $(y, \eta)$ is away from $\Lambda^{(1)}$ and before the first cut point of $q$.
    Then $(y, \eta) \in \nxxi$ and the principal symbol of $\lU_3$ is given by
    \[
    \sigma_p(\lU_3)(y, \eta) =
    -6(2\pi)^{-2}
    \sigma_p(Q_g)(y, \eta, q, \zeta)\beta(q)
    \prod_{m=1}^3\sigma_p(v_m) (q, \zeta^m), \quad \text{ where } \zeta = \zeta^1 + \zeta^2 + \zeta^3.
    \]
\end{prop}
Next, recall $\lR_\pm$ is the restriction operator for the smooth null hypersurfaces $S_\pm$.
By \cite[Chapter 5.1]{Duistermaat2010},
such $\lR_\pm$ are Fourier integral operators of order $1/4$ associated with the canonical relation
\[
\Lambda_\pm = \{(y_|, \eta_|, y, \eta) \in T^*(S_\pm) \times T^*M\setminus 0: \quad y_| = y, \  \eta_| = \eta|_{T_y^*S_\pm}\}.
\]
Although with null hypersurfaces we have $N^*S_\pm \cap \Char(\sq_g) \neq \emptyset$, yet the singularities of $\lU_3$ interests $S_+$ transversally and thus
\[
\sigma_p(\partial_{\ep_1}\partial_{\ep_2}\partial_{\ep_3} \lN(\ep_1\Upsilon_1 + \ep_2\Upsilon_2+ \ep_3\Upsilon_3)|_{\ep_1 = \ep_2 = \ep_3 = 0})(y_|, \eta_|) = \sigma_p(\lR_+[\lU_3])(y_|, \eta_|) \neq 0
\]
as long as $\sigma_p(\lU_3)(y, \eta) \neq 0$, according to Proposition \ref{pp_RFFIO}.

%% file: layerstripping.tex
\section{A layer stripping method}\label{sec_layer}
In this section, we propose a layer stripping procedure, which allows us to reconstruct the spacetime in small pieces, starting in a neighborhood of spacelike infinity $R$.
The idea is that in a sufficiently small region of the spacetime, there are no conjugate points such that one can generate conormal waves and use them to produce new singularities.
By concatenating such local reconstructions, we cam eventually get to reconstruct the whole spacetime, by a compactness argument.
For this purpose, we consider the radius of injectivity of Lorentzian manifolds, for example see \cite{chen2008inj}.
To define it, we consider a reference Riemannian metric given by
$g_R = \beta(T,X) \dif T^2 + \kappa(T,X).$
We consider the geodesic ball
\[
B(0,r) = \{ v\in T_q M: g_R(v, v) < r\} \subseteq T_qM
\]
determined by the reference Riemannian inner product at $q$.
Let $\exp_q: T_qM \rightarrow M$ be the exponential map given by the Lorentzian metric $g$.
For fixed $g_R$,
the radius of injectivity $\mathrm{Inj}(q)$ is defined as the largest radius $r$ such that $\exp_q$ is a diffeomorphism from $B(0,r)$ onto $\lB(q,r) \subseteq M$.
Note that in the compact subset $\overline{M}$ the Euclidean metric $g_e = \dif T^2 + \dif X^2$ is equivalent to the reference Riemannian metric $g_R$.
In the following, for convenience we regard $\lB(q,r)$ as its equivalence in the Euclidean case.

\begin{lm}\label{lm_radius}
    Let $0< T_0 < 1$ be fixed  and let $K \coloneqq \textstyle \bigcup_{T< T_0} J(\Smi(T), \Spl(T))$.
    There exists $\delta > 0$ such that for any $q \in K$,
    there are no cut points along any null geodesic segments contained in $\lB(q, \delta) \cap K$.
\end{lm}
\begin{proof}
    As $K$ is compact, the radius of injectivity $\mathrm{Inj}(q)$ for any $q \in K$ has a lower bound $r_0 > 0$.
    We choose $\delta = r_0/2$.
    For any $q \in K$, if $q', q'' \in \lB(q, \delta)$ are connected by a null geodesic segment, then
    \[
    \mathrm{dist}_{g_R}(q',q'') \leq \mathrm{dist}_{g_R}(q',q)  + \mathrm{dist}_{g_R}(q,q'') < r_0.
    \]
    Thus, $q''$ is contained in $\lB(q', r_0)$, within the radius of injectivity,  and therefore the exponential map there is a diffeomorphism.
\end{proof}
\subsection{Parameterize $S_\pm$ using null geodesics}\label{subsec_familycurves}
The fact that no cut points of $i_\pm$ along the past and future null infinity  $S_\pm$ implies
that $S_\pm$ can be smoothly parameterized by a family of null geodesics.
Assuming the conformal class of $g|_{S^+}$, one can follow the ideas in \cite[Section 2.1.1]{MScfinal} to construct a family of null pregeodesics
\[
\mupl_{a}: [-\varsm_a^+, 0] \rightarrow \bar{S}_+,
\quad \text{ with }\mupl_a(0) = i_+ \text{ and }  \mu_a(-\varsm_a^+) \in R,
\]
which covers $\bar{S}_+$ and smoothly depends on the parameter $a \in \mS^{2}$.
Similarly, one can construct a family of null pregeodesics
\[
\mumn_{b}: [0, \varsm_b^-] \rightarrow \bar{S}_-,
\quad \text{ with }\mumn_b(0) = i_- \quad \text{ and } \quad \mumn_b(\varsm_b^-) \in R,
\]
which covers $\bar{S}_-$ and smoothly depends on the parameter $b \in \mS^{2}$.
Moreover, the proof of \cite[Lemma 2.1.1]{MScfinal} implies the following properties about the spacelike infinity
\[
R \coloneqq \partial J^+({i^-}) \cap \partial J^-({i^+}) = \{(0, X) \in \tM: |X| = 1\}.
\]
\begin{lm}
Then for any $p \in R$, we have
    \begin{itemize}
        \item[(1)] $\lL_p^- \cap M = J^-(p) \cap M  = \emptyset$ and $\lL_p^+ \cap M = J^+(p) \cap M  = \emptyset$;
        \item[(2)] for $q \in M$, one has $q \notin J^-(p)$ and $q \notin J^+(p)$.
    \end{itemize}
\end{lm}

\subsection{Null normal geodesic congruences}\label{subsec_congruence}
In the following, recall some results about null normal geodesic congruences in \cite{aretakis2012lecture}.
\begin{definition}\label{def_twosurface}
We say $S \subseteq \bar{M}$ is a regular $2$-surface if it is a smooth $2$-codimensional submanifold of $\tM$ contained in a Cauchy surface and is homeomorphic to $\mS^2$.
\end{definition}
Note that such $S$ is compact and \textit{acausal}, in the sense that no causal curve meets $S$ more than once.
The Lorentzian metric restricted to $T_pS$ for any $p \in S$ is positive definite and therefore $(T_pS)^\perp$ is a $2$-dimensional Lorentzian space.

There are exactly $2$ future pointing null vectors in $(T_pS)^\perp$.
One of them projects to the exterior of $S$ and we call it $\vout \in L^+\tM$.
Another one projects to the interior of $S$ and we call it $\vin \in L^+\tM$.
Similarly, there are exactly $2$ past pointing null vectors $(T_pS)^\perp$, which are outward and inward respectively.
We consider the future pointing null geodesics $\gamma_{p, \vout}$ starting from $p$ in the direction of $\vout$ and $\gamma_{p, \vin}$ starting from $p$ in the direction of $\vin$.
We define the following sets in $M$ formed by these null geodesics for each point in $S$, i.e.,
\beq\label{def_Cset}
\Coutpl = \bigcup_{p \in S} \gamma_{p, \vout}(\mR_+) \cap M
\quad\text{ and }\quad
\Cinpl = \bigcup_{p \in S} \gamma_{p, \vin}(\mR_+) \cap M
\eeq
as the future pointing outward null geodesic congruences and the future pointing inward null geodesic congruences normal to $S$.
In general, these sets are not smooth hypersurfaces as caustics may exist.
By \cite[Proposition 2.2.2]{aretakis2012lecture}, they are null hypersurfaces and one can use the null geodesics $\gamma_{p, \vout}$ or $\gamma_{p, \vin}$ to parameterize them.
Similarly, we can define the past pointing null geodesic congruences $\Coutmi$ and $\Cinmi$ normal to $S$.

Now consider the causal future $\Jpl(S)$ and the causal past $\Jmi(S)$, which are future set and past set respectively.
We consider the boundaries of them.
\begin{lm}\cite[Chap 3 Theorem 3.9]{Beem2017}
The boundary $\partial\Jpl(S)$ and $\partial\Jmi(S)$ are closed \textit{achronal} Lipschitz topological hypersurfaces, in the sense that no timelike curve meets them more than once.
\end{lm}
By \cite[Proposition 2.4.2]{aretakis2012lecture},
these boundaries are contained in the null normal geodesic congruences, i.e.,
\[
\partial \Jpl(S) \subseteq \Coutpl(S) \cup \Cinpl(S) \quad \partial \Jmi(S) \subseteq \Coutmi(S) \cup \Cinmi(S).
\]
In addition, we have the following results.
\begin{lm}{\cite[Theorem 1]{akers2018boundary}}\label{lm_achronalboundary}
A point $y$ is on $\partial \Jpl(S)$ if and only if
\begin{itemize}
    \item[(1)] $y$ lies on a future pointing null geodesic $\gamma$ normal to $S$, i.e., $y \in \Coutpl(S) \cap \Cinpl(S)$, before conjugate points; and
    \item[(2)] $\gamma$ does not intersect any other null geodesic orthogonal to $S$ strictly between $S$ and $y$.
\end{itemize}
\end{lm}
Thus, for any regular $2$-surface, we can construct its null normal conjugate congruences.
Before caustics and focal points, these congruences are smooth null hypersurfaces and they form the boundary of $\partial \Jpl(S)$ or $\partial \Jmi(S)$.
In the following, we consider the regular $2$-surfaces given by
\[
\Spl(T_0) \coloneqq \{p\in \Spl: T(p) = T_0\}, \quad \Smi(T_0) \coloneqq \{p\in \Smi: T(p) = -T_0\},
\]
for some $0<T_0< 1$.


\subsection{Layer stripping steps}\label{subsec_layersteps}
We state a layer stripping method in the following.
For fixed $0< T_0 < 1$, we use the following steps to reconstruct the metric in the open set
\[
I(T_0) \coloneqq
I(\Smi(T_0), \Spl(T_0)).
\]
\input{step1}
\input{step2}
\input{step3}

\input{step4}

%% file: step1.tex
\subsubsection{Step 1}
Recall  $R = \{(0, X) \in \tM: |X| = 1\}$ is the spacelike infinity.
In this step, we would like to reconstruct $g$ up to conformal diffeomorphisms in a small neighborhood of $R$ in $M$, i.e., we consider the reconstruction in $I(T_1)$ for small $T_1 > 0$.

For this purpose, suppose $T_1$ is small and to be specified later.
First, we pick an arbitrary $p_0 \in R$.
By Section \ref{subsec_familycurves}, there exist unique null geodesics
\[
\mupl_a: [-\varsm_a^+, 0] \rightarrow \bar{S}_+ \quad \text{ such that }
\mu_a(-\varsm_a^+) = p_0 \text{ and } \mupl_a(0) = i_+
\]
and
\[
\mumn_b: [0, \varsm_b^-] \rightarrow \bar{S}_- \quad \text{ such that }
\mumn_b(0) = i_- \text{ and } \mu_b(\varsm_b^-) = p_0.
\]
for some $a,b \in \mS^2$.
\begin{figure}[h]
    \centering
    \includegraphics[width=0.25\linewidth]{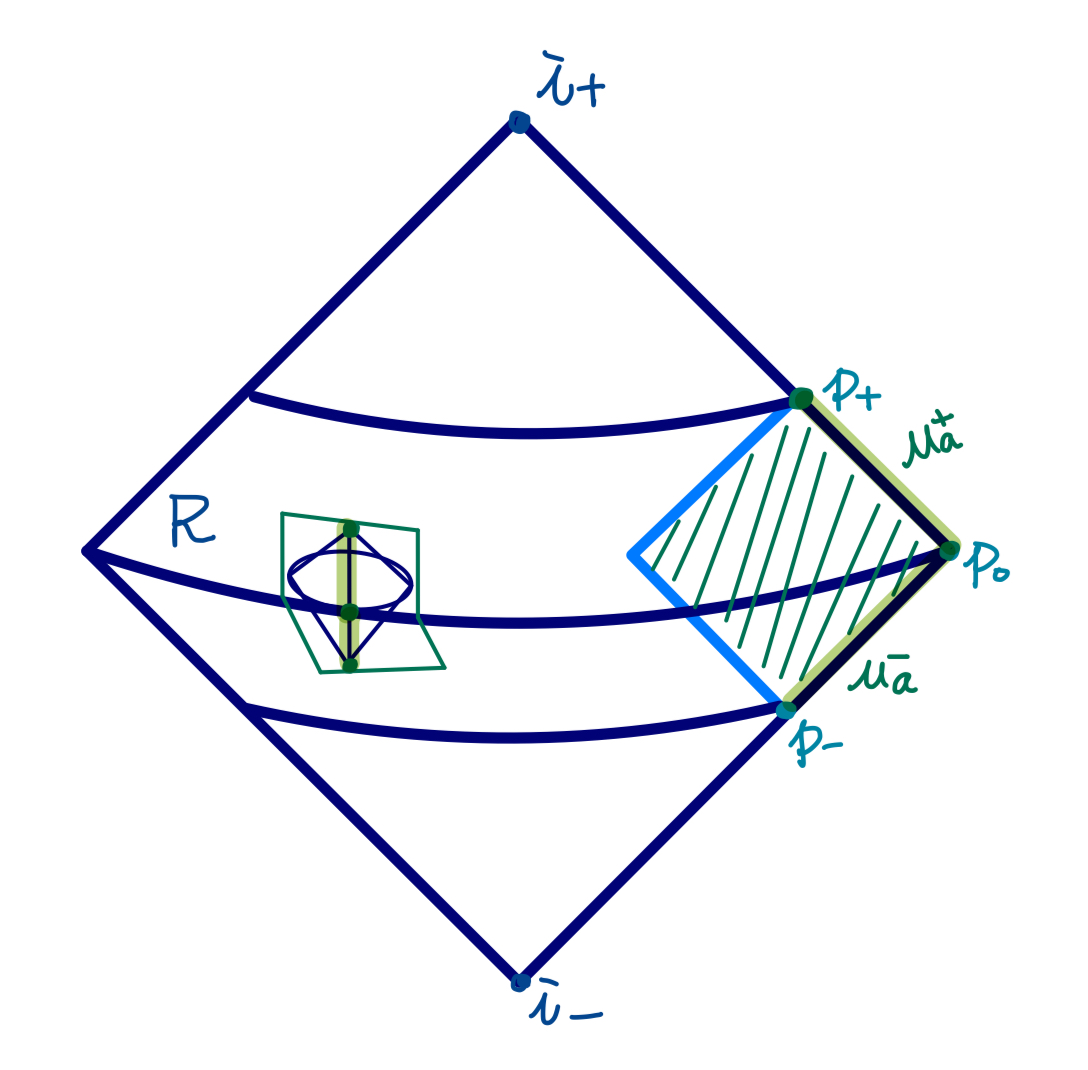}
    \caption{Step 1.}
\end{figure}
Let $\delta > 0$ satisfy Lemma \ref{lm_radius} such that $\lB(p_0, \delta)$ has no cut points along any null geodesic segments.
With $T_1 > 0$ given, we can find $\varsm_1, \varsm_2 > 0$ such that
\[
p_+  = \mupl_a(-\varsm_1) \text{ with }T(p_+) = T_1,
\quad  p_- = \mumn_b(\varsm_2)\text{ with } T(p_-) = -T_1.
\]
We observe we can choose $T_1$ small enough such that $p_-$ and $p_+$ are arbitrarily close to $p_0$.
Thus, there exists $T_1> 0$ such that
\[
J(p_-, p_+) \subseteq \lB(p_0, \delta).
\]
Further, let $U_+ \subseteq \Spl(0, T_1)$ be a small open neighborhood of the geodesic segment $\mupl_a([-\varsm_a^+, -\varsm_1])$.
Let $U_- \subseteq \Smi(0, T_1)$ be a small open neighborhood of the geodesic segment $\mupl_b([\varsm_2, \varsm_b^-])$.
Here we use the notation
\[
{S}_\pm(T_1, T_2) \coloneqq {S}_\pm \cap \{T_1 < \pm T < T_2\}.
\]
As $J(p_-, p_+)$ is closed and contained in an open set $\lB(p_0, \delta)$,
we can find $U_+, U_-$ small enough such that
\[
I(U_-, U_+) \subseteq \lB(p_0, \delta),
\]
where there are no cut points along any null geodesic segment.
Indeed, for sufficiently small $U_\pm$,
the set $I(U_-, U_+)$ is an open small neighborhood of $I(p_-, p_+)$.

For each $p_0 \in R$,
we find such $p_\pm$ and $U_\pm$.
and set $W = I(U_-, U_+)$.
We use Scheme 1 in Section \ref{sec_scheme1} to reconstruct the metric in $W$, by sending receding waves on $U_-$ and detecting new singularities on $U_+$.
Thus, we choose
\[
\begin{split}
    T_{1}  = \sup\{T \in (0,1): &\text{ for each $p_0 \in R$, we have $J(p_-, p_+) \subseteq \lB(p_0, \delta)$},\\
    &\quad \quad \quad \quad \text{ where $p_\pm \in S_\pm(T)$ are constructed for $p_0$ as above}\}.
\end{split}
\]

This enables us to we reconstruct a connected region given by the union of these diamond sets.
Note this region does not necessarily fully cover $I(T_1)$, but part of its boundary is given by $\mathrm{cl}(S_\pm(0, T_1))$.
For this purpose, we consider the set
\[
Y_0 \coloneqq \partial J^-(S_+(T_1')) \cap \partial J^+(S_-(T_1'))
\]
for some small $T_1'> 0$ to be specified later.
If $Y_0 \neq \emptyset$, we pick arbitrary $y_0 \in Y_0$ and consider the following construction.
With $y_0 \in \partial J^-(S_+(T_1'))$, by Lemma \ref{lm_achronalboundary},
there exists a unique future pointing null geodesic $\gamma_+(\mR_+)$ starting from $y_0$ and hits $\Spl(T_1')$ at point $y_+$, on or before the first conjugate point and the first focal point.
With $y_0 \in \partial J^+(S_-(T_1'))$,
there exists a unique past pointing null geodesic $\gamma_-(\mR_-)$ starting from $y_0$ and hits $\Smi(T_1')$ at point $y_-$, on or before the first conjugate point and the first focal point.
\begin{figure}[h]
    \centering
    \includegraphics[width=0.27\linewidth]{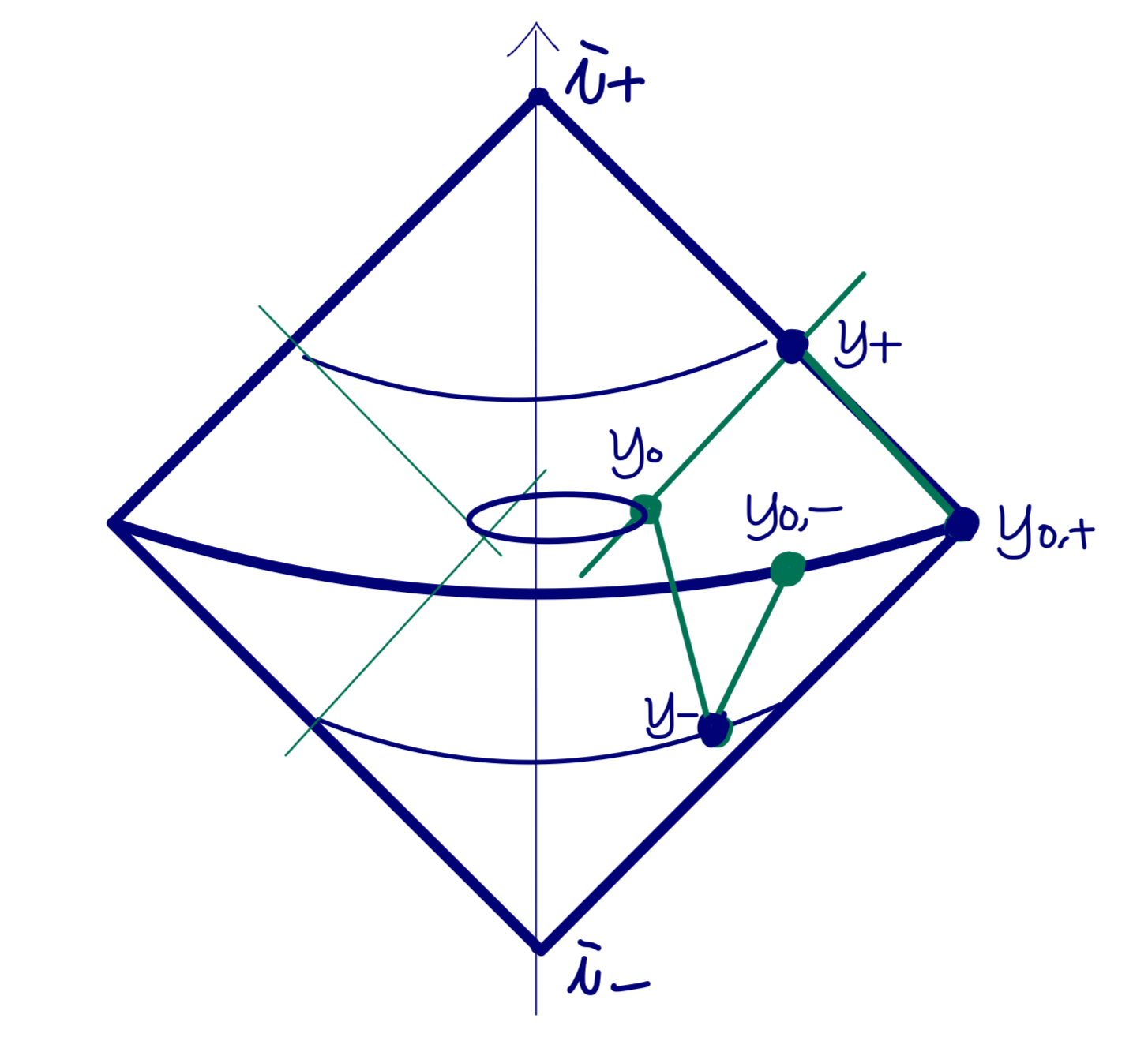}
    \caption{Step 1, part 2.}
\end{figure}
As before,
by Section \ref{subsec_familycurves} there exist unique null geodesics
\[
\mupl_a: [-\varsm_a^+, 0] \rightarrow \bar{S}_+ \quad \text{ such that }
\mupl_a(-\varsm_1) = y_+ \text{ and }
\mu_a(-\varsm_a^+) = \yzeropl,
\]
and
\[
\mumn_b: [0, \varsm_b^-] \rightarrow \bar{S}_- \quad \text{ such that }
\mu_b(\varsm_b^-) = \yzeromi \text{ and }  \mumn_b(\varsm_2) = y_-.
\]
for some $a,b \in \mS^2$ and $\varsm_1, \varsm_2 > 0$.
We choose small $T_1'$ such that
\[
J(y_-, y_+) \subseteq \lB(y_0, \delta).
\]
As before, we set $U_+ \subseteq \Spl(0,T_1')$ be a small open neighborhood of the null geodesic segment $\mupl_a([-\varsm_a^+, -\varsm_1])$ and $U_- \subseteq \Smi(0, T_1')$ be that of the null geodesic segment $\mupl_b([\varsm_2, \varsm_b^-])$.
As $J(y_-, y_+)$ is closed and contained in an open set $\lB(y_0, \delta)$,
we can find $U_+, U_-$ small enough such that
\[
I(U_-, U_+) \subseteq \lB(y_0, \delta),
\]
where there are no cut points along any null geodesic segment.
For $W = I(U_-, U_+)$, we reconstruct the metric there using Scheme 1 in Section \ref{sec_scheme1}.
In this case, we choose
\[
\begin{split}
    T_{1}'  = \sup\{T \in (0,1): &\text{ for each $y_0 \in Y_0$, we have $J(y_-, y_+) \subseteq \lB(y_0, \delta)$},\\
    &\quad \quad \quad \quad \text{ where $y_\pm \in S_\pm(T)$ are constructed for $y_0$ as above}\}.
\end{split}
\]
This enables us to reconstruct a region given by the union of such $I(y_-, y_+)$.
Note this region has the same boundary as $I(T_1')$ within $M$, i.e., the part
\[
\big(\partial J^-(S_+(T_1')) \cap J^+(S_-(T_1'))\big) \cup  \big(J^-(S_+(T_1')) \cap \partial J^+(S_-(T_1')) \big).
\]
Indeed,
the set $\Spl(T_1')$ is a regular $2$-surface by Definition \ref{def_twosurface}.
Its past set has a boundary $\partial \Jmi(\Spl(T_1))$, which is an achronal Lipschitz topological hypersurface contained in the null normal geodesic congruences by \cite[Proposition 2.4.2]{aretakis2012lecture} \ref{lm_achronalboundary}.
In particular, we have
\[
\Pmi(T_1) \coloneqq \partial \Jmi(\Spl(T_1')) \cap M  \subseteq  \Cinmi(\Spl(T_1')),
\]
and similarly we have
\[
\Ppl(T_1) \coloneqq  \partial \Jpl(\Smi(T_1')) \cap M  = \Cinpl(\Smi(T_1')),
\]
where $\Cinmi$ and $\Cinpl$ are defined in (\ref{def_Cset}).
Thus, the union of such $I(y_-, y_+)$ gives us the same boundary as that of $I(T_1)$ within $M$.

Now we set $T_1 = T_1'$, if we have $T_1'< T_1$.
Combining these two parts, we reconstructed the open region $I(T_1)$.

%

%


%% file: step2.tex
\subsubsection{Step 2}
In the following, assume we have reconstructed $I(T_1)$, for some $T_1> 0$.
Now the goal is to use the reconstructed region to recover $g$ in a small neighborhood of $\bar{S}_+$, i.e., in the future set $I^+(\Smi(T_2))$ for some $0< T_2 \leq T_1$.

{
We emphasize the metric $g$ in $I(T_1)$ is reconstructed up to conformal diffeomorphisms.
Thus we can pick an arbitrary representative in the equivalent class, say $\hat{g} = \phi^*(\rho^2 g)$, where $\rho \in C^\infty(M)$ is a conformal factor.
Then by the identity, one has
\beq\label{eq_conormalfactor}
\rho^{-(n-2)/2}  \square_g (\rho^{(n-2)/2} u)  = \rho^2 \square_{\rho^2 g} u  -\rho^2  (\rho^{(n-2)/2}\square_{\rho^2 g} \rho^{(-n-2)/2}) u,
\eeq
where $\gamma \coloneqq \rho^{(n-2)/2}\square_{\rho^2 g} \rho^{(-n-2)/2}$ is smooth on $M$.
As is stated in \cite{hintz2015semilinear}, a conformal factor only reparameterizes bicharacteristics, and thus the interaction of conormal waves for $\hat{g}$ and $g$ essentially have the same structure of singularities.
In particular, $\hat{g}$ and $g$ determines the same lens relation.
}

Now suppose $(I(T_1), \hat{g}|_{I(T_1)})$ is known.
First, we pick a new $p_0 \in \Spl(T_1)$.
Recall in Step 1 we reconstruct a slightly larger region $I(U_-, U_+)$, with $U_-, U_+$ defined there.
Thus, we may assume we reconstructed a small neighborhood of $I(T_1)$.
Recall we denote by $\Pmi(T_1)$  the boundary $\partial \Jmi(\Spl(T_1))$ within $M$.
It is an achronal Lipschitz topological hypersurfaces contained in the null normal geodesic congruences of $\Spl(T_1)$.
Thus, we may assume $\Pmi(T_1)$ is  contained in the reconstructed region.


\begin{figure}[h]
    \centering
    \includegraphics[width=0.5\linewidth]{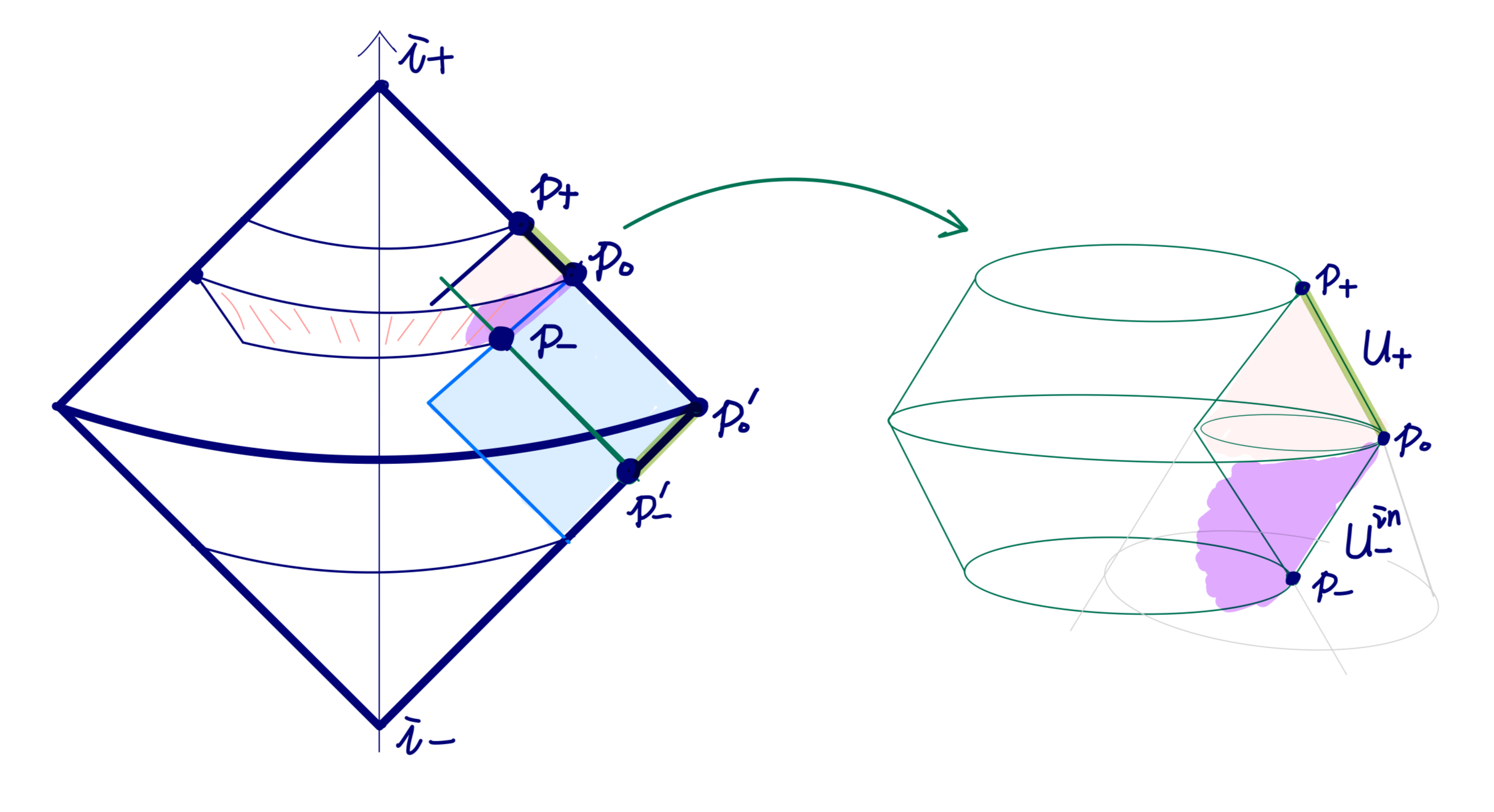}
    \caption{Step 2.}
\end{figure}

Let $\delta>0$ be given by Lemma \ref{lm_radius} and we focus on the reconstruction in $\lB(p_0, \delta)$, for $p_0 \in \Spl(T_1)$.
Let $T_2 > 0$ be small and to be specified in the following.
Again, there exists a unique null pregeodesic $\mupl_a:[-\varsm_a^+, 0] \rightarrow \bar{S}_+$ passing through $p_0$ and we write
\[
p_0 = \mupl_a(-\varsm_0) \text{ for some } 0 < \varsm_0 < \varsm_a^+
\quad \text{ and }\quad
p_0''\coloneqq \mupl_a(-\varsm_a^+) \in R.
\]
With $T_2> 0$ given, we can find some $0 < \varsm_1< \varsm_0 < \varsm_a^+$ such that
\[
p_+  = \mupl_a(-\varsm_1)
\quad \text{ and } \quad
T(p_+)  = T_1  + T_2.
\]
Then we consider the unique past pointing null geodesic starting from $p_0$ normal to $\Spl(T_1)$ and we denote it by $\gamma_-(\mR_+)$.
We consider the future set $\Jpl(\Smi(T_2))$ and the achronal boundary
\[
\Ppl(T_2) \coloneqq \partial \Jpl(\Smi(T_2)) \cap M.
\]
Note that  $\gamma_-(\mR_+)$ intersects $\Ppl(T_2)$ exactly at one point
\[
\pmi \in \gamma_-(\mR_+) \cap  \Ppl(T_2).
\]
{\teal Write a lemma about this.}
By Lemma \ref{lm_achronalboundary}, such $\pmi$ lies in a normal null geodesic starting from $\Smi(T_2)$ and thus we can find a $\pmi' \in \Smi(T_2)$ satisfying $\pmi'< \pmi$.
Further, by considering the null geodesics on $\Smi$,
we can find a unique
$\mumn_b:[0, \varsm_b^-] \rightarrow \bar{S}_-$ passing through $\pmi'$
and some  $0< \varsm_3 < \varsm_b^-$ such that
\[
p_0' = \mumn_b(\varsm_b^-) \in R \quad \text{ and } \quad  p_-' = \mumn_b(\varsm_3) \in \Smi(T_2).
\]
We emphasize that the point $p_0'' \in R$ that we found before may not be exactly $p_0'$ and we do not necessarily have $p_0' < \ppl$.
But these do not affect the reconstruction.


Moreover, we observe one can choose $T_2$ such that $p_+$ is arbitrarily close to $p_0$ and $p_-'$ arbitrarily close to $p_0'$, which enables $p_-$ to be arbitrarily close to $p_0$ as well.
Thus, there exists $T_2 > 0$ such that $p_+$ and $p_-$ are contained in a given small neighborhood of $p_0$.
Then we can find  $T_2 > 0$ such that
\[
J(p_-, p_+) \subseteq \lB(p_0, \delta)
\]
has no cut points along any null geodesic segments.
Then let $U_+ \subseteq \Spl(T_1, T_1 + T_2)$ be a small open neighborhood of the null geodesic segment $\mupl_a([-\varsm_0, -\varsm_1])$ from $p_0$ to $\ppl$.
Let $U_- \subseteq \Spl(0, T_2)$ be that of the null geodesic segment $\mumn_b([\varsm_3, \varsm_b^-])$ from $\pmi'$ to $p_0'$.

In addition,
with the achronal boundary $\Pmi(T_1) = \partial \Jmi(\Spl(T_1)) \cap M$,
the diamond set $I(U_-, U_+)$ can be partitioned as
\[
I(U_-, U_+) = W \cup \tUmi \cup W_0,
\]
where we denote by
\[
W = I(U_-, U_+) \setminus \Jmi(\Spl(T_1)),
\quad \tUmi = I(U_-, U_+) \cap \Pmi(T_1),
\quad W_0 = I(U_-, U_+) \cap I^-(\Spl(T_1)).
\]
Note such $W$ is a precompact set near $p_0$ and it can be contained in a small neighborhood of $p_0$ when $T_2$ is sufficient small.
By choosing small $T_2 > 0$, we can expect
\[
\overline{W}  \subseteq \lB(p_0, \delta)
\]
has no cut points along any null geodesic segments.

We emphasize in general $\tUmi$ is an achronal Lipschitz hypersurface.
Observe null geodesics starting from future pointing lightlike vectors transversal to $U_-$ will stay in $\Jpl(\Smi(T_2))$.
Such null geodesics might enter the region $W$ or might not.
If it enters $W$, then it will intersect $\tUmi$ exactly once.
This guarantees such $\tUmi$ is enough for our reconstruction.

In Section \ref{subsec_waves}, we would like to construct conormal distributions propagating along the null geodesic $\gamma_{\tp, \tw}(\mR_+)$ for $(\tp, \tw) \in L^+M$ with $\tp \in \tUmi$, by sending proper Lagrangian distributions singular near some $(p, w) \in \LUmi$.
Using these waves, we would like to reconstruct the metric in $W$ by detecting new singularities on $U_+$, using the scattering light observation sets in Section \ref{sec_SLOS}.
A more detailed reconstruction can be found in Section \ref{sec_scheme2}.

Note for each $p_0 \in \Spl(T_1)$, we can find $p_\pm$, $U_\pm$, and $\tUmi$ as above.
We perform the same reconstruction for each $p_0$ and this recovers a connected new region given by the union of diamond sets $I(p_-, p_+)$.
Then we choose
\[
\begin{split}
    T_{2}  = \sup\{T \in (0,1): &\text{ for each $p_0 \in \Spl(T_1)$, we have $\overline{W} \subseteq \lB(p_0, \delta)$},\\
    &\quad \quad \quad \quad \text{ where $\overline{W}$ are constructed for $p_0$ as above}\}.
\end{split}
\]
%
Note the boundary of this new region includes
\[
\mathrm{cl}(S_+(T_1, T_1 + T_2)) \cup (\Pmi(T_1) \cap \Jpl(\Smi(T_2)),
\]
as they are contained in the normal null geodesics congruences starting from $S_+(T_1)$.
Although this new region might not fully cover
\[
I(\Smi(T_2), \Spl(T_1 + T_2)) \setminus I^-(\Spl(T_1)),
\]
we can follow a similar argument as Step 1.
More explicitly, for $T_2'$ to be specified later, one can consider the compact set
\[
Y_0 \coloneqq \partial J^-(S_+(T_1 + T_2')) \cap \partial J^+(S_-(T_2')).
\]
If $Y_0 \neq \emptyset$, we pick arbitrary $y_0 \in Y_0$ and consider the same construction as in Step 1.
This enables us to reconstruct a new region, whose boundary includes
\[
\big(\partial J^-(S_+(T_1 + T_2')) \cap J^+(S_-(T_2'))\big) \cup  \big(J^-(S_+(T_1 + T_2')) \cap \partial J^+(S_-(T_2')) \big).
\]
By choosing the smaller one of $T_2, T_2'$ as $T_2$, we reconstruct $I(\Smi(T_2), \Spl(T_1 + T_2))$.

Next, we repeat this procedure to reconstruct
\[
I(T_1) \cup I(\Smi(T_2), \Spl(T_1 + T_2 + T_3)),
\]
for some $0 < T_3 \leq T_2 \leq T_1$, {\teal if the radius of $\Spl(T_2)$ is not too small.}
Otherwise, we can cover the rest region using a $\delta$-neighborhood of one point.
With our assumptions, such $T_1, T_2, \ldots$ cannot be too small so we have finite steps to reconstruct a small neighborhood of $S_+$ given by $I^+(\Smi(T_0))$, for some $T_0 > 0$.


%% file: step3.tex
\subsubsection{Step 3}
In the following, assume we have reconstructed $I^+(\Smi(T_1))$, for some $T_1> 0$.
Now we flip the positive and negative sign to consider a region below $I(T_1)$.
The goal is to use the reconstructed region to recover $g$ in this new region.
More explicitly, suppose $(I^+(\Smi(T_1)), \hat{g}|_{I^+(\Smi(T_1))})$ is known.
First, we pick a new $p_0 \in \Smi(T_1)$.
In Step 2, we reconstruct a slightly larger region $I(U_-, U_+)$ and therefore we may assume we have reconstructed a small neighborhood of $I^+(\Smi(T_1))$.

\begin{figure}[h]
    \centering
    \includegraphics[width=0.5\linewidth]{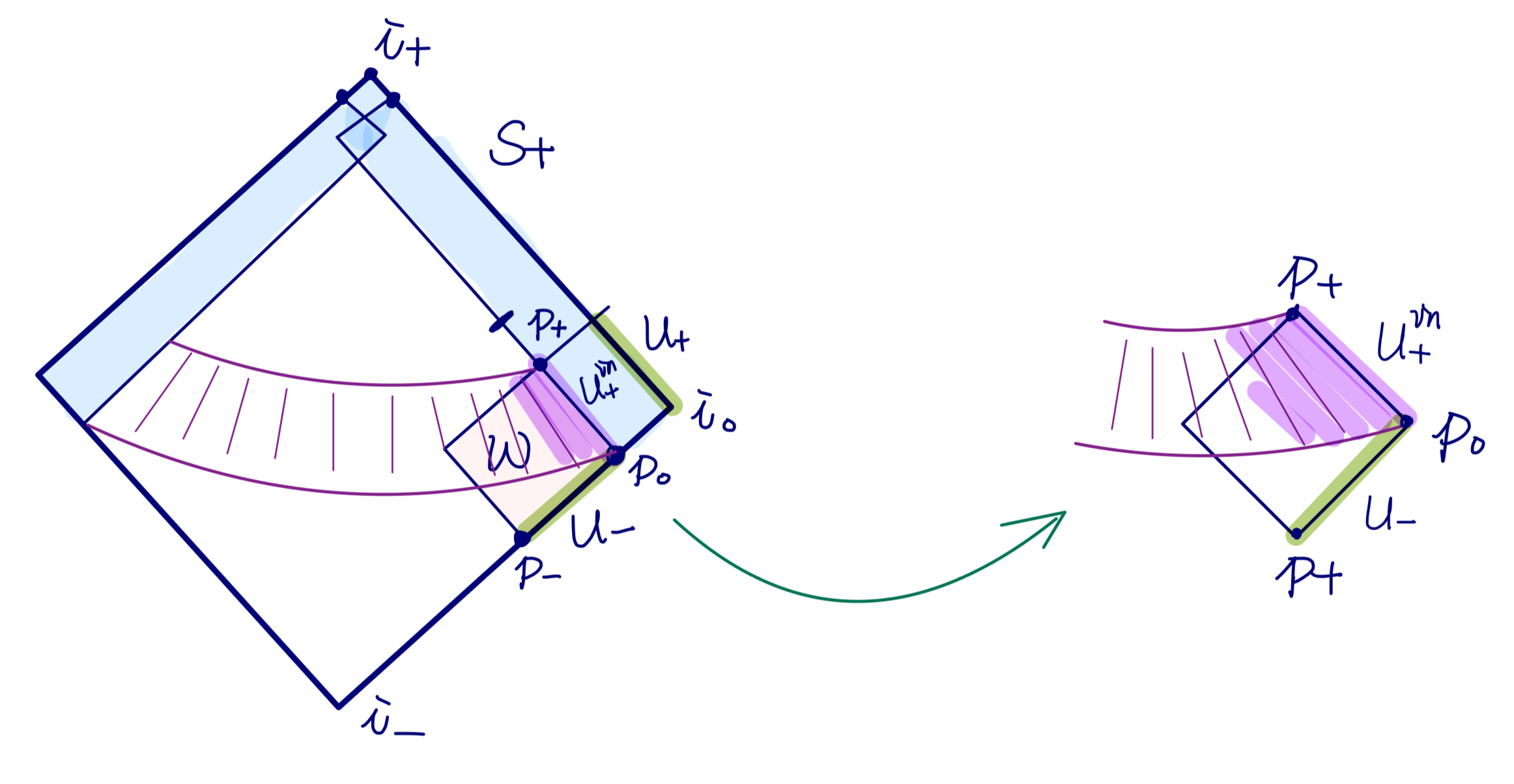}
    \caption{Step 3.}
\end{figure}

Recall we denote by $\Ppl(T_1)$ the boundary  $\partial \Jpl(\Smi(T_1))$ within $M$.
It is an achronal Lipschitz topological hypersurface contained in the null normal geodesic congruences of $\Smi(T_1)$.
Thus, we may assume $\Ppl(T_1)$ is contained in the reconstructed region.

Let $\delta>0$ be given by Lemma \ref{lm_radius} and we focus on the reconstruction in $\lB(p_0, \delta)$, for $p_0 \in \Smi(T_1)$.
Let $T_2 > 0$ be small and to be specified in the following.
Again, there exists a unique null geodesic $\mumn_b:[0, \varsm_b^-] \rightarrow \bar{S}_-$ passing through $p_0$ and we write
\[
p_0 = \mumn_b(\varsm_0) \text{ for some } 0 < \varsm_0 < \varsm_b^-
\quad \text{ and }\quad
p_0''\coloneqq \mumn_b(\varsm_b^-) \in R.
\]
With $T_2> 0$ given, we can find some $0 < \varsm_1< \varsm_0 < \varsm_b^-$ such that
\[
p_-  = \mumn_b(\varsm_1)
\quad \text{ and } \quad
T(p_-)  = -(T_1  + T_2).
\]
Then we consider the unique future pointing null geodesic starting from $p_0$ normal to $\Smi(T_1)$ and we denote it by $\gamma_+(\mR_+)$.
We consider the future set $\Jmi(\Spl(T_2))$ and the achronal boundary
\[
\Pmi(T_2) \coloneqq \partial \Jmi(\Spl(T_2)) \cap M.
\]
Note that  $\gamma_+(\mR_+)$ intersects $\Pmi(T_2)$ exactly at one point
\[
\ppl \in \gamma_+(\mR_+) \cap  \Pmi(T_2).
\]
By Lemma \ref{lm_achronalboundary}, such $\ppl$ lies in a normal null geodesic starting from $\Spl(T_2)$ and thus we can find a $\ppl' \in \Spl(T_2)$ satisfying $\ppl' > \ppl$.
Further, by considering the null geodesics on $\Spl$,
we can find a unique
$\mupl_a:[-\varsm_a^+, 0] \rightarrow \bar{S}_+$ passing through $\ppl'$
and some  $0< \varsm_3 < \varsm_a^+$ such that
\[
p_0' = \mupl_a(-\varsm_a^+) \in R \quad \text{ and } \quad  \ppl' = \mupl_a(-\varsm_3) \in \Spl(T_2).
\]
We emphasize that the point $p_0'' \in R$ that we found before may not be exactly $p_0'$ and we do not necessarily have $p_0' < \ppl$.
But these do not affect the reconstruction.


Moreover, we observe one can choose $T_2$ such that $\pmi$ is arbitrarily close to $p_0$ and $\ppl'$ arbitrarily close to $p_0'$, which enables $\ppl$ to be arbitrarily close to $p_0$ as well.
Thus, there exists $T_2 > 0$ such that $\pmi$ and $\ppl$ are contained in a given small neighborhood of $p_0$.
Then we can find  $T_2 > 0$ such that
\[
J(p_-, p_+) \subseteq \lB(p_0, \delta)
\]
has no cut points along any null geodesic segments.
Then let $U_+ \subseteq \Spl(0, T_2)$ be a small open neighborhood of the null geodesic segment $\mupl_a([-\varsm_a^+, -\varsm_3])$ from $p_0'$ to $\ppl'$.
Let $U_- \subseteq \Smi(T_1, T_1 + T_2)$ be that of the null geodesic segment $\mumn_b([\varsm_1, \varsm_0])$ from $\pmi$ to $p_0$.

In addition,
with the achronal boundary $\Ppl(T_1) = \partial \Jpl(\Smi(T_1)) \cap M$,
the diamond set $I(U_-, U_+)$ can be partitioned as
\[
I(U_-, U_+) = W \cup \tUpl \cup W_0,
\]
where we denote by
\[
W = I(U_-, U_+) \setminus \Jpl(\Smi(T_1)),
\quad \tUpl = I(U_-, U_+) \cap \Ppl(T_1),
\quad W_0 = I(U_-, U_+) \cap I^+(\Smi(T_1)).
\]
Note such $W$ is a precompact set near $p_0$ and it can be contained in a small neighborhood of $p_0$ when $T_2$ is sufficient small.
By choosing small $T_2 > 0$, we can expect
\[
\overline{W}  \subseteq \lB(p_0, \delta)
\]
has no cut points along any null geodesic segments.

We emphasize in general $\tUpl$ is an achronal Lipschitz hypersurface.
Observe null geodesics starting from future pointing lightlike vectors transversal to $U_-$ will enter the region $W$.
Such null geodesics might enter the reconstructed region $W_0$ or might not.
If it enters $W_0$, then it will intersect $\tUpl$ exactly once and then leave $M$ from $U_+$ .
This guarantees such $\tUpl$ is enough for our reconstruction.


In Section \ref{subsec_waves}, we would like to construct conormal waves propagating along the null geodesic $\gamma_{p, w}(\mR_+)$ for $(p, w) \in \LUmi$ and detect new singularities on $U_+$.
To avoid new singularities produced in the reconstructed region $W_0$, we use $\tUpl$ to observe when the conormal waves enter $W_0$.
A more detailed reconstruction can be found in Section \ref{sec_scheme3}.

Note for each $p_0 \in \Spl(T_1)$, we can find $p_\pm$, $U_\pm$, and $\tUpl$ as above.
We perform the same reconstruction for each $p_0$ and this recovers a connected new region given by the union of diamond sets $I(p_-, p_+)$.
Then we choose
\[
\begin{split}
    T_{2}  = \sup\{T \in (0,1): &\text{ for each $p_0 \in \Smi(T_1)$, we have $\overline{W} \subseteq \lB(p_0, \delta)$},\\
    &\quad \quad \quad \quad \text{ where $\overline{W}$ are constructed for $p_0$ as above}\}.
\end{split}
\]
Note the boundary of this new region includes
\[
\mathrm{cl}(\Smi(T_1, T_1 + T_2)) \cup (\Ppl(T_1) \cap \Jmi(\Spl(T_2)),
\]
as they are contained in the normal null geodesics congruences starting from $\Smi(T_1)$.
Although this new region might not fully cover
\[
I(\Smi(T_1 + T_2), \Spl(T_2)) \setminus I^+(\Smi(T_1)),
\]
we can follow a similar argument as Step 1.
More explicitly, for $T_2'$ to be specified later, one can consider the compact set
\[
Y_0 \coloneqq \partial \Jpl(\Smi(T_1 + T_2')) \cap \partial \Jmi(\Spl(T_2')).
\]
If $Y_0 \neq \emptyset$, we pick arbitrary $y_0 \in Y_0$ and consider the same construction as in Step 1.
This enables us to reconstruct a new region, whose boundary includes
\[
\big(\partial \Jpl(\Smi(T_1 + T_2')) \cap \Jmi(\Spl(T_2'))\big) \cup  \big(\Jpl(\Smi(T_1 + T_2')) \cap \partial \Jmi(\Spl(T_2')) \big).
\]
By choosing the smaller one of $T_2, T_2'$ as $T_2$, we reconstruct $I(\Smi(T_1 + T_2), \Spl(T_2)) \setminus I^+(\Smi(T_1))$.

%% file: step4.tex
\subsubsection{Step 4.}
In the following, assume we have reconstructed
the region given by
\[
I(\Smi(T_1 + T_2), \Spl(T_2)) \cup I^+(\Smi(T_1))
\]
for some $T_1, T_2 > 0$.
The goal to find $T_3 > 0$ such that we can reconstruct
\[
I(S_-(T_1 + T_3), S_+(T_2+T_3)) \cup I^+(\Smi(T_1))
\]
for some $0 < T_3 \leq T_2$, which includes the pink region given by
\[
\big(I^-(S_+(T_2+T_3)) \setminus I^-(S_+(T_2))\big) \cap
\big(I^+(S_1(T_1+T_3)) \setminus I^+(S_1(T_1))\big).
\]
For this purpose, we want to cover the pink region by a sequence of small diamond sets as before, such that each diamond set is contained in a small neighborhood without cut points.

First, consider the intersection of two achronal Lipschitz boundaries
\[
R_- \coloneqq \partial \Jmi(\Spl(T_2)) \cap \partial \Jpl(\Smi(T_1 + T_3)).
\]
We pick an arbitrary $\pmi \in R_-$.
With $\pmi \in \partial \Jpl(\Smi(T_1 + T_3))$,
there exists a future pointing normal null geodesic $\gamma_-(\mR_+)$ connecting $\pmi$ with a point $\pmi' \in \Smi(T_1 + T_2)$ on or before the first conjugate point and the first focal point, by Lemma \ref{lm_achronalboundary}.
Similarly, with $\pmi \in \partial \Jmi(\Spl(T_2))$,
there exists a future pointing normal null geodesic $\gamma_+(\mR_+)$ connecting $\pmi$ with a point $\tp_- \in \Spl(T_2)$ on or before the first conjugate point and the first focal point.
\begin{figure}[h]
    \centering
    \includegraphics[width=0.5\linewidth]{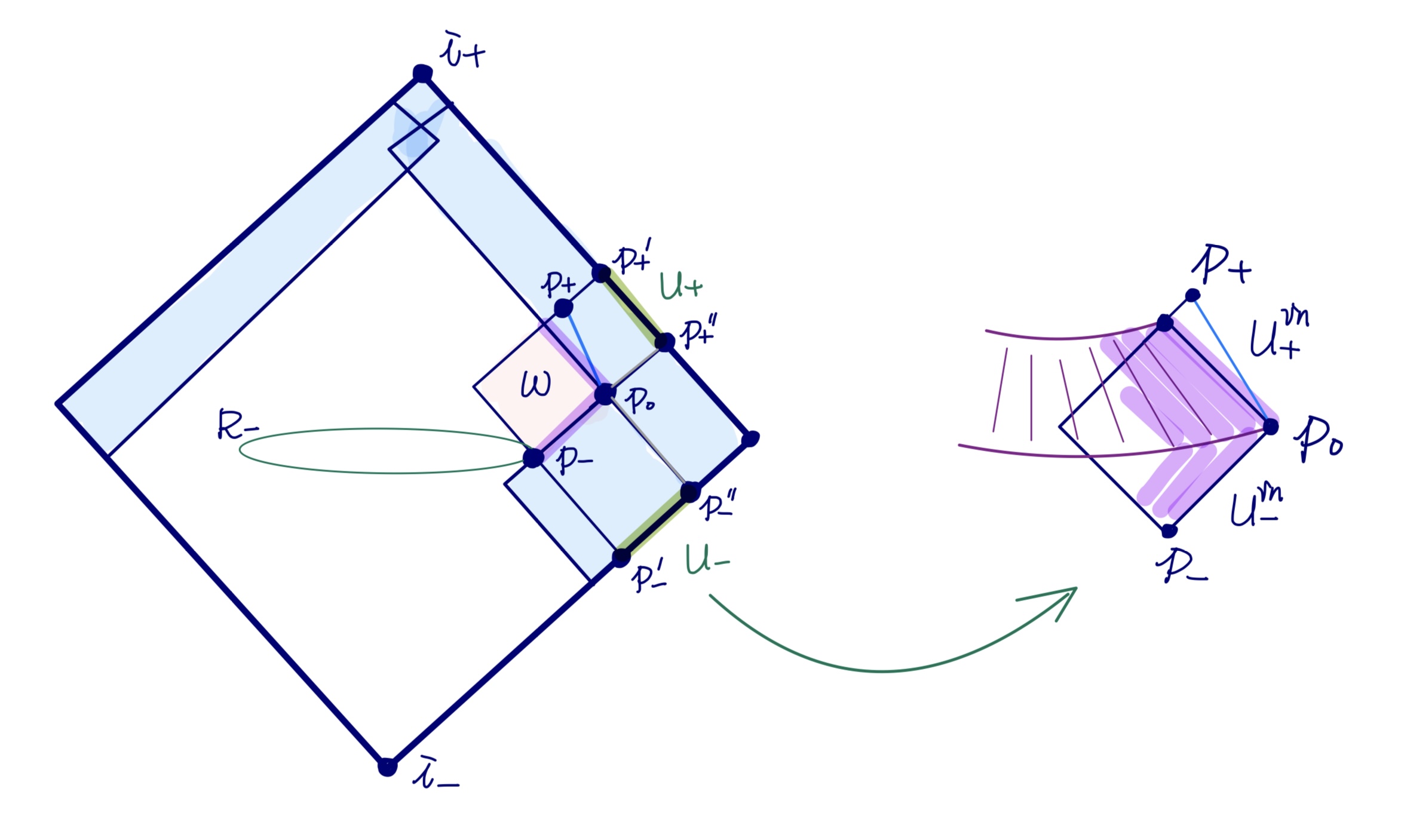}
    \caption{Step 4.}
\end{figure}
Moreover,
the null geodesic $\gamma_+(\mR_+)$ intersects
$
\Ppl(T_1) = \partial \Jpl(\Smi(T_1)) \cap M
$ exactly at one point
\[
p_0 \in \gamma_+(\mR_+) \cap  \Ppl(T_1).
\]
Note $p_0$ is connected with $\pmi'' \in \Spl(T_2)$ by the same $\gamma_+(\varsm)$ and is before the first conjugate point and the first focal point.
Then by Lemma \ref{lm_achronalboundary} again, we have
$
p_0 \in \Pmi(T_2) \coloneqq \partial \Jmi(\Spl(T_2)) \cap M
$ as well.
Recall $p_0 \in \Ppl(T_1)$.
There exists a future pointing normal null geodesic $\gamma_0(\mR_+)$ connecting $p_0$ with a point $p_0'' \in \Smi(T_1)$ on or before the first conjugate point and the first focal point.
We consider $\gamma_0(\mR_+)$ and its unique intersection with the achronal boundary
\[
\Pmi(T_2 + T_3) \coloneqq \partial \Jmi(\Spl(T_2 + T_3)) \cap M
\]
at the point $p_+$.
We emphasize such $p_+$ may not locate on the boundary of $\Ppl(T_1)$
but it is always contained in $\Jpl(\Smi(T_1))$.
Later we can repeat the same procedure by starting from an arbitrary $p_+ \in R_+$, see (\ref{def_Rpl}).
Now for $p_+$, similarly, we can find $\ppl' \in \Spl(T_2 + T_3)$ such that $\ppl' < \ppl$ along a future pointing normal null geodesic on or before the first conjugate point or focal point.

Further, we consider the unique
$\mupl_a:[-\varsm_a^+, 0] \rightarrow \bar{S}_+$ passing through $\ppl'$.
There exists $0< \varsm_1 < \varsm_2 < \varsm_a^+$ such that
\[
\ppl' = \mupl_a(-\varsm_1) \Spl(T_2 + T_3) \quad \text{ and } \quad
\ppl'' = \mupl_a(-\varsm_2) \in \Spl(T_2).
\]
There exists a unique null geodesic $\mumn_b:[0, \varsm_b^-] \rightarrow \bar{S}_-$ passing through $\pmi'$ with
$0< \varsm_3 < \varsm_4 < \varsm_b^-$ such that
\[
\pmi' = \mumn_b(\varsm_3)\in \Smi(T_1 + T_3) \quad \text{ and }\quad
\pmi'' = \mumn_b(\varsm_4)\in \Smi(T_1).
\]
As before, we might have  $\ppl'' \neq \tp_-$ and $\pmi'' \neq p_0''$.
Moreover, it is not necessarily true that $p_0 < \ppl''$ or  $p_0 > \pmi''$.
However, with $p_- < p_0 < p_+$ and $p_+ \in \Jpl(\Smi(T_1))$, our construction works.

Moreover, we observe one can choose $T_3$ such that $\pmi$ and $\ppl$ are arbitrarily close to $p_0$.
Thus, there exists $T_3 > 0$ such that $\pmi$ and $\ppl$ are contained in a given small neighborhood of $p_0$.
Then we can find  $T_3 > 0$ such that
\[
J(p_-, p_+) \subseteq \lB(p_0, \delta)
\]
has no cut points along any null geodesic segments as before.
Then let $U_+ \subseteq \Spl(T_2, T_2+ T_3)$ be a small open neighborhood of the null geodesic segment $\mupl_a([-\varsm_2, -\varsm_1])$ from $\ppl''$ to $\ppl'$.
Let $U_- \subseteq \Smi(T_1, T_1 + T_3)$ be that of the null geodesic segment $\mumn_b([\varsm_3, \varsm_4])$ from $\pmi'$ to $\pmi''$.

In addition,
with the achronal boundary $\Ppl(T_1)$ and $\Pmi(T_2)$,
the diamond set $I(U_-, U_+)$ can be partitioned as
\[
I(U_-, U_+) = W_0 \cup \tUmi \cup W \cup \tUpl \cup W_1,
\]
where we denote by
\[
\begin{split}
W &= I(U_-, U_+) \setminus (\Jpl(\Smi(T_1)) \cup \Jmi(\Spl(T_2))),\\
\tUmi &= I(U_-, U_+) \cap \Pmi(T_2), \quad \tUpl = I(U_-, U_+) \cap \Ppl(T_1),\\
W_0 &= I(U_-, U_+) \cap I^-(\Spl(T_2)), \quad W_1 = I(U_-, U_+) \cap I^+(\Smi(T_1)).
\end{split}
\]
Note such $W$ is a precompact set near $p_0$ and it can be contained in a small neighborhood of $p_0$ when $T_3$ is sufficient small.
By choosing small $T_3 > 0$, we can expect
\[
\overline{W}  \subseteq \lB(p_0, \delta)
\]
has no cut points along any null geodesic segments.

%

In this case, we would like to combine the ideas for Step 2 and 3.
On the one hand, we would like to construct conormal distributions propagating along the null geodesic $\gamma_{\tp, \tw}(\mR_+)$ for $(\tp, \tw) \in L^+M$ with $\tp \in \tUmi$, by sending proper Lagrangian distributions singular near some $(p, w) \in \LUmi$.
On the other hand, we would like to detect new singularities on $U_+$ and to avoid new singularities produced in $W_1$ by using the information on $\tUpl$.
A more detailed reconstruction can be found in Section \ref{sec_scheme4}.

Note for each $p_- \in R_-$, we can find $p_\pm$, $U_\pm$, $\tUmi$, and $\tUpl$ as above.
We perform the same reconstruction for each $p_-$ and this recovers a connected new region given by the union of diamond sets $I(p_-, p_+)$.
We choose
\[
\begin{split}
    T_{3}  = \sup\{T \in (0,1): &\text{ for each $\pmi \in R_-$, we have $\overline{W} \subseteq \lB(p_0, \delta)$},\\
    &\quad \quad \quad \quad \text{ where $p_0$ and $\overline{W}$ are constructed for $p_-$ as above}\}.
\end{split}
\]
Note the boundary of this new region includes
\[
\Ppl(T_1) \cap \Jmi(\Spl(T_2)),
\]
as they are contained in the normal null geodesics congruences starting from $R_-$.
Although this new region might not fully cover
\[
I(S_-(T_1 + T_3), S_+(T_2+T_3)) \cup I^+(\Smi(T_1)),
\]
we can follow a similar argument as Step 1.
More explicitly, for $T_3'$ to be specified later,
we consider
\beq\label{def_Rpl}
R_+ \coloneqq \partial \Jmi(\Spl(T_2 + T_3)) \cap \partial \Jpl(\Smi(T_1)).
\eeq
and perform a similar construction for each $\ppl \in R_+$.
Next, for $T_3''$, we consider the compact set
\[
Y_0 \coloneqq \partial \Jpl(\Smi(T_1 + T_3'')) \cap \partial \Jmi(\Spl(T_2 + T_3'')).
\]
This enables us to reconstruct a new region, whose boundary includes the boundary of the desired region.
By choosing the smaller one of $T_3, T_3', T_3''$ as $T_3$, we reconstruct $I(S_-(T_1 + T_3), S_+(T_2+T_3)) \cup I^+(\Smi(T_1))$.

Then we repeat these steps until we reconstruct the region $I(T_0)$ for some fixed $T_0 \in (0,1)$.
We emphasize that in each steps, the parameters, such as $T_1, T_2$, that we choose above do not get too small, due to a compactness argument.
This allows us to reconstruct the desired region in finitely many steps.
Moreover, the choice of such $T_1, T_2$ may not be clear during the reconstruction but one can use the following strategy.
First, we consider some $\eps_0 > 0$ such that by choosing $T_1 = \eps_0$ in Step 1, we have all scattering light observation sets observed in $U_+$ are smooth.
Although $\eps_0$ might be much bigger than the actually step size, yet we could use it to recover the spacetime in this $\eps_0$-size domain, which might be the wrong one.
Next, our task is to determine if this reconstruction does produce the actual spacetime.
To this end, consider the value $\eps_N=\eps_0 / N$, for $N=1,2,3,\ldots$.
One can try to reconstruct the original $\eps_0$-size domain using a layer-stripping procedure proposed above by reconstructing spacetime region of size $\eps_N$.
Note that this may fail for small $N$, if $\eps_N$ is too big. Indeed, some scattering observation sets in some $\eps_N$-diamond regions may end up being singular because of conjugate points. In this case, we know $\eps_0$ was too big.
On the other hand, if this layer-stripping succeeds for all $N$ large enough (we denote the $\eps_0$-size spacetime region reconstructed using $\eps_N$ by $D_N$, then we know that if $N$ is big enough so that $\eps_N<\eps'$, then $D_N$ correctly reconstructs the $\eps_0$-piece of our spacetime.
Thus, the idea is to pick the smallest $N_0$ so that $D_{N_0} = D_N$ for all $N\geq N_0$. Such an $N_0$ does exist.
Finally, take the largest $\eps_0$ so that $N_0=1$ works.
We can then reconstruct the $\eps_0$-size region, and be certain that we have reconstructed the correct smooth and conformal structure.

%% file: simple_scattering_light.tex
\section{The scattering light observation Sets}\label{sec_SLOS}
In each step of Section \ref{sec_layer},
we need to reconstruct $g$ from measurements in an open set $U_+ \subseteq \Spl$ under the assumption that there are no cut points in the region of interest.
For this purpose, we consider the following model and prove the reconstruction in a slightly more general setting.
Let $M$ be a globally hyperbolic subset equipped with a Lorentzian metric $g$.
Now let $U_+ \subseteq \Spl$
and $U_- \subseteq \Smi$ be open subsets.
We consider
\[
W = I(U_-, U_+),
\]
which is an precompact open set, where we would like to reconstruct the topological, differential, and conformal structure using $U_+$.
In the following, for convenience we denote $U_+$ by $U$.


For $q \in M$ and $v \in L_q M$,
we denote by $\gamma_{q, v}: (\varsigma_0, \varsigma_1) \rightarrow M$ the inextendible null geodesic that starts from $q$ in the lightlike direction $v$.
For an open subset $U \subseteq \Spl$,
we define
\beq\label{def_CU}
\begin{split}
    C_U(q) \coloneqq \{(p, w) \in L^+U: &\text{ $p = \gamma_{q, v}(1)$, $w = \dot{\gamma}_{q, v}(1)$ for some $v \in L_q^+M$} \},
\end{split}
\eeq
as the union of null geodesic flow from a point $q \in M$ restricted to $U$.
Note that $C_U(q)$ is always a smooth submanifold in $T_UM$, even with cut points or conjugate points.
We consider the following subsets
\beq
\begin{split}\label{def_CUear}
    \Cear_U(q) \coloneqq \{(p, w) \in L^+U: &\text{ $p = \gamma_{q, v}(1)$, $w = \dot{\gamma}_{q, v}(1)$},\\
    &\text{for some $v \in L_q^+ M$ with $1 \leq \rho(q, v)$}\},
\end{split}
\eeq
and its regular part
\beq\label{def_CUreg}
\begin{split}
    \CUreg(q) \coloneqq \{(p, w) \in L^+U: &\text{ $p = \gamma_{q, v}(1)$, $w = \dot{\gamma}_{q, v}(1)$},\\
    &\text{for some $v \in L_q^+ M$ with $1 < \rho(q, v)$}\}.
\end{split}
\eeq
We call $\CUear(q)$ the \textit{earliest scattering direction set}
and $\CUreg(q)$ the \textit{regular scattering direction set}
of $q$ within $U$.
Let $\pi: TU \rightarrow U$ be the canonical projection and
we define 
\beq
\begin{split}\label{def_LU}
\LUear(q) \coloneqq \pi(\CUear(q)), \quad
\LUreg(q) \coloneqq \pi(\CUreg(q))
\end{split}
\eeq
as the \textit{earliest} or \textit{regular scattering light observation set}
of $q$ within $U$.
We prove in Proposition \ref{pp_LUreg} that we can reconstruct $\CUreg(q)$ from knowing $\CUear(q)$ for all $q \in W$.
We would like to follow the idea in \cite{Hintz2017} to prove the collection of such earliest scattering direction sets
\beq
\begin{split}
    \SUear(W) = \{\CUear(q): q \in W\}
\end{split}
\eeq
determines the topological, differential, and conformal structure of $W$.
This is a simplified but localized version of the result proved in \cite{MScfinal}.

\begin{thm}\label{thm_SLOS}
    For $j = 1, 2$, let $(M^{(j)}, g^{(j)})$ be two globally hyperbolic subsets satisfying Assumption \ref{assump_Mg}. Let $U^{(j)}$, $W^{(j)}$ be defined as above.
    Assume for each $q \in W^{(j)}$, we have
        \[
        \CUjreg(q) \neq \emptyset, \quad j = 1,2.
        \]
    Suppose there exists a conformal diffeomorphism $\Phi: U^{(1)} \rightarrow U^{(2)}$ such that
    \[
    \SUone(W_2) = \{\Phi_*(C): C \in \SUtwo(W_1)\}.
    \]
    Then there is a conformal diffeomorphism $\Psi: (W^{(1)}, g^{(1)}|_{W^{(1)}}) \rightarrow (W^{(2)}, g^{(2)}|_{W^{(2)}})$.
\end{thm}
\begin{remark}
If there are no cut points along any null geodesic segments in $\overline{W}$,
then we have
\[
C_U(q) = \CUear(q) = \CUreg(q),
\]
which implies the result that we need for Step 1 and 2 in Section \ref{sec_layer}.
The other steps of Section \ref{sec_layer} uses the reconstruction from $\Creg_U(W)$.
\end{remark}
\begin{remark}
In this theorem, we assume there exists a conformal diffeomorphism $\Phi: U^{(1)} \rightarrow U^{(2)}$ such that
the two collections of the earliest light observation sets for $W^{(j)}$ are related by $\Phi$.
For the inverse problems we consider, this assumption is satisfied if we have
$
\lN^{(1)}(u_-) = \lN^{(2)}(u_-)
$ 
for each $u_- \in D(\rho)$.
Indeed, by considering the first order linearization of $\lN^{(j)}$ for $j = 1,2$, one obtain the linear scattering operator.
Recall in Section \ref{sec_receding}, we find local coordinates such that the past null infinity is locally given by $x_1 = 0$ and the metric is given in certain forms.
Then one can use a similar idea as in \cite{Stefanov2018} to prove this linear scattering operator determines the lens relation and moreover the jets of the metric on the boundary up to conformal diffeomorphisms.
\end{remark}
\subsection{Null hypersurfaces and cut points}
\subsubsection{The structure of null hypersurfaces}\label{subsec_pM}
In the following, we focus on the smooth part of $\pM$, that is, the future or past null infinity $S_\pm$.
For $p \in \partial M$, we use
$T_p \partial M$ to denote the set of tangent vectors in $\partial M$ at $p$.
We define the normal vector space
\[
T_p^\perp \partial M = \{v \in T_pM: {g}(v, w) = 0 \text{ for any } w \in T_p \partial M\}.
\]
By our assumption, $\pM = \bar{S}_+ \cup \bar{S}_-$, where $S_+$ and $S_-$ are null hypersurfaces, in the sense that the restriction of $g$ to $S_\pm$ is degenerate.
Then $T_p \partial M$ and $T_p^\perp \partial M$ have a nontrivial intersection and their direct sum does not equal to $T_pM$, which is not the case for timelike or spacelike submanifolds.
Instead, one can introduce the so called \textit{radical space} of $T_p \partial M$ given by
\[
\mathrm{Rad}(T_p \partial M) = \{v \in T_p \partial M: {g}(v, w) = 0 \text{ for any } w \in T_p \partial M\} = T_p \partial M \cap T_p^\perp \partial M.
\]
for more details see \cite{duggal2011differential}.
In our case $\mathrm{Rad}\ T_p \partial M = T_p^\perp \partial M$ is a one-dimensional subspace.
Furthermore, one can define the radical distribution $\mathrm{Rad}\ T \partial M$ such that
\[
\mathrm{Rad}(T \pM): p \mapsto \mathrm{Rad}(T_p \partial M).
\]
Moreover, we can define the complementary vector space $\mathrm{S}(T_p\pM)$ of $\mathrm{Rad}(T_p \pM)$ in the sense that
\beq
\begin{split}\label{eq_TpM}
    \mathrm{S}(T_p\pM) + \mathrm{Rad}(T_p \pM) = T_p \pM, \quad
    \mathrm{S}(T_p\pM) \cap \mathrm{Rad}(T_p \pM) = \{0\}.
\end{split}
\eeq
In this work, we call such $\mathrm{S}(T_p\pM)$ a \textit{screen space} of $T_p\pM$.
Note that with the definition of $\mathrm{Rad}(T_p \pM)$, the direct sum (\ref{eq_TpM}) is actually an orthogonal one.
This implies $\mathrm{S}(T_p\pM)$ is a nondegenerate subspace, in the sense that $g$ restricted to $\mathrm{S}(T_p\pM)$ is nondegenerate.
Actually in our case the restriction of $g$ is a Riemannian metric and therefore $\mathrm{S}(T_p\pM)$ is spacelike.
In addition, the choice of a screen space is not unique.
Based on these definitions, we have the following lemma.
\begin{lm}\label{lm_screenandrad}
    Suppose $\dim(\mathrm{Rad}(T_p \pM)) = 1$.
    Let $V \subseteq T_p \pM$ be a vector space with codimension one.
    If $V \cap \mathrm{Rad}(T_p \pM) = \{0\}$, then $V$ is a screen space of $T_p \pM$.
\end{lm}

In the following, we show there is a one-to-one correspondence between a screen space and a future pointing lightlike vector that is outwardly transversal to $\pM$.
For this purpose, we consider the set $L_pM$ of all lightlike vectors in $T_pM$ for $p \in \pM$.
{Let $\varpi$ be a future-pointing timelike vector field}, for example, $\varphi = \dif T$.
We define the set of future or past pointing lightlike vectors as
\[\begin{split}
    L^\pm_pM = & \{v \in T_pM: \pm g(v, \varpi) < 0\}. 
\end{split}\]
For $p \in \pM$, one can also introduce the following subspaces
\beq
\begin{split}\label{def_outwardTM}
    T^\pm_{p} M = \{v \in T_{p} M: \pm g(v, \nu) < 0\}
\end{split}
\eeq
of outward or inward vectors that are transversal to the boundary, where $\nu$ is a future pointing normal vector to $\pM$.
Indeed, a light vector $v \in L_p M$ is outward to $\pM$ if and only if $v$ is future pointing and not tangential to $\pM$, i.e., $g(v, \nu) < 0$.

The following lemma for lightlike boundary $\partial M$ is an analog to \cite[Lemma 2.5]{Hintz2017} for timelike boundaries.
It is proved in \cite[Lemma 2.1.9]{MScfinal}.
\begin{lm}\label{lm_screen}
    Let $(M, {g})$ be as above. For $p\in \pM$,
    we denote by
    \[\begin{split}
        \mathcal{S} &\coloneqq \{\text{$\mathrm{S}(T_p \pM) \subseteq T_p\pM$ is a screen space}\}, \\
        \mathcal{V} &\coloneqq \{\mathbb{R}_+ v \subseteq T_pM: v \in L^+_p \M \cap T^+_{p} M\}.
    \end{split}\]
    the set of all screen spaces of $T_p\pM$ and the set of all future pointing outward lightlike rays in $T_pM$.
    Then there exists an isomorphism $\phi: \mathcal{S} \rightarrow \mathcal{V}$, given by mapping $\mathrm{S}(T_p \pM)$ to the unique future pointing outward lightlike ray
    $\mathbb{R}_+ v$ contained in $\mathrm{S}(T_p \pM)^\perp$.
\end{lm}
\subsubsection{Null geodesics and cut points}\label{subsec_cut}
Recall for $(q, v) \in L\intM$, we denote by $\gamma_{q, v}$ the unique null geodesic starting from $q$ in the direction $v$.
\begin{lm}[{\cite[Lemma 2.1.1]{MScfinal}}]\label{lm_nulltransversal}
    If $\gamma_{q,v}: (-\varsigma_o, 0] \rightarrow M$ is a future pointing null geodesic segment with $p = \gamma_{q,v}(1) \in \partial M$ and $\gamma_{q,v}(\varsigma) \in \intM$ for $\varsigma < 1$,
    then $\dot{\gamma}_{q,v}(1) \in L^+_{p}M \cap T^+_{p} M$.
\end{lm}
The following lemma is a variant of \cite[Proposition 2.10]{Hintz2017}.
\begin{lm}\label{lm_smoothexp}
    Let $(q, v) \in L^+ \intM$ such that $\gamma_{q,v}(\varsigma)$ hits $\pM$ at $\varsigma = {\varsigma_{q,v}} > 0$.
    Then ${\varsigma_{q,v}}$ depends on $(q, v)$ smoothly and so does the point $p = \gamma_{q,v}({\varsigma_{q,v}})$.

    More explicitly, for fixed $(q_0, v_0) \in L^+ \intM$, there exists an open neighborhood $N_0$ of $(q_0, v_0)$ and a smooth function $d_0: N_0 \rightarrow \mR$ such that ${\varsigma_{q,v}} = d_0(q, v)$.
\end{lm}
\begin{proof}
    We follow the proof in \cite[Proposition 2.10]{Hintz2017}.
    Let $R_c: L\intM \rightarrow L\intM$ be the map acting by dilation in the fibers such that $R_c(q, v) = (q, c v)$, for $c \in \mR$.
    Let $x \in C^\infty(M)$ be the boundary defining function locally near $p \in \pM$.
    We consider the exponential map $\exp: L\intM \rightarrow \intM$ written as
    $\exp(q, v)= \gamma_{q,v}(1)$.
    This map can be regarded as the null geodesic flow on $LM$ projected to the manifold.
    For fixed $(q_0, v_0) \in L^+ \intM$, for simplification we write $\varsigma_0 = \varsigma(q_0,v_0)$.
    Then we define the set
    \[
    Z_0 = F^{-1}(0), \quad \text{where } F  = x \circ \exp \circ R_{\varsigma_0}.
    \]
    Certainly we have $(q_0, v_0) \in Z_0$. Moreover, we claim, in a neighborhood of $(q_0, v_0)$, this set is a smooth submanifold of $L\intM$ of codimension one, which is transversal to $\mR_+ v_0$.
    Indeed, we compute the differential of $F$ at $(q_0, v_0)$ to have
    \[
    \dif F(q_0, v_0) = (\frac{\partial x_{q, v}}{\partial q}(\varsigma_0), \varsigma_0\frac{\partial x_{q, v}}{\partial v}(\varsigma_0)) , \quad \text{where  we write } x_{q, v}(\varsigma) = x\circ \gamma_{q,v}(\varsigma).
    \]
    By Lemma \ref{lm_nulltransversal}, we have
    $\frac{\partial x_{q, v}}{\partial \varsigma}(\varsigma_0) \neq 0$, as $\dot{\gamma}_{q, v}(\varsm_0) \notin T_p \pM$.
    Note that $\frac{\partial x_{q, v}}{\partial \varsigma}(\varsigma_0)$ is in the span of $\frac{\partial x_{q, v}}{\partial v}(\varsigma_0)$ and therefore the smooth map $F$ has a nonzero differential at $(q_0, v_0)$.
    Then the claim comes from the implicit function theorem and $Z_0$ is transversal to $\mR_+v_0$ at $q_0$, as the null geodesic $\gamma_{q_0,v_0}(\mR_+)$ hits $\partial M$ transversally by Lemma \ref{lm_nulltransversal}.

    Now we choose a small neighborhood $N_0 \subseteq L^+\intM$ of $(q_0, v_0)$ such that
    \[
    N_0 \subseteq \textstyle \bigcup_{c \in (1-\ep, 1+\ep)} R_c (Z_0)
    \]
    for some $\ep > 0$.
    Additionally, we choose $N_0$ sufficiently small such that
    $\bar{N}_0 \cap Z_0$ is a smooth connected submanifold traversal to all dilation orbits intersecting $N_0$.
    Then for $(q, v) \in N_0$,
    we define $d_0 \in C^\infty(N_0)$ by
    \[
    d_0(q, v) = r\varsigma_0,
    \]
    where $r$ is the unique real number such that $R_r(q, v) \in Z_0$.
    Note that $d_0(q_0, v_0) = \varsigma_0$.
    The uniqueness of $r$ and its smooth dependence on $(q,v)$ comes from the transversal intersection of $Z_0$ and all dilation orbits.
    Thus, for any $(q, v) \in N_0$, we have $x \circ \exp \circ R_{d_0(q,v)}(q, v) = 0$, which implies $\gamma_{q,v}(d_0(q,v)) \in \pM$ and thus ${\varsigma_{q,v}} = d_0(q,v)$ is smooth.

    Next, to show the boundary point $p  = \gamma_{q,v}({\varsigma_{q,v}})$ depends on $(q, v)$ smoothly, we consider {its extension across $\partial M$}.
    Indeed, as $\gamma_{q,v}(\varsigma)$ hits $\partial M$ transversally, we can extend it smoothly to $\varsigma \in [0, {\varsigma_{q,v}} + \ep)$ in $(\tM, \tghat)$, for some $\ep> 0$.
    Then $\gamma_{q,v}(\varsigma)$ is the solution to the geodesic equation in $(\tM, \tghat)$, for $\varsigma \in [0, {\varsigma_{q,v}} + \ep)$, which smoothly depends on the initial conditions $\gamma_{q,v}(0) = q$ and $\dot{\gamma}_{q,v}(0) = v$.
    With ${\varsigma_{q,v}}$ smooth, we have desired result.
\end{proof}

\subsubsection{The Cut Locus Function}
We recall the definition of cut points  and the cut locus function in Section \ref{sec_prelim}.
Recall the following lemmas about cut points.
%
\begin{lm}[{\cite[Proposition 10.46]{Neill1983}}]\label{lm_shortcut}
    If there is a future pointing causal curve from $q$ to $y$ that is not a null pregeodesic, then there is a timelike curve from $q$ to $y$ arbitrarily close to this curve and therefore $q \ll y$ (equivalently $\tau(q,y)> 0$).
\end{lm}
\begin{lm}[{\cite[Lemma 9.13]{Beem2017}}]\label{lm_cutunique}
    In a causal spacetime, if there are two future directed null geodesic segments from $q$ to $y$, then $y$ comes on or after the null cut point of $q$ on each of the two segments.
\end{lm}
In other words, if $y$ is before the first cut point along a null geodesic segment from $y$ to $q$, then this null geodesic segment is the only pregeodesic from $q$ to $y$.
We have the following lemma.
\begin{lm}\label{lm_LUreg}
    Let $q \in \intM$ and $y \in \pM$.
    Then $y \in \LUreg(q)$ if and only if there is a future pointing null geodesic segment from $q$ to $y$ and $y$ is before the first cut point.
\end{lm}
The following lemma is inspired by \cite[Lemma 6.7]{feizmohammadi2021inverse}.
\begin{lm}\label{lm_cutlocus}
    Let $(q_j,v_j) \rightarrow (q,v)$ in $L^+\intM$.
    Suppose $\gamma_{q_j, v_j}(1) \in \pM$ is on or before the first cut point.
    Then $\gamma_{q, v}(\varsigma)$ hits $\pM$ at $\varsigma = 1$ on or before the first cut point.
\end{lm}
\begin{proof}
    Since $\gamma_{q, v}(\varsigma)$ smoothly depends on $(q,v)$, with $\pM$ closed, we have $\gamma_{q, v}(1) \in \pM$.
    It remains to prove $1 \leq \rho(q,v)$.
    We assume for contradiction that $1> \rho(q,v)$,
    which implies $\tau(q, \gamma_{q, v}(1)) > 0$.
    As $\gamma_{q_j, v_j}(1)$ is on or before the first cut point, we have $1 \leq \rho(q_j, v_j)$, which implies $\tau(q_j, \gamma_{q_j, v_j}(1)) = 0$.
    Since $\tau$ is continuous on $M \times M$, we must have
    \[
    0< \tau(q, \gamma_{q, v}(1)) = \lim_{j \rightarrow +\infty}\tau(q_j, \gamma_{q_j, v_j}(1)) = 0,
    \]
    which leads to contradiction.
\end{proof}
We need the following auxiliary lemma.
\begin{lm}\label{lm_cutconverge}
    Let $q_j \rightarrow q$ in $W$.
    Let $p_j \rightarrow p$ in $U$, where
    $p_j = \gamma_{q_j, v_j}(1)$ with $v_j \in L_{q_j}^+\intM$.
    Then we can find a subsequence of $(q_j, v_j)$ and $v \in L^+_q \intM$ such that
$(q_{k_j}, v_{k_j}) \rightarrow (q, v)$ in $L^+ \intM$ with
$p  = \gamma_{q,v}(1)$.
\end{lm}
\begin{proof}
Consider a smooth Riemannian metric $g^+$ on $M$.
Let $w_{j}  = v_{j}/\|v_{j}\|_{g^+}$, which correspond to the direction of each $v_j$.
For large $j$, we can assume $(q_{j}, w_{j})$ is contained in a compact set of $L^+\intM$.
Then by passing through a subsequence, we can assume $(q_j, w_j) \rightarrow (q, w)$ in $L\intM$.

By Lemma \ref{lm_smoothexp}, there exists a small neighborhood $N_0 \subseteq L^+\intM$ of $(q,w)$ and a smooth function $d_0: N_0 \rightarrow \mR_+$ such that for any $(q', w') \in N_0$, one has  $\gamma_{q', w'}(d_0(q', w')) \in \pM$.
%
Note that $\{d_0(q_{j}, w_{j})\}$ is contained in a compact set of $\mR_+$ for large $j$, as $d_0$ is smooth and we can assume $(q_{j}, w_{j})$ is contained in a compact subset of $N_0$.
Thus, we set $v = d_0(q, w) w$ and use Lemma \ref{lm_smoothexp} to have
\[
p = \lim_{j \rightarrow \infty} p_{j} = \lim_{j \rightarrow \infty} \gamma_{q_j, w_j}(d_0(q_{j}, w_{j})) =  \gamma_{q, w}(d_0(q, w)) = \gamma_{q, v}(1).
\]
With $v_j = d_0(q_j, w_j)w_j$, we have $(q_j, v_j) \rightarrow (q, v)$ in $N_0$.
\end{proof}
\begin{lm}\label{lm_cutneighborhood}
Let $(q_0, v_0) \in L^+ \intM$ such that $\gamma_{q_0,v_0}(\varsigma)$ hits $\pM$ at $\varsigma  = \varsigma_0$ before the first cut point.
Then there exists an open neighborhood $N_0 \subseteq L^+ \intM$ of $(q_0, v_0)$ such that for any $(q, v) \in N_0$,
one has $\gamma_{q,v}(\varsigma)$ hits $\pM$ at $\varsigma = \varsigma_{q,v}$ before the first cut point.
\end{lm}
\begin{proof}
As $\gamma_{q,v}(\varsigma_{0})$ is before the first cut point, we have $\varsigma_0 < \rho(q_0, v_0)$.
Then there exists $\ep > 0$ such that $\varsigma_0 + \ep< \rho(q_0, v_0)$.
Since $\rho$ is lower semi-continuous, there exists a small open neighborhood $N_0$ of $(q_0, v_0)$ such that
\[
\varsigma_0 + \ep< \rho(q, v), \quad \text{ for any }(q, v) \in N_0.
\]
By choosing sufficiently small $N_0$, using Lemma \ref{lm_smoothexp}, we can assume $\varsigma_{q,v}$ depends on $(q,v)$ smoothly such that
$\varsigma_{q,v}< \varsigma_0 + \ep$.
Thus, we have $\varsigma_{q,v} < \rho(q, v)$, which implies $\gamma_{q,v}(\varsigma_{q,v}) \in \pM$ is before the first cut point.
\end{proof}
By \cite[Theorem 9.15]{Beem2017}, the first null cut point comes at or before the first future conjugate point in a globally hyperbolic spacetime.
Then the following lemma applies if we assume $p$ is before the first cut point along the null geodesic $\gamma_{q,v}$.
\begin{lm}\label{lm_explocal}
Let $(q,v) \in L\intM$ and $p = \exp_q(v) \in \pM$.
If $(q, v)$ and $p$ are not conjugate, then
\begin{itemize}
    \item[(1)] There exists an open neighborhood $N_v \subseteq L_q M$ of $v$ such that the restriction of $\exp_q$ in $N_v$ is a diffeomorphism.
    Thus, $L \coloneqq \exp_q(N_v) \cap \pM$ is a smooth $1$-codimensional submanifold of $\partial M$.
    \item[(2)]
    Denoting $\gamma_{q,v}(t) = \exp_q(tv)$, one has $\dot{\gamma}_{q,v}(1) \in (T_p L)^\perp$. Moreover, this conclusion holds for any point in $L$ that is sufficiently close to $p$.
\end{itemize}
\end{lm}
\begin{proof}
The first statement comes from the implicit function theorem. We prove the second statement in the following.

Note that $p = \gamma_{q,v}(1) \in L$ and we would like to
parameterize $L$ near $p$.
For this purpose, let $v(r) \in N_v$ be a smooth family of lightlike vectors, for $r \in (-\ep, \ep)$, with $v(0) = v$ and $\dot{v}(0) \neq 0$ such that $\exp_q(v(r)) \in L \subseteq \pM$.
This can be done according to Lemma \ref{lm_smoothexp} by shrinking $N_v$ if necessary.
Further, we define
\[
p(r) \coloneqq \exp_q(v(r)) = \gamma_{q,v(r)}(1)\in L,
\]
which is a smooth function depending on $r$.
With $T_p L$ spanned by several such $\{\dot{p}(0)\}$, our goal is to prove
\[
g(\dot{p}(r), \dot{\gamma}_{q,v(r)}(1)) = 0 \quad \text{ for any } r \in (-\ep, \ep).
\]
For this purpose, we consider the family of curves given by $\gamma(r, t) \coloneqq \gamma_{q, v(r)}(t)$.
We observe that
\[
g(\dot{p}(r), \dot{\gamma}_{q,v(r)}(1))
= g(
\partial_r \gamma(r, 1) , \dot{\gamma}_{q,v(r)}(1))
= g(\partial_r \gamma(r, 1), \partial_t \gamma(r, 1))
\eqqcolon f(r, 1),
\]
where we define a function \[
f(r, t) \coloneqq g(\partial_r \gamma(r, t), \partial_t \gamma(r, t)), \quad r \in (-\ep, \ep), \ t \in [0, 1].
\]
Note that {$f$ is smooth in $(-\ep, \ep)\times (0, 1)$ and continuous up to $t_0$ and $t=1$}.
Then we compute
\[\begin{split}
    \partial_t f(r, t) & = g(D_t\partial_r \gamma(r, t), \partial_t \gamma(r, t)) + g(\partial_r \gamma(r, t), D_t\partial_t \gamma(r, t))\\
    & = g(D_r\partial_t \gamma(r, t), \partial_t \gamma(r, t))\\
    & = (1/2)\partial_r g(\partial_t \gamma(r, t), \partial_t \gamma(r, t)) = 0,
\end{split}\]
where for the second line we use the symmetric property and the fact that $\gamma(r, \cdot)$ are null geodesics.
Then it suffices to prove $g(\dot{p}(r), \dot{\gamma}_{q,v{r}}(1)) = f(r, 1) = f(r, 0) = 0$.
Indeed, we observe that $\gamma(r, 0) \equiv q$ and therefore
$\partial_r \gamma(r, 0) = 0.$
This proves the desired result.
\end{proof}
Even with $q$ and $p$ that are not conjugate, there might be two or more null geodesics connecting them, as the exponential map is only a local diffeomorphism there.
However, with Lemma \ref{lm_cutunique}, we have the following lemma.

\begin{lm}\label{lm_finitevectors}
Let $(q,v) \in L\intM$ and $p = \gamma_{q,v}(1) \in \pM$.
If $p$ is before the first cut point, then $v \in L_q \intM$ for which $\exp_q(v) = p$ is unique.
\end{lm}
The following lemma is a direct result of Lemma \ref{lm_smoothexp} and \ref{lm_finitevectors}.
\begin{lm}\label{lm_finite}
Let $p \in \LUreg(q)$, where $q \in \intM$.
Then there exists an open neighborhood $O \subseteq \pM$ of $p$ such that
$\mathcal{L}_q^+ \cap O$ is a smooth $1$-codimensional submanifold and $\LUreg(q)\cap O = \mathcal{L}_q^+ \cap O$.
\end{lm}
\begin{proof}
Suppose $(q,v) \in L\intM$ such that $p = \gamma_{q,v}(1) \in \pM$.
Since $p$ is before the first cut point, using Lemma \ref{lm_finitevectors} the choice of $v$ for which $p = \exp_q(v)$ is unique.
By Lemma \ref{lm_explocal}, there exists a small neighborhood $N_v \subseteq L_q^+ \intM$ of $v$ such that $L \coloneqq \exp_q(N_{v}) \cap \pM$ is a smooth submanifold of $\pM$ of codimension one.
For vectors in $L_q^+ \setminus N_{v}$, their images under the exponential map are away from $p$, by Lemma \ref{lm_cutconverge} and the fact that the choice of $v$ is unique.
Thus, we can find a small open neighborhood $O$ of $p$ such that $\mathcal{L}_q^+ \cap O = L$.
It remains to show any $p \in L$ is before the first cut point of $q$.
Indeed, by Lemma \ref{lm_cutneighborhood}, for $(q,v')$ sufficiently close to $(q,v)$, the null geodesic $\gamma_{q,v'}(\varsigma)$ hits $\pM$ before the first cut point.
By shrinking $O$ if necessary,
\end{proof}
Next, we prove these results are stable under small perturbation on $q$ and then $p$.
\begin{lm}\label{lm_conjugate_perturb}
Let $p \in \LUreg(q)$, where $q \in \intM$.
Then there exist open small neighborhoods $N \subseteq W$ of $q$ and $O \subseteq \pM$ of $p$, such that for any $q' \in N$, if one has $p' \in \lL_{q'}^+ \cap O$, then $p'$ is before the first cut point, i.e., $p' \in \LUreg(q')$.
\end{lm}
\begin{proof}
Suppose $(q,v) \in L\intM$ such that $p = \gamma_{q,v}(1) \in \pM$.
Assume for contradiction that such open neighborhoods $N$ of $q$ and $O$ of $p$ do not exist.
Then we can find a sequence $q_j \rightarrow q$ in $W$ and $v_j \in L_{q_j}^+\intM$, such that $p_j = \gamma_{q_j,v_j}(1) \in \pM$ converges to $p$ with $1 \geq \rho(q_j, v_j)$.
By Lemma \ref{lm_cutconverge}, by passing through a subsequence, we can assume $(q_j, v_j) \rightarrow (q,v)$ with $p = \gamma_{q,v}(1)$.
As $p$ is before the first cut point, we have $1< \rho(q, v) $.
Since $\rho$ is lower semi-continuous, for $(q', v')$ sufficiently close to $(q,v)$, one has $1< \rho(q', v')$.
This contradiction with $1 \geq \rho(q_j, v_j)$.
\end{proof}
Lemma \ref{lm_conjugate_perturb} implies the following result.
\begin{lm}\label{lm_smooth_perturb}
Let $p \in \LUreg(q)$, where $q \in \intM$.
Then there exists open neighborhoods $N \subseteq W$ of $q$ and $O \subseteq \pM$ of $p$, such that
$\mathcal{L}_{q'}^+ \cap O$ is a smooth $1$-codimensional submanifold for any $q' \in N$ and
$\LUreg(q') \cap O = \mathcal{L}_{q'}^+ \cap O$.
\end{lm}
\begin{lm}\label{lm_LUregspacelike}
    Let $q \in \intM$ and $p \in \LUreg(q)$.
    Then $\LUreg(q)$ is locally a spacelike $1$-codimensional submanifold near $p$.
\end{lm}
\begin{proof}
    Recall $\pM$ is a null hypersurface with a normal vector field $\nu \in \mathrm{Rad}(T \partial M)$.
    Then $\LUreg(q)$ is spacelike if it is not tangential to $\nu$ at any $p' \in \LUreg(q)$ near $p$.
    We write $L  = \LUreg(q)$.
    Assume for contradiction that $\nu$ is tangential to $L$ at $p_0 \in L$.
    Then we find $p_1, p_2 \in L$ such that they are connected by an integral curve of $\nu$ (a null geodesic in $\pM$).
    Suppose without loss that $p_1 < p_2$.
    Then we find a causal path from $q$ to $p_1$ and then to $p_2$, which is not a pregeodesic.
    By Lemma \ref{lm_shortcut}, we must have $q \ll p_2$, which contradicts with $p_2$ is on or before the first cut point.
\end{proof}
\begin{lm}\label{lm_uniqueLreg}
For any $q, q' \in W$, if for some open subsets $O \subseteq L^+U$ we have
\[
\emptyset \neq \CUreg(q) \cap O \subseteq \CUear(q') \cap O,
\]
then $q = q'$.
\end{lm}
\begin{proof}
A similar argument can be found in \cite[Proposition 2.2.9]{MScfinal}.
Assume for contradiction that $\emptyset \neq \CUreg(q) \cap O \subseteq \CUear(q') \cap O$ but $q \neq q'$.
By Lemma \ref{lm_finite}, for $(p_1, w_1)\in \CUreg(q) \cap O$, there exists a  distinct point $p_2$ such that  $(p_2, w_2)\in \CUreg(q) \cap O$.
By our assumptions $(p_j, w_j) \in \CUreg(q') \cap O$ for $j = 1,2$.
The definition (\ref{def_CUreg}) implies there exists $v_j \in L_q^+ M$ and $v_j' \in L_{q'}^+M$ such that
\[
(p_j, w_j) = (\gamma_{q, v_j}(1), \dot{\gamma}_{q, v_j}(1)) = (\gamma_{q', v_j'}(1), \dot{\gamma}_{q', v_j'}(1)), \quad \text{ for } j = 1,2.
\]
Now $\gamma_j \coloneq \gamma_{p_j, -w_j}(\mR_+)$ are two past pointing null geodesics starting from $p_j$ to $q$ or $q'$, for $j = 1, 2$.
Thus, we must have either $q < q'$ or $q' < q$.
In the first case, consider the geodesic segments from $q$ to $q'$ along $\gamma_1$ and from $q'$ to $p_2$ along $\gamma_2$.
By Lemma \ref{lm_shortcut}, we have $\tau(q, p_2) > 0$, which contradicts with $(p_2, w_2) \in \CUreg(q)$.
In the second case, consider the geodesic segments from $q'$ to $q$ along $\gamma_1$ and from $q$ to $p_2$ along $\gamma_2$.
By Lemma \ref{lm_shortcut} again, we have $\tau(q', p_2) > 0$, which contradicts with $(p_2, w_2) \in \CUear(q)$.

\end{proof}
\begin{proof}
We can prove it using Lemma \ref{lm_smoothexp} and the proof in Section \ref{subsection_smooth_structure}.
\end{proof}
\begin{lm}\label{lm_uniqueL}
Suppose $\CUreg(q) \neq \emptyset$ for each $q \in W$.
Then the map given by
\[\begin{split}
    \psi^{\mathrm{reg}}: W &\rightarrow \SUreg(W)\\
    q & \mapsto \CUreg(q).\nonumber
\end{split}\]
is a bijection, where we denote by $\SUreg(W) = \{\CUreg(q): q \in W\}$.
\end{lm}
\begin{proof}
By Lemma \ref{lm_uniqueLreg}, if $\CUreg(q) = \CUreg(q')$, then $q = q'$ and thus
$\psi^{\mathrm{reg}}$ is one-to-one.
It is onto by the definition of $\SUreg(W)$.
\end{proof}
As $\emptyset \neq \CUreg(q) \subseteq\CUear(q)$ for each $q \in W$, this indicates the map
\beq
\begin{split}\label{def_psi}
    \psi: W &\rightarrow \SUear(W)\\
    q & \mapsto \Cear_U(q).
\end{split}
\eeq
is a bijection as well.
\subsection{The regular scattering light observation sets}\label{subsec_LUreg}
In this part, we discuss the difference between $\Cear_U(q)$ and $\CUreg(q)$.
Moreover, we propose a way to reconstruct $\CUreg(q)$ from the knowledge of  $\SUear(W)$.
Recall from Lemma \ref{lm_finite}, we know $\LUreg(q) = \pi(\CUreg(q))$ is a smooth $1$-codimensional submanifold of $U$.
In the following, we would like to first collect all smooth parts of $\LUear(q) = \pi(\Cear_U(q))$.
We define
\beq\label{def_LUinfity}
\begin{split}
\LUsmooth(q) = \{p \in \LUear(q): &\text{ there is an open neighborhood $O$ of $p$ such that}\\
&\text{$\LUear(q) \cap O$ is a smooth $1$-codimensional submanifold of $O$}\}.
\end{split}
\eeq
Note that $\LUsmooth(q)$ is the union of smooth pieces of $\LUear(q)$ and clearly we have $\LUreg(q) \subseteq \LUsmooth(q)$.
For each $p \in \LUsmooth(q)$, there exists open neighborhood $O$ of $p$ such that
\beq
\begin{split}\label{def_Lpq}
L_p(q) \coloneqq \LUsmooth(q) \cap O
\end{split}
\eeq
is a smooth $1$-codimensional submanifold.
By definition such $L_p(q)$ may contain points on or before the first cut point and therefore even conjugate points.
To reconstruct $\LUreg(q)$ from $\SUear(W)$, we introduce the following definition and prove the following sequences of lemmas, which are also used in the reconstruction of the conformal structure.

\begin{df}\label{def_tangent}
Let $q_j \in \intM$ and $p \in \LUsmooth(q_j)$ with $L_p(q_j)$ defined in (\ref{def_Lpq}) for $j = 1,2$.
We say $\LUsmooth(q_1)$ and $\LUsmooth(q_2)$ are tangential at $p$
if
\begin{itemize}
    \item[(1)] $p \in L_p(q_j)$ such that $(p,v) \in \Cear_U(q_j)$ for some $v \in L_p^+ \intM$ and $j = 1, 2$;
    \item[(2)] $T_p L_p(q_1) = T_p L_p(q_2)$.
\end{itemize}
\end{df}

\begin{lm}\label{lm_Lpqcausal}
Let $q_1, q_2 \in W$ be distinct and $p \in \pM$.
If $\LUsmooth(q_1)$ and $\LUsmooth(q_2)$ are tangential at $p$, then either $q_1 < q_2$ or $q_2 < q_1$.
\end{lm}
\begin{proof}
By definition, with $(p,v) \in \Cear_U(q_j)$ for $j = 1, 2$, we must have $q_1, q_2$ in the same null geodesic.
In particular, this lemma holds for $\LUreg(q_j)$, $j = 1,2$.
\end{proof}
\begin{lm}\label{lm_Lpqtangential}
Let $q_1, q_2 \in W$ be distinct and $p \in \pM$.
If $\LUsmooth(q_1)$ and $\LUsmooth(q_2)$ are tangential at $p$,
then
\[
L_p(q_1) \cap L_p(q_2) = \{p\}.
\]
\end{lm}
\begin{proof}
By the Lemma \ref{lm_Lpqcausal}, we can assume with loss that $q_1 < q_2$.
Assume for contradiction there exists another $p' \in L_p(q_1) \cap L_p(q_2)$.
Consider the causal path from $q_1$ to $q_2$ to $p'$, which is not a pregeodesic segment.
By Lemma \ref{lm_shortcut}, we must have $\tau(q_1, p') > 0$, which contradicts with $p'$ on or before the first cut point.
\end{proof}
\begin{df}\label{def_tangentialp}
Let $q_1, q_2 \in W$ and $p \in \pM$.
We write $\LUsmooth(q_1)  <_p \LUsmooth(q_2)$ if  they are tangential at $p$ and there exist $y_1 \in L_p(q_1)$ and $y_2 \in L_p(q_2)$ such that
$y_1 < y_2$.
\end{df}
\begin{lm}\label{lm_Lpqcausal2}
Let $q_1, q_2 \in W$ and $p \in \pM$.
If $\LUsmooth(q_1) <_p \LUsmooth(q_2)$, then $q_1 < q_2$.
Moreover, for any $y_1 \in L_p(q_1)$ and $y_2 \in L_p(q_2)$ different from $p$, either they are not causally related or $y_1 < y_2$.
\end{lm}
\begin{proof}
As $q_1 \neq q_2$, by Lemma \ref{lm_Lpqcausal}, we have either $q_1 < q_2$ or $q_2 < q_1$.
Assume for contradiction that $q_2 < q_1$.
Then by Definition \ref{def_tangentialp}, there exists $p_1 \in L_p(q_1)$ and $p_2 \in L_p(q_2)$ such that $p_1 < p_2$.
Consider the causal paths from $q_2$ to $q_1$ and then to $p_1$ and then to $p_2$, which is not a null pregeodesics.
By Lemma \ref{lm_shortcut}, we have $\tau(q_2, p_2) > 0$, which contradicts with $p_2 \in \LUear(q_2)$.
The second statement can be proved by a similar shortcut argument.
\end{proof}
Next, we define
\beq\label{def_lLUinfity}
\SUsmooth(W) = \{\LUsmooth(q): q \in W\}.
\eeq
\begin{lm}\label{lm_reg1}
Let $q \in \intM$ and $p \in \LUreg(q)$. Then there exists $L \in \SUsmooth(W)$ such that $L <_p \LUreg(q)$.
\end{lm}
\begin{proof}
Suppose $p = \gamma_{q,v}(1) \in U$ for some $v \in L_q^+ \intM$.
Consider $q_\ep = \gamma_{q,v}(-\ep)$ for some $\ep> 0$.
Note that $p$ is before the first cut point of $q_\ep$ and therefore $p \in \LUreg(q_\ep)$.
Then we claim $\LUreg(q_\ep)$ and $\LUreg(q)$ are tangential at $p$ and moreover $\LUreg(q_\ep) <_p \LUreg(q)$.
\end{proof}
\begin{lm}\label{lm_reg2}
Let $p \in L_p(q) \subseteq \LUear(q) \setminus \LUreg(q)$. Then there does not exist $L \in \SUsmooth(W)$ such that $L <_p \LUreg(q)$.
\end{lm}
\begin{proof}
Assume for contradiction there exists $q' \in \intM$ such that
$\LUsmooth(q') <_p \LUreg(q)$.
By Lemma \ref{lm_Lpqcausal2}, we have $q' < q$.
Now $p \in \LUear(q) \setminus \LUreg(q)$, the point $p = \gamma_{q,v}(1)$ is on the first cut point of $q$, for some $v \in L_q^+M$.
As $q' < q$, we must have $p$ is after the first cut point of $q'$ and therefore  $\tau(q', p) > 0$, which contradicts with $p \in \LUear(q')$.
\end{proof}
Combining Lemma \ref{lm_reg1} and \ref{lm_reg2}, we prove the following proposition.
\begin{prop}\label{pp_LUreg}
Let $q \in \intM$.
Given $\Cear_U(q)$ for each $q \in W$, we can construct $\SUsmooth(W)$ using (\ref{def_LUinfity}) and (\ref{def_lLUinfity}).
Then
we have
\[\begin{split}
    \LUreg(q) = \{p \in \LUsmooth(q): \text{ there is $L \in \SUsmooth(W)$ such that $L  <_p \LUsmooth$}\},
\end{split}\]
and therefore
\[\begin{split}
    \CUreg(q) = \{(p, w) \in \CUear(q): p \in \LUreg(q)\}.
\end{split}\]
\end{prop}


\subsection{Proof of Theorem \ref{thm_SLOS}}\label{subsec_proof}
Recall $U \subseteq S_+$ and $W \subseteq M$ are open subsets.
In Section \ref{sec_SLOS}, we define $\CUear(q)$ as the earliest scattering light observation set in $U$ for $q \in W$
and $\SUear(W)$ is the collection of such sets for all $q \in W$.

Moreover, by Lemma \ref{lm_uniqueL}, the map $\psi: W \rightarrow \SUear(W)$ is a bijection.
As is stated in \cite{Hintz2017}, it induces a conformal diffeomorphism which pushes the topological, differential, and conformal structures on $W$ to those on $\SUear(W)$.
Using this map, we can identify the set $W$ with the set $\SUear(W)$.
Then reconstructing the structures on $W$ is the same as reconstructing those on $\SUear(W)$.

\subsubsection{Topology}
Recall by Proposition \ref{pp_LUreg}, we can construct the collection of regular scattering direction set
$\SUreg(W)$ from the knowledge of $\SUear(W)$.
Then for $O \subseteq L^+U$ open, we define
\[\begin{split}
U_O &:= \{C\in \SUreg(W) : C \cap O \neq \emptyset\},
\end{split}\]
We collect sets of the form as above for any open set $O \subseteq L^+U$.
Then we define a topology $\mathcal{T}$ on $\SUreg(W)$ by using these sets as a subbasis.
The bijection $\psi: W \rightarrow \SUreg(W)$ induces a topology on $W$.
We prove this induced topology coincides with the subspace topology given by that of $\intM$
in the proposition below, using the same idea as in \cite[Proposition 3.8]{Hintz2017}.

%
%
\begin{prop}\label{pp_topology}
The topology $\mathcal{T}$ of $\SUreg(W)$ is equal to the subspace topology $\mathcal{T}_M$ of $W \subseteq M$, if we identify $\SUreg(W)$ with $W$  using the map $\psi$ in (\ref{def_psi}).
\end{prop}
\begin{proof}
\underline{$\mathcal{T} \subset \mathcal{T}_M$}:
For any open set $O$ of $L^+U$, we would like to prove $\psi^{-1}(U_O)$ are open sets in $W \subseteq \intM$.
We only need to consider that case when $O\neq \emptyset$.
To prove $\psi^{-1}(U_O)$ is open in $W \subseteq \intM$, let $q_0 \in \psi^{-1}(U_O)$.
This means $\CUreg(q_0) \cap O \neq \emptyset$ and therefore
we can find $v_0 \in L_{q_0}^+M$ such that $(\gamma_{q_0, v_0}(1), \dot{\gamma}_{q_0, v_0}(1))\in O$ before the first cut point, i.e., $1 < \rho(q_0, v_0)$.
The goal is to prove for any $(q, v) \in L^+M$ that is sufficiently close to $(q_0, v_0)$, there exists ${\varsigma_{q,v}}$ such that
\[
(\gamma_{q,v}({\varsigma_{q,v}}), \dot{\gamma}_{q,v}({\varsigma_{q,v}}))\in O
\text{ with } {\varsigma_{q,v}} < \rho(q,v),
\]
 and therefore $\CUreg(q) \cap O \neq \emptyset$.
Then we can conclude that $\psi^{-1}(U_O)$ contains a small neighborhood of $q_0$, for arbitrary $q_0$ in this set, and thus $\psi^{-1}(U_O)$ is open.
Indeed, by Lemma \ref{lm_cutneighborhood}, we can find a small neighborhood $N_0 \subseteq L^+\intM$ of $(q_0,v_0)$ such that for any $(q,v) \in N_0$, the point  $\gamma_{q, v}({\varsigma_{q,v}}) \in S_+$ smoothly depends on $(q,v)$ and is before the first cut point.
Moreover, one has $\dot{\gamma}_{q, v}({\varsigma_{q,v}})$ smoothly depends on $(q,v)$.
Thus, with $(q, v) \in L^+M$ sufficiently close to $(q_0, v_0)$, we are able to have $(\gamma_{q, v}({\varsigma_{q,v}}), \dot{\gamma}_{q, v}({\varsigma_{q,v}}))\in O$ with ${\varsigma_{q,v}} < \rho(q,v)$.


\underline{$\mathcal{T}_M \subset \mathcal{T}$}: For this part,
the goal is to show for any $\mathcal{T}_M$-open set $O_M \subseteq W$, each fixed $q \in O_M$ has a $\mathcal{T}$-open neighborhood contained in $O_M$.
More explicitly, we would like to find an open neighborhood $\lU_0 \subseteq \SUreg(W)$ of $\CUreg(q)$ such that for any $\CUreg(q') \in \lU_0$, we have $q' \in O_M$.
For this purpose, we consider an open set $O$ and a compact set $K$ such that $O \Subset K \subseteq L^+U\setminus\{0\}$.
We denote by
\[
C \coloneqq \CUreg(q) \cap {O}
\]
the regular light observation set of $q$ within $O$.
In particular, we choose $O$ as the open neighborhood such that its projeciton $\pi(O)$ onto $U$ satisfies Lemma \ref{lm_smooth_perturb}, i.e.,
\[
\pi(\CUreg(q) \cap {O}) = \LUreg(q) \cap \pi(O) = \lL_q^+ \cap \pi(O).
\]
We pick a countable dense subset
\[
P \coloneqq\{(p_i, w_i) \in L^+U: i \in \mathbb{Z}^+\} \subseteq C
\]
and define $O_{i, \ep} = \{(p, w) \in U: d_{g^+}((p, w), (p_i, w_i)) < \ep\}$,
where $g_+$ is the auxiliary Riemannian metric.
As $C \subseteq K$ is compact, for each fixed $\ep>0$, there exists a finite number $I(\ep)$ such that
$ C \subseteq \textstyle \bigcup_{i=1}^{I(\ep)}(\ep) O_{i, \ep}.$
Then we define a nested sequence given by
\[
{\lU}_j \coloneqq \bigcap_{i=1}^{I(1/j)} U_{O_{i, 1/j}}, \quad j \in \mathbb{Z}^+.
\]
We observe $\lU_j$ contains $\CUreg(q)$ and
$\psi^{-1}(\lU_j)$ are open subsets of $W$ in $\mathcal{T}$.

We claim $\psi^{-1}(\lU_j) \subseteq O_M$ for large $j$ and thus we find a $\mathcal{T}$-open neighborhood of $q$ in $O_M$.
Indeed, assume for contradiction that $\psi^{-1}(\lU_j) \nsubseteq O_M$ for all large $j$.
Then we can pick a sequence $q_j \in \psi^{-1}(\lU_j) \setminus O_M \subseteq W$.
Further, by passing through a subsequence, we can assume $q_j$ converges to $q'$ in $\overline{W}$.
Then we consider the earliest light observation set
\[
C' \coloneqq \CUear(q') \cap {O}
\]
of $q'$ within $O$.
The goal is to prove
$
C \subseteq C'
$
and therefore by Lemma \ref{lm_uniqueLreg} we must have $q'  = q$, which contradicts with $q_j \notin O_M$ and $q \in O_M$.

Indeed, assume for contradiction that ${C} \nsubseteq {C'}$.
Then we can find $i_0 \in \mathbb{Z}^+$ such that $(p_{i_0}, w_{i_0}) \in P \setminus {C'}$.
Moreover, by the definition of $O_{i,\ep}$, we can find $j_0 \in \mathbb{Z}^+$ such that $O_{i_0, 1/j} \cap C'  = \emptyset$ for $j \geq j_0$.
However, since $q_j \in \psi^{-1}(\lU_j)$,
we must have $\CUreg(q_j) \cap O_{i_0, 1/j}$ is nonempty.
Then there exists $(x_j, w_j) \in \CUreg(q_j) \cap O_{i_0, 1/j}$ for any $j \geq j_0$.
Thus, we must have $(p_{i_0}, w_){i_0} = \lim_{j \rightarrow \infty} (x_j, w_j)$.
Now using Lemma \ref{lm_cutconverge}, we can find a subsequence $(q_j, v_j) \rightarrow (q',v') \in L^+\intM$ such that $p_{i_0} = \gamma_{q',v'}(1)$.
Moreover, this implies $w_{i_0} = \dot{\gamma}_{q',v'}(1)$
By Lemma \ref{lm_cutlocus}, one has $p_{i_0}$ is on or before the first cut point of $q'$ and thus $(p_{i_0}, w_{i_0}) \in C'$, which is a contradiction.
Therefore, we must have $C \subseteq C'$, which leads to $q = q'$.

\end{proof}

\subsubsection{Smooth Structure}\label{subsection_smooth_structure}
To reconstruct the smooth structure in $W$, we would like to define a coordinate system locally near any fixed $q \in W$, by using earliest observation time (see (\ref{def_xmu})) along suitable curves passing through the regular part $\LUreg(q)$.
This is the idea used in \cite{Hintz2017}.
We emphasize that for each $q \in W$, the set $\LUreg(q)$ is nonempty, and by Lemma \ref{lm_finite},
there exists an open subset $O \subseteq U$ such that $L \coloneqq \LUreg(q) \cap O$ is a smooth $1$-dimensional submanifold.
Now let $\mu:[-1, 1] \rightarrow O$ be a smooth curve in $\pM$ such that
\begin{itemize}
\item[(1)] $\mu$ is traversal to $L$;
\item[(2)]$\mu'(s) \neq 0$ for $s \in [-1,1]$;
\item[(2)] $\mu(0) \in L$ and $\mu(s) \notin L$ for $s \neq 0$.
\end{itemize}
By Lemma \ref{lm_smooth_perturb}, we can shrink $O$ such that there exists an open neighborhood $N \subseteq W$ of $q$, such that $\LUreg(q') \cap O = \lL_{q'}^+ \cap O$ is a smooth $1$-dimensional submanifold, for any $q' \in N$.
We define
\[
R'(\mu)\coloneqq\{q' \in W: \text{$\LUreg(q')$ intersect $\mu$ once and transversally}\}.
\]
as the set  of all points such that their smooth part of the future light cone surface intersects $\mu$ transversally.
As is pointed out in \cite{Hintz2017}, $R'(\mu)$ is neither open nor closed in general but it contains an open neighborhood of $q$ by Lemma \ref{lm_smooth_perturb}.
One can further define
\[
R(\mu) \coloneqq \bigcup_{\text{$R \subseteq R'(\mu)$ open  in } W} R,
\]
which is a {nonempty} open neighborhood of $q$.
Moreover, we define the earliest observation time along $\mu$ as
\beq
\begin{split}\label{def_xmu}
x^\mu: R(\mu) &\rightarrow [-1,1] \nonumber\\
q &\mapsto s,
\end{split}
\eeq
where $s$ is determined by $\mu(s) \in \LUreg(q)$.
Note that $x^\mu$ is a well-defined function, due to the definition of $R(\mu)$.
Moreover, by Lemma \ref{lm_smoothexp}, the point $p = \gamma_{q,v}({\varsigma_{q,v}}) \in \LUreg(q)$ smoothly depends on $(q, v)$ and therefore $x^\mu$ is smooth function on $R(\mu)$.
As \cite{Hintz2017}, we would like to show that a suitable family of curves $\mu$ give us functions $x^\mu$, which provides local coordinates near $q$.
The key step is to prove there are always enough curves $\mu$ for which $x^\mu$ is non-degenerate at $q$.

\begin{lm}
For fixed $q \in W$, consider the set
\[\begin{split}
\lM \coloneqq \{\mu \in C^\infty([-1, 1];U): \text{$\mu$ satisfies assumptions (1)-(3) above}\},
\end{split}\]
which contains curves that intersects $\LUreg(q)$ once and transversally.
Then we have
\[
\bigcap_{\mu \in \lM} \ker(\dif x^\mu\mid_q) = \{0\} \subseteq T_q M.
\]
\end{lm}
\begin{proof}
We follow the same ideas as in \cite{Hintz2017}, but with Lemma \ref{lm_nulltransversal}, \ref{lm_smoothexp}, \ref{lm_screen} in our setting.
Assume for contradiction there exists a nonzero $V \in T_qM$ such that
\[
\dif x^\mu(V)  = 0 \text{ for any } \mu \in \mathcal{M}.
\]
Then there is a smooth curve $q(r)$ in $W$ for $r \in (-1,1)$, with $q(0) = q$ and $V = \dot{q}(0)$.
For fixed $\mu \in \mathcal{M}$, with sufficiently small $r$, the point $q(r) \in R(\mu)$ and $x^\mu(q(r)) \in [-1,1]$ is well-defined.
If a curve $\tilde{\mu}(r)$ in $U$ is defined by $\tilde{\mu}(r) = \mu(x^\mu(q(r)))$, then we must have $\tilde{\mu}(0) = q$ and $\dot{\tilde{\mu}}(0) = 0$, as $\mu(0) = q$ and $\dif x^\mu(\dot{q}(0)) = 0$. In the following, we would like to use this argument to derive the contradiction.

For this purpose, let $p_0 \in \LUreg(q)$ and $O \subseteq \pM$ be an open neighborhood of $p_0$ as is chosen in Lemma \ref{lm_smooth_perturb}.
Let $O' \Subset O$ be nonempty and we denote by $L \coloneqq \LUreg(q) \cap O'$ the intersection.
Recall $L$ is a smooth $1$-codimensosubmanifold of $O$ of codimension one.
There exists a smooth open map
\[
\upsilon: L \times (-2, 2) \rightarrow O
\]
such that for each fixed $p \in L$, the smooth curve $\upsilon_p: s \mapsto \upsilon(p, s)$ is transversal to $L$ with $\upsilon_p(0) = p$ and $\upsilon$ is a diffeomorphism onto its range $O_\upsilon \subseteq U$.
Note that for fixed $p \in L$, $\upsilon_p \in \mathcal{M}$ and we can define $R(\upsilon_p)$ and $x^{\upsilon_p}$ as above.

Now we denote by  $L(r) =  \LUreg(q(r)) \cap O_\upsilon$ the intersection of the future light cone surface at $q(r)$ with the range of $\upsilon$.
For sufficiently small $r$, the preimage $\upsilon^{-1}(L(r))$ can be written as the graph of a smooth function $f(r, \cdot): L \rightarrow (-2, 2)$.
Indeed, we can define $f$ by
\[
f(r, p) = x^{\upsilon_p}(q(r)), \quad \text{for $(r, p) \in (-\ep, \ep) \times L$}.
\]
Then we have $\upsilon(p, f(r, p)) \in  \LUreg(q)$ by the definition of $x^{\upsilon_p}$ and therefore $\upsilon(p, f(r, p)) \in L(r)$.
Moreover, we have $f(0, p) = x^{\upsilon_p}(q) \equiv 0$ as $\upsilon_p(0) = p$ and we have $\partial_r f(0, p) \equiv 0$ as $\dif x^{\upsilon_p}(\dot{q}(0))  = 0$.
This implies
\[
\upsilon(p, f(0, p)) = p , \quad \partial_r \upsilon(p, f(0, p))  = 0.
\]
It follows that the tangent space
\[
T(r, p) \coloneqq T_{\upsilon(p, f(r,p))}  \LUreg(q(r))
\]
is $r^2$-close to $T(0, p) = T_p  \LUreg(q)$, uniformly for all $p \in O'$.

We claim $T(r,p)$ is a screen space of $T_p \pM$.
Indeed,
$T(r,p)$ is the projection to $M$ of a null bicharacteristic, which hits $\pM$ transversally.
Then we have {\teal $T(r,p) \cap T_{\upsilon(p, f(r,p))} \pM  = \{0\}$}, by Lemma \ref{lm_screenandrad}, $T(r,p)$ is a screen space.
According to Lemma \ref{lm_screen}, there is an isomorphism mapping $T(r,p)$ to a future pointing outward lightlike ray $\mR_+w \subseteq  T(r,p)^\perp$, with $w \in L^+_p M \cap T_p^+M$.  We call this unique ray $l(r, p)$.
Thus,  such $l(r, p)$ is also $r^2$-close to $l(0, p)$.

Next, let $v_1, v_2 \in L^+_q \intM$ be two distinct lightlike vectors such that $p_j  = \exp_q(v_j) \in L$ for $j = 1,2$.
Note that we have $\upsilon(p_j, f(0, p_j)) = p_j$ and we set $w_j = \dot{\gamma}_{q, v_j}(1)$, for $j = 1,2$.
It turns out $w_j \in T(0, p_j)^\perp$, by Lemma \ref{lm_explocal}, and therefore $\mR_+ w_j = l(0, p_j)$.
Now for $(r, p) \in (-\ep,\ep) \times L$, since each $l(r,p)$ is well-defined, we can consider a generator $w(r, p)$ of $l(r,p)$, which depends on $(r, p)$ smoothly and is $r^2$-close to $w(0, p)$.
Observe that the two null geodesics $s \mapsto \exp_{p_j}(-s w(0, p_j))$ intersect cleanly at $q$ by our construction.
For small $r$, the null geodesics $s \mapsto \exp_{\upsilon(p, f(r,p))}(-s w(r, p_j))$ {intersect exactly at $q(r)$}, by Lemma \ref{lm_explocal} again.
This implies $q(r)$ depends smoothly on $\upsilon(p, f(r,p))$ and $w(r, p_j)$ and therefore is $r^2$-close to $q(0) = q$.
It follows that $V = \dot{q}(0) = 0$, which contradicts with our assumption.
\end{proof}

Thus, for every $q \in W$, there exist $(n+1)$ curves $\mu_j \in \lM$ such that the set $\{\dif x^{\mu_j}: j = 0, \ldots, n\}$ is linearly independent at $q$ and thus $x^{\mu_j}$ forms a smooth local coordinate system near $q$.
Then the argument in \cite{Hintz2017} shows one can recover the algebra of smooth functions on $W$ from the family of sets $\SUear(W)$.

\subsubsection{Conformal Structure}
We follow \cite{Hintz2017}.
By the assumption in Theorem \ref{thm_SLOS}, for each $q \in W$, the set $\LUreg(q)$ is nonempty.
Using Proposition \ref{pp_LUreg}, we construct the collection of regular scattering light observation set
$\LUreg(W)$ from the knowledge of $\SUear(W)$.
Now let $p \in \LUreg(q)$ such that $p = \gamma_{q,v}(1)$ for some $v\in L_q^+ \intM$.
Recall Definition \ref{def_tangent} and \ref{def_tangentialp}.
We define the set
\[
\begin{split}
    \mathcal{Q} =\{\mu \in C^\infty((-1,1); W):  &\text{ $\mu(0) = q$ and $p \in \LUreg(\mu(r))$ for any $r \in (-\ep,\ep)$}\\
    & \quad \quad \quad \quad \text{such that $\LUreg(\mu(r))$ and  $\LUreg(q)$ are tangential at $p$}\}.
\end{split}
\]
By Lemma \ref{lm_screen} and Lemma \ref{lm_explocal},
this set contains all smooth curves which have the same tangent space at $p$ and therefore the same future pointing lightlike ray there.
Then from $\{\dot{\mu}(0): \mu \in \mathcal{Q}\} = \mR_+ v \in L_q^+\intM$, we recover a one-dimensional lightlike subspace of $L_q\intM$.

Next we repeat this procedure for all points \(p \in \LUreg(q)\).
Since $\LUreg(q)$ is an open set, this allows us to reconstruct an open subset of the light cone \({L}^+_q M \subseteq T_q M\).
As \({L}_q M\) is a real-analytic submanifold of \(T_q M\) determined by a quadratic equation, this determines \(\mathcal{L}_q M\) uniquely.
\subsubsection{Time orientation}
One can proof follow the following lemmas using the same idea as \cite{Hintz2017} or \cite{MScfinal}.
\begin{lm}
    Suppose $(q(r), v(r)) \in L^{+}M$ for $r \in (-1,1)$ is a smooth path such that $p(r) := \exp_{q(r)}^{b} v(r) \in \partial M$.
    Let $\gamma(r,s) \coloneqq \exp_{q(r)}^{b}(sv(r))$ and $\gamma(s) := \gamma(0, s)$. Then we have
    \[
    g(q'(0), v(0)) = g(p'(0), \gamma'(1)).
    \]
\end{lm}
\begin{lm}
    Let $p(r) \in\LUreg(q(r)) \cap \mathcal{U}$ be a smooth path and denote by $N \in T_{p(0)}\partial M$ the future-directed unit normal to the spacelike hypersurface $T_{p(0)}\LUreg(q(0))$. Then $q$ is future timelike if and only if $g(p'(0), N) < 0$.
\end{lm}


%% file: reconstructionscheme.tex
\section{Reconstruction of the scattering light observation sets}\label{sec_reconstruction}
In this section, we propose schemes to reconstruct the scattering light observation sets in each step of the layer stripping method in Section \ref{subsec_layersteps}.
Recall in Step 1, we would like to reconstruct the scattering light observation sets by sending receding waves in an open subset $U_-$ of the past null infinity $S_-$ and detecting new singularities in an open subset $U_+$ of the future null infinity $S_+$. In particular, we assume there are no caustics in the region of interest $W = I(U_-, U_+)$.
First, in Section \ref{subsec_3to1}, we prove the scattering light observation sets for all points in $W$ can be reconstructed from a so-called three-to-one scattering relation.
Essentially, this relation indicates how suitable lightlike vectors on $U_-$ are related with lightlike vectors on $U_+$ by nonlinear wave interaction.
Next, in Section \ref{sec_detection}, we show how to extract such a three-to-one scattering relation using the nonlinear scattering operator $\lN$.
Then, in Section \ref{subsec_schemes}, we explain how to use these two ingredients to reconstruct the scattering light observation sets in each step of Section \ref{subsec_layersteps}.

\input{three_to_one_relation}

\input{detection_condition}

\subsection{Schemes to reconstruct the scattering light observation sets}\label{subsec_schemes}
In the following, we explain how to  reconstruct the scattering light observation sets in each step of the layer stripping method in Section \ref{subsec_layersteps}.
\subsubsection{Scheme 1} \label{sec_scheme1}
We start from the reconstruction of the scattering light observation sets for Step 1.
In this step, we send conormal waves in $U_-$ and detect new singularities in $U_+$.
By Section \ref{sec_detection}, from the third-order linearization of the nonlinear scattering operator,
we determine a three-to-one scattering relation $R_s \subseteq \LUpl \times \Vmi$.
In this case, we take $\Vmi = (\LUpl)^3$ for Theorem \ref{thm_3to1}.
As there are no cut points in $\overline{W} = J(U_-, U_+)$, by Theorem \ref{thm_3to1}, this $R_s$ determines the scattering light observation sets of any points in $I(U_-, U_+)$ restricted to $U_+$.
Then by Theorem \ref{thm_SLOS}, this collection of the scattering light observation sets determines the metric in $W$ up to conformal diffeomorphisms.
\input{scheme2}
\input{scheme3}
\input{scheme4}

%% file: three_to_one_relation.tex
\subsection{Reconstruction from a three-to-one scattering relation}\label{subsec_3to1}
In the following, let $U_\pm \subseteq S_\pm$ be open subsets such that $I(U_-, U_+)$ is nonempty.
We consider an open subset
\[
W \subseteq I(U_-, U_+),
\]
for which we would like to reconstruct the metric.
In Step 1, we choose $W  = I(U_-, U_+)$, but later $W$ can be a proper subset of $I(U_-, U_+)$.
For example, in Step 2 we choose $W = I(U_-, U_+)\setminus \Jmi(\Spl(T_1))$ and in Step 3 we choose $W = I(U_-, U_+)\setminus \Jpl(\Smi(T_1))$.
By the construction in Section \ref{sec_layer}, we assume there are no cut points along any geodesic segments restricted to $\overline{W}$.

For convenience, we use the notation $\ups = (p, w) \in L^+M$ to denote a lightlike vector and $\gamma_{\ups}(\varsm)$ to denote a null geodesic starting from $\ups$.
Our goal is to reconstruct the scattering light observation sets from a relation between lightlike vectors on $U_-$ and $U_+$.
Recall the definition of a three-to-one scattering relation in \cite{feizmohammadi2021inverse} defined for lightlike vectors in open sets.
As an analog, we introduce the following definition for lightlike vectors on $U_\pm$, in a simplified but localized setting, in the sense that we assume there are no caustics in $\overline{W}$ but our reconstruction works for a subset of lightlike vectors on $U_-$.

For this purpose, first we define
\beq\label{def_LUpl}
\LUpl \coloneqq L^+U_+ \cap T_{U_+}^+ M
\quad \text{and} \quad
\LUmi \coloneqq L^+U_- \cap T_{U_-}^+ M
\eeq
as the set of future pointing lightlike vectors that are transversal to $U_\pm$,
where $T_p^+ M$ is defined in (\ref{def_outwardTM}) for $p \in U_\pm$.
The following lemma comes from \cite[Lemma 7.2]{feizmohammadi2021inverse} and a shortcut argument.
\begin{lm}\label{lm_finiteF}
    Suppose there are no cut points along any geodesic segments in $\overline{W}$.
    Let $\ups_1, \ups_2 \in \LUmi$ be distinct.
    Then the set
    \[
    F \coloneqq \gamma_{\ups_1}(\mR_+) \cap \gamma_{\ups_2}(\mR_+) \cap \overline{W}
    \]
     contains at most one point.
\end{lm}
Then we consider a subset $\Vmi \subseteq (\LUmi)^3$ and introduce the following definition.
\begin{definition}\label{def_sufficientV}
    Let $\Vmi \subseteq (\LUmi)^3$.
    We define
    \[
    \begin{split}
        \Vmione(q) = \{\ups_1 \in \LUmi: &\text{ there is a neighborhood $N_1 \subseteq \LUmi$ of $\ups_1$ such that  if $\tups_j \in N_1$}\\
        &\text{  are distinct with $q = \gamma_{\tups_j}(\varsm_j)$ for some $\varsm_j>0$, then $(\tups_1, \tups_2, \tups_3) \in \Vmi$}.\}
    \end{split}
    \]
    We say $\Vmi \subseteq (\LUmi)^3$ is \textit{sufficient} for $W \subseteq I(U_-, U_+)$, if
    \begin{itemize}
        \item[(a)] for each $(\ups_1, \ups_2, \ups_3) \in \Vmi$, there does not exist $q_0 \in \gamma_{\ups_i}(\mR_+) \cap \gamma_{\ups_j}(\mR_+) \cap (M \setminus W)$ such that
        $q_0 < q$ for some $q \in \overline{W}$, and
        \item[(b)] for each $q \in W$,  the set $\Vmione(q) \neq \emptyset$.
    \end{itemize}
\end{definition}
Note that for Step 1 we choose $W = I(U_-, U_+)$ and $\Vmi= (\LUmi)^3$, which satisfies the definition above. For Step 2 and later, we choose $\Vmi$ to be a subset containing lightlike vectors such that the corresponding null geodesics do not intersect before entering $W$.
\begin{df}\label{def_3to1}
    Let $\Vmi \subseteq (\LUmi)^3$ be sufficient for $W\subseteq I(U_-, U_+)$ as in Definition \ref{def_sufficientV}.
    We say a relation $R \subseteq \LUpl\times \Vmi$ is a three-to-one scattering relation for $W$, if it has the following two properties:
    \begin{itemize}
        \item[(R1)]\label{def_R1} If $(\upsilon_0, \upsilon_1, \upsilon_2, \upsilon_3) \in R$, then there exists an intersection point
        \[q \in \Gammaset \cap \overline{W}.\]
        \item[(R2)]\label{def_R2} Assume that $\gamma_{\upsilon_j}$ for $j=0,1,2,3$ are distinct and there exists
        \[
        q \in \Gammaset \cap \overline{W}
        \]
        such that $q = \gamma_{\upsilon_j}(\varsm_j)$ with $\varsm_j >0$ for $j=1,2,3$ and
        \[
        q = \gamma_{\upsilon_0}(\varsm_0) \text{ with }
        \varsm_0 \in (-\rho(\upsilon_0^{-}), 0),
        \]
        where $\upsilon_0^- \coloneqq (y, -w) \in L^-_{U_+}M$ if we write $\upsilon_0 = (y, w)$.
        In addition, assume
        $\dot{\gamma}_{\upsilon_0}(\varsm_0) \in \text{span}(\dot{\gamma}_{\upsilon_1}(\varsm_1), \dot{\gamma}_{\upsilon_2}(\varsm_2), \dot{\gamma}_{\upsilon_3}(\varsm_3))$.
        Then, it holds that $(\upsilon_0, \upsilon_1, \upsilon_2, \upsilon_3) \in R$.
    \end{itemize}
\end{df}
We remark that, compared to \cite{feizmohammadi2021inverse}, for (R2) we do not assume $q$ is before the first cut point along each $\gamma_{\upsilon_j}$ for $j = 1,2,3$.
Moreover, such a three-to-one scattering relation is only defined in a subset of $\LUpl\times (\LUmi)^3$ and we assume a stronger (R1) due to our construction in Section \ref{sec_layer}.
The sufficient condition guarantees $\Vmi$ is large enough for the reconstruction.

Next, we would like to consider lightlike vectors that are corresponding to null geodesics intersecting at points away from $U_\pm$. Thus, we introduce the following definition.
\begin{df}\label{def_proper}
    We say $(\ups_0, \ups_1, \ups_2, \ups_3) \in \LUpl\times (\LUmi)^3$ are proper if
    \begin{itemize}
        \item[(1)] the points $\pi(\ups_j)$ for $j = 1, 2, 3$ are distinct;
        \item[(2)] the points $\pi(\ups_0)$ and $\pi(\lL(\ups_j))$ for $j = 1,2,3$ are distinct, where $\lL: \LUmi \rightarrow  \LUpl$ is the lens relation.
    \end{itemize}
\end{df}
Note that $(\ups_0, \ups_1, \ups_2, \ups_3)$ is proper if and only if they are corresponding to distinct null geodesics and do not intersect with each other at $U_\pm$.
Combining Lemma \ref{lm_finiteF}, (R1), and Definition \ref{def_proper}, we have the following lemma.
\begin{lm}\label{lm_properintersection}
    Suppose  there are no cut points along any geodesic segments in $\overline{W}$.
    Let $R$ be a three-to-one scattering relation and $(\ups_0, \ups_1, \ups_2, \ups_3) \in R$ be proper.
    Then there exists a unique $q$ such that
    \[
    \Gammaset \cap \overline{W}= \{q\}.
    \]
\end{lm}
Next for two lightlike vectors $\ups_1, \ups_2$ in $U_+$,
we consider conical pieces given by
\beq\label{def_CP}
\CP(\upsilon_1, \upsilon_2) = \{\upsilon_0 \in \LUpl: \text{ there is $\upsilon_3 \in \LUmi$ such that $(\upsilon_0, \upsilon_1, \upsilon_2, \upsilon_3) \in R$}\}.
\eeq
We emphasize if there are no such $\ups_3$ that $(\ups_1, \ups_2, \ups_3) \in \Vmi$, then  $\CP(\upsilon_1, \upsilon_2) = \emptyset$.

\begin{lm}\label{lm_CPq}
    Suppose there are no cut points along any geodesic segments in $\overline{W}$.
    Let $\ups_1, \ups_2 \in \LUmi$ such that $(\ups_0', \ups_1, \ups_2, \ups_3') \in R$ is proper, for some $\ups_0' \in \LUpl$ and $\ups_3' \in \LUmi$.
    Then there exists a unique $q$ such that  $ \gamma_{\ups_1}(\mR_+) \cap \gamma_{\ups_2}(\mR_+) \cap  \overline{W} = \{q\}$ and moreover
    \[
    \CP(\upsilon_1, \upsilon_2)\subseteq C_{U_+}(q),
    \]
    where $C_{U_+}(q)$ is defined in (\ref{def_CU}).
\end{lm}
\begin{proof}
    Indeed, applying Lemma \ref{lm_properintersection} to $(\ups_0', \ups_1, \ups_2, \ups_3')$ and using Lemma \ref{lm_finiteF}, we can find a unique $q$ such that
    $
     \gamma_{\ups_1}(\mR_+) \cap \gamma_{\ups_2}(\mR_+) \cap \overline{W} = \{q\}.
    $
    Now let $\ups_0 \in \CP(\upsilon_1, \upsilon_2)$.
    There exists $\upsilon_3 \in \LUmi$ such that $(\upsilon_0, \upsilon_1, \upsilon_2, \upsilon_3) \in R$.
    (R1) implies
    $
    \Gammaset \cap \overline{W} \neq \emptyset
    $ and therefore  by Lemma \ref{lm_finiteF} again, this set must equal to $\{q\}$.
    Thus, such $\ups_0 \in C_{U_+}(q)$ and we conclude $\CP(\upsilon_1, \upsilon_2)\subseteq C_{U_+}(q)$.
\end{proof}
Now we would like to use these conical pieces to generate the future light cone surface of $q$ within $U_+$.
For this purpose, for $\ups_1, \ups_2 \in \LUmi$ we define
\[
\begin{split}
    C(\ups_1, \ups_2) \coloneqq \{C &\subseteq \LUpl: \text{ $C$ is a smooth submanifold of dimension $n-1$ such that}\\
    &\text{$C \subseteq \CP(\ups_1, \ups_1')\cap \CP(\ups_2, \ups_2')$ for some $\ups_1', \ups_2' \in  \LUpl$ with $\ups_1 \neq \ups_1'$ and $\ups_2 \neq \ups_2'$}\}.
\end{split}
\]
and define
\[
\lC(\ups_1, \ups_2) \coloneqq \textstyle \bigcup_{C \in C(\ups_1, \ups_2)} \bar{C}.
\]
\begin{prop}\label{pp_CCU}
    Suppose there are no cut points along any geodesic segments in $\overline{W}$.
    Let $\ups_1, \ups_2 \in \LUmi$ such that $(\ups_0', \ups_1, \ups_2, \ups_3') \in R$ is proper, for some $\ups_0' \in \LUpl$ and $\ups_3' \in \LUmi$.
    Recall there exists a unique $q$ such that
    $\gamma_{\ups_1}(\mR_+) \cap \gamma_{\ups_2}(\mR_+) \cap \overline{W} = \{q\}$.
    Then we have
    \[
    \lC(\ups_1, \ups_2) \subseteq C_{U_+}(q).
    \]
\end{prop}
\begin{proof}
    Let $\ups \in \lC(\ups_1, \ups_2)$.
    Then by the definition, such $\ups$ is contained in $\CP(\ups_1, \ups_1') \cap \CP(\ups_1, \ups_1')$ for some $\ups_1', \ups_2' \in  \LUmi$ with $\ups_1 \neq \ups_1'$ and $\ups_2 \neq \ups_2'$.
    By Lemma \ref{lm_CPq}, there exists $q_j \in \gamma_{\ups_j}(\mR_+) \cap \gamma_{\ups_j'}(\mR_+) \cap \overline{W}$ such that
    \[
    \CP(\ups_j, \ups_j') \subseteq C_{U_+}(q_j),
    \]
    and thus $\ups \in C_{U_+}(q_j)$ for $j = 1,2$.
    We claim that  $q = q_1 = q_2$.
    With $\ups \in C_{U_+}(q_j)$
    for $j = 1,2$, we have either $q_1 \leq q_2$ or $q_2 < q_1$.
    We may assume for contradiction that $q_2 < q_1$ along $\gamma_{\ups}(\mR_-)$.
    Note that both $q$ and $q_1$ are contained in $\gamma_{\ups_1}(\mR_+)$
    and both $q$ and $q_2$ are contained in $\gamma_{\ups_2}(\mR_+)$.
    If $q > q_1$, then we have causal path from $q_2$ to $q_1$ along $\gamma_{\ups}(\mR_-)$ and from $q_1$ to $q$ along $\gamma_{\ups_1}(\mR_+)$.
    This causal path is not a null pregeodesic and by Lemma \ref{lm_shortcut} we have $q \ll q_2$ is after the first cut point, which contradicts with our assumption.
    If $q < q_1$, we can use a similar shortcut argument.
    Thus, we must have $q = q_1$.
    This implies $q \in \gamma_{\ups_1}(\mR_+) \cap \gamma_{\ups_2}(\mR_+) \cap \gamma_{\ups}(\mR_-) \cap \overline{W}$.
    By Lemma \ref{lm_finiteF} again, we have $q = q_2$.
    This proves $\ups \in C_{U_+}(q)$.
\end{proof}
Next, we define
\[
\lCear(\ups_1, \ups_2) \coloneqq \{\ups \in \lC(\ups_1, \ups_2): \text{ there is no $\tups \in \lC(\ups_1, \ups_2)$ such that $\pi(\tups) < \pi{(\ups)}$ in $S_+$}\}
\]
as the earliest part of the conical pieces given by $\ups_1, \ups_2$.
\begin{prop}[{\cite[Proposition 2.2.7]{MScfinal}}]\label{prop_CearU}
    Let $q \in M$ and recall the definition of $\Cear_{U_+}(q)$ in (\ref{def_CUear}).
    Then
    \[
    \Cear_{U_+}(q) = \{\ups \in C_{U_+}(q): \text{ there is no $\ups' \in C_{U_+}(q)$ such that $\pi(\ups') < \pi(\ups)$ in $S_+$}\}.
    \]
\end{prop}
Combining Proposition \ref{prop_CearU} and Proposition  \ref{pp_CCU}, we have the following lemma.
\begin{lm}\label{lm_CearU}
Let $\ups_1, \ups_2$ and $q$ be defined as in Proposition \ref{pp_CCU}.
Then we have
\[
\lCear(\ups_1, \ups_2) \subseteq \Cear_{U_+}(q).
\]
\end{lm}
\begin{prop}\label{prop_Cearv1v2}
Let $\ups_1, \ups_2$ and $q$ be as in Proposition \ref{pp_CCU}.
Suppose $\ups_1, \ups_2 \in \Vmione(q)$ and {$V_-$ only contains lightlike vectors of which the corresponding null geodesics do not intersect before $W$.}
Then we have
    \[
    \Creg_{U_+}(q) \subseteq \lCear(\ups_1, \ups_2),
    \]
    where $\Creg_{U_+}(q)$ is defined in (\ref{def_CUreg}).
\end{prop}
\begin{proof}
    Let $\ups \in \Creg_{U_+}(q)$ in the sense that there exists $w_0 \in L^+_qM$ such that
    \[
    \ups = (\gamma_{q,w_0}(\varsm_0), \dot{\gamma}_{q,w_0}(\varsm_0))\quad \text{ for some } 0<\varsm_0<\rho(q, w_0).
    \]
    The goal is to show $\ups \in \lCear(\ups_1, \ups_2)$.
    First we prove $\ups \in \lC(\ups_1, \ups_2)$.
    Indeed, we want to find $\ups_1', \ups_2'\in  \LUmi$ with $\ups_1 \neq \ups_1'$ and $\ups_2 \neq \ups_2'$ such that
    there is a small neighborhood $C$ of $\ups$ in $\Creg_{U_+}(q)$ satisfying
    \[
    C \subseteq \CP(\ups_1, \ups_1') \cap \CP(\ups_2, \ups_2').
    \]
    It suffices to find such $\ups_1'$ for $\ups_1$.
    For this purpose, we write
    \[
    (q,w_1) = (\gamma_{\ups_1}(\varsm_1), \dot{\gamma}_{\ups_1}(\varsm_1))
    \quad \text{ and }\quad
    \ups = (\gamma_{q,w_0}(\varsm_0), \dot{\gamma}_{q,w_0}(\varsm_0)),
    \]
    for some $\varsm_0, \varsm_1 > 0$.
    By Lemma \ref{lm_cutneighborhood}, there exists a small neighborhood $N_0' \subseteq L_q^+M$ of $w_0$ such that for any $w_0' \in N_0'$, the null geodesic $\gamma_{q,w_0}$ hits $U_+$ before the first cut point.
    By \cite[Lemma 6.23]{feizmohammadi2021inverse},
    for any neighborhood $N_1 \subseteq L_q^+M$ of $w_1$,
    there exists a neighborhood $N_0 \subseteq L^+_q M$ of $w_0$
    and $w_3 \in N_1$ such that for any $w_0' \in N_0$, there exists $w_4\in N_1$ satisfying
    \[
    w_0' \in \mathrm{span}(w_1, w_3, w_4) \quad \text{ and }\quad  w_0' \notin \mathrm{span}(w_j), \text{ for } j = 1,3,4.
    \]
    We can shrink $N_0'$ and $N_0$ if necessary to set them equal.
    Now let \[
    \text{$\ups_j = (\gamma_{q,w_j}(-\varsm_j), \dot{\gamma}_{q,w_j}(-\varsm_j))$  such that $\gamma_{q,w_j}(\varsm_j) \in S_-$}
    \]for some $\varsm_j > 0$, where $j = 3,4$.
    Recall $\ups_1 \in \Vmione(q)$, which implies there is a small neighborhood $\tN_1 \subseteq \LUmi$ of $\ups_1$ given by Definition \ref{def_sufficientV}.
    In particular, we can choose $N_1$ sufficiently small such that
    $\ups_j = (\gamma_{q,w_j}(\varsm_j), \dot{\gamma}_{q,w_j}(\varsm_j))$ are sufficiently close to $\ups_1$ and therefore are contained in $\tN_1$, for $j = 3,4$.
    Then we must have $(\ups_1, \ups_3, \ups_4) \in \Vmi$.
    Let \[
    \text{$\ups_0' = (\gamma_{q,w_0'}(\varsm_0'), \dot{\gamma}_{q,w_0'}(\varsm_0'))$ such that $\gamma_{q,w_0'}(\varsm_0') \in S_+$}
    \]
    for some $\varsm_0' > 0$ .
    As $\ups \in \Creg_{U_+}(q)$, by our choice of $N_0'$, there are no cut points from $q$ to $\pi(\ups_0')$.
    Moreover, by Lemma \ref{lm_nulltransversal}, one has $\ups_0 \in \LUpl$.
    Thus, by Condition (R2), we have $(\ups_0', \ups_1, \ups_3, \ups_4) \in R$ and therefore $\ups_0' \in \CP(\ups_1, \ups_1')$ for any $\ups_0'$ sufficiently close to $\ups$.
    This implies $\ups \in \lC(\ups_1, \ups_2)$ and it remains to show it is actually in $\lCear(\ups_1, \ups_2)$.
    We assume for contradiction that there exists $\ups' \in \lC(\ups_1, \ups_2)$ such that $\pi(\ups') < \pi(\ups)$.
    A shortcut argument using Lemma \ref{lm_shortcut} shows we must have $\pi(\ups) \gg q$, which contradicts with $\ups \in \Creg_{U_+}(q)$.
\end{proof}
To reconstruct $\Creg_{U_+}(q)$, we recall that the analysis in Section \ref{subsec_LUreg} motivates us to define
\[
\begin{split}
    \lCsmooth(\ups_1, \ups_2) \coloneqq \{\ups \in \lCear(\ups_1, \ups_2): &\text{ there is an open neighborhood $O$ of $\pi(\ups)$ such that}\\
    &\text{$\pi(\lE(\ups_1, \ups_2)) \cap O$ is a smooth $1$-codimensional submanifold of $O$}\}
\end{split}
\]
as the set of lightlike vectors whose projection to $U_+$ is contained in a smooth part.
Further, recall Definition \ref{def_tangentialp} and
we define
\[
\begin{split}
    \lCreg(\ups_1, \ups_2) \coloneqq \{\ups \in \lCsmooth(\ups_1, \ups_2): &\text{ there are $\ups_1', \ups_2' \in  \LUmi$ such that }\\
    &\quad \quad \quad \quad \quad \quad \quad \quad \text{ $\lCsmooth(\ups_1', \ups_2') <_{\pi(\ups)} \lCsmooth(\ups_1, \ups_2)$}\}.\end{split}
\]
as the regular part.
\begin{prop}\label{pp_CregU}
    Suppose there are no cut points along any geodesic segments in $\overline{W}$.
    Let $\ups_1, \ups_2 \in \LUmi$ such that $(\ups_0', \ups_1, \ups_2, \ups_3') \in R$ is proper, for some $\ups_0' \in \LUpl$ and $\ups_3' \in \LUmi$.
    Then exists a unique $q$ such that
    $\gamma_{\ups_1}(\mR_+) \cap \gamma_{\ups_2}(\mR_+) \cap \overline{W} = \{q\}$ and moreover
    \[
    \lCreg(\ups_1, \ups_2) \subseteq \Creg_{U_+}(q).
    \]
\end{prop}
\begin{proof}
By Lemma \ref{lm_CearU}, first we have $\lCear(\ups_1, \ups_2) \subseteq \Cear_{U_+}(q)$ and therefore $\lCreg(\ups_1, \ups_2) \subseteq \Cear_{U_+}(q)$.
Then we consider the set $\Cear_{U_+}(q) \setminus \Creg_{U_+}(q)$.
By Lemma \ref{lm_reg2}, this part is disjoint from $\lCreg(\ups_1, \ups_2)$.
Thus, we must have $\lCreg(\ups_1, \ups_2) \subseteq \Creg_{U_+}(q)$.
\end{proof}
\begin{lm}[{\cite[Lemma 6.23]{feizmohammadi2021inverse}}]\label{lm_null_construction}
    For $q \in W$, let $w_0, w_1 \in L_q^+M$ be linearly independent and let $N_1 \subseteq L_q^+M$ be a small neighborhood of $w_1$.
    Then there exists a neighborhood $N_0 \subseteq L_q^+M$ of $w_0$ and $w_2 \in N_1$ such that for any $w_0' \in N_0$, we can find $w_3 \in N_1$ satisfying
    \[
    w_0' \in \mathrm{span}(w_1, w_2, w_3) \quad \text{ and }\quad  w_0' \notin \mathrm{span}(w_j), \text{ for } j = 1,2,3.
    \]
\end{lm}
%

\begin{prop}
    Suppose there are no cut points along any geodesic segments in $\overline{W}$.
    Let $\ups_1, \ups_2 \in \LUmi$ such that $(\ups_0', \ups_1, \ups_2, \ups_3') \in R$ is proper, for some $\ups_0' \in \LUpl$ and $\ups_3' \in \LUmi$.
    There exists a unique $q$ such that
    \[
    \gamma_{\ups_1}(\mR_+) \cap \gamma_{\ups_2}(\mR_+) \cap \overline{W} = \{q\}.
    \]
    In addition, suppose $\ups_1, \ups_2 \in \Vmione(q)$ and {$V_-$ only contains lightlike vectors of which the corresponding null geodesic do not intersect before $W$.}
    Then we have
    \[
    \lCreg(\ups_1, \ups_2) = \Creg_{U_+}(q).
    \]
\end{prop}
\begin{proof}
    By Proposition \ref{pp_CregU}, we have $\lCreg(\ups_1, \ups_2) \subseteq \Creg_{U_+}(q)$. It remains to prove the opposite direction.
    By Proposition \ref{prop_Cearv1v2}, we have
    \[
    \Creg_{U_+}(q)\subseteq \lCear(\ups_1, \ups_2).
    \]
    Now recall some results in Section \ref{subsec_LUreg}.
    First by the definition of $\lCsmooth(\ups_1, \ups_2)$ and Lemma \ref{lm_smooth_perturb}, we have
    $\Creg_{U_+}(q)\subseteq \lCsmooth(\ups_1, \ups_2).$
    {Next by Proposition \ref{pp_LUreg}, we must have
    $\Creg_{U_+}(q)\subseteq \lCreg(\ups_1, \ups_2). $}
%
%
%
    Thus, we conclude that $\Creg_{U_+}(q) = \lCreg(\ups_1, \ups_2)$.
\end{proof}
Next, for suitable $\ups_1, \ups_2 \in \LUmi$, we state how to determine whether they are in $\Vmione(q)$ or not.
More precisely, we consider $\ups_1, \ups_2 \in \LUmi$ that satisfies
$
\gamma_{\ups_1}(\mR_+) \cap \gamma_{\ups_2}(\mR_+) \cap \overline{W} = \{q\}
$
and $\lCreg(\ups_1, \ups_2) \neq \emptyset$.
By Lemma \ref{lm_CearU}, the second assumption enables us to find distinct $\ups_0, \tups_0 \in \lCreg(\ups_1, \ups_2)$ such that
\[
\gamma_{\ups_0}(\mR_+) \cap \gamma_{\tups_0}(\mR_+) \cap \overline{W} = \{q\}.
\]
We define
\[
\begin{split}
\lV(\ups_0, \tups_0) \coloneqq \{ \ups_1 \in \LUmi:
&\text{$(\ups_0, \ups_1, \ups_3, \tups_3) \in R$ for some $\ups_3, \tups_3 \in \LUmi$ and}\\
    & \quad \quad\quad \quad \quad \quad
   \text{$(\tups_0, \ups_1, \ups_4, \tups_4) \in R$ for some $\ups_4, \tups_4 \in \LUmi$}
\}.
\end{split}
\]
\begin{lm}\label{lm_Vonetwo}
Let $q$ and $\ups_0, \tups_0$ be defined as above.
If $\tups_1 \in \lV(\ups_0, \tups_0)$, then $\tups_1 \in C_{U_-}^-(q)$, where
\[
C_{U_-}^-(q) = \{\ups \in \LUmi: q = \gamma_\ups(\varsm_q) \text{ for some } \varsm_q > 0\}.
\]
\end{lm}
\begin{proof}
By the definition of $\lV(\ups_0, \tups_0)$, we have
$(\ups_0, \ups_1, \ups_3, \tups_3) \in R$ and $(\tups_0, \ups_1, \ups_4, \tups_4) \in R$ for some $\ups_3, \tups_3, \ups_4, \tups_4 \in \LUmi$.
Condition (R\ref{def_R1}) implies there exists $q_1, q_2$ such that
\[
 \gamma_{\ups_0}(\mR_+) \cap \gamma_{\ups_1}(\mR_+) \cap \overline{W} = \{q_1\}, \quad \gamma_{\tups_0}(\mR_+) \cap \gamma_{\ups_1}(\mR_+) \cap \overline{W} = \{q_2\}.
\]
On the other hand, as $\ups_0, \tups_0 \in \lCreg(\ups_1, \ups_2)$, by Lemma \ref{pp_CregU} and Lemma \ref{lm_finiteF}, we have
\[
\gamma_{\ups_0}(\mR_+) \cap \gamma_{\tups_0}(\mR_+) \cap \overline{W} = \{q\}.
\]
A same shortcut argument as in the proof of Proposition \ref{pp_CCU} implies $q = q_1 = q_2$.
Thus, we have $\ups_1 \in C_{U_-}^-(q)$.
\end{proof}
Then we define
\[
\begin{split}
    \lVsmooth(\ups_0, \tups_0) \coloneqq \{C &\subseteq \lV(\ups_0, \tups_0): \text{ $C$ is a smooth $(n-1)$-dimensional submanifold}\}
\end{split}
\]
and then
\beq
\begin{split}
\lV_1(\ups_0, \tups_0) \coloneqq \{
&\ups_1 \in \LUmi: \text{$\ups_1 \in C$ for some $C \in \lVsmooth(\ups_0, \tups_0)$ and $\exists$ a small neighborhood} \\
&\text{$N_1 \subseteq \LUmi$ of $\ups_1$ satisfying if $\tups_j \in N_1 \cap C$ are distinct, then $(\tups_1, \tups_2, \tups_3) \in \Vmi$}.
\}
\end{split}
\eeq
We emphasize the set $\lV_1(\ups_0, \tups_0)$ is determined by $R$ and $\Vmi$.
\begin{prop}\label{pp_determineV}
Suppose there are no cut points along any geodesic segments in $\overline{W}$.
Let $\ups_1, \ups_2 \in \LUmi$ such that $(\ups_0', \ups_1, \ups_2, \ups_3') \in R$ is proper, for some $\ups_0' \in \LUpl$ and $\ups_3' \in \LUmi$.
There exists a unique $q$ such that
$
\gamma_{\ups_1}(\mR_+) \cap \gamma_{\ups_2}(\mR_+) \cap \overline{W} = \{q\}.
$
If $\lCreg(\ups_1, \ups_2) \neq \emptyset$,
then $\ups_1 \in \Vmione(q)$ if and only if $\ups_1 \in \lV_1(\ups_0, \tups_0)$ for some $\ups_0, \tups_0 \in \lCreg(\ups_1, \ups_2)$.
\end{prop}
\begin{proof}
Suppose $\lCreg(\ups_1, \ups_2) \neq \emptyset$.
First assume $\ups_1 \in \Vmione(q)$.
Let $\ups_0, \tups_0 \in \lCreg(\ups_1, \ups_2)$ and the goal is to prove $\ups_1 \in \lV(\ups_0, \tups_0 )$.
Indeed, by the definition of $\Vmione(q)$, there exists a small neighborhood $N_1 \subseteq \LUmi$ of $\ups_1$ such that
if $\tups_j \in N_1$ are distinct with $q = \gamma_{\tups_j}(\varsm_j)$ for some $\varsm_j>0$, then $(\tups_1, \tups_2, \tups_3) \in \Vmi$.
Now we write
\[
(q,w_1) = (\gamma_{\ups_1}(\varsm_1), \dot{\gamma}_{\ups_1}(\varsm_1))
\quad \text{ and }\quad
\ups_0 = (\gamma_{q,w_0}(\varsm_0), \dot{\gamma}_{q,w_0}(\varsm_0)),
\]
Thus, using the same idea as the proof of Proposition \ref{prop_Cearv1v2},
for any $\tw_1$ near $w_1$, we can find $\tw_3, \tw_3$ near $\tw_1$ such that
\[
w_0 \in \mathrm{span}(\tw_1, \tw_2, \tw_3) \quad \text{ and }\quad  w_0 \notin \mathrm{span}(\tw_j), \text{ for } j = 1,3,4.
\]
By choosing $\tw_1$ sufficiently close to $w_1$, we can assume $\tups_j = (\gamma_{\tw_j}(\varsm_j), \dot{\gamma}_{\tw_j}(\varsm_j)) \in \LUmi$ is contained in $N_1$ of $\ups_1$ and therefore $(\tups_1, \tups_2, \tups_3) \in \Vmi$.
As $\ups_0 \in \Creg_{U_+}(q)$ by Proposition \ref{pp_CregU}, Condition (R2) implies $(\ups_0, \tups_1, \tups_2, \tups_3) \in R$.
Similarly, we have $(\ups_0, \tups_1, \tups_4, \tups_5) \in R$ for some $\tups_4, \tups_5 \in \LUmi$.
Thus, the set $\lV(\ups_0, \tups_0)$ contains a small neighborhood of $\ups_1$ in $C_{U_-}^-(q)$ and it follows $\ups_1 \in C$ for some $C \in \lVsmooth(\ups_0, \tups_0)$.
Indeed, there exists small neighborhood $\tN_1 \subseteq \LUmi$ of $\ups_1$ such that $C \cap \tN_1 = C_{U_-}^-(q) \cap \tN_1$.
Now for any $\tups_j \in C \cap \tN_1$, by the definition, if they are distinct, then $(\tups_1, \tups_2, \tups_3) \in \Vmi$.
This proves $\ups_1 \in \lV_1(\ups_0, \tups_0 )$.

Next, we assume $\ups_1 \in \lV_1(\ups_0, \tups_0 )$.
Suppose $\ups_1 \in C$ for some $C \in \lVsmooth(\ups_0, \tups_0)$.
Note that by Lemma \ref{lm_Vonetwo} we have $\ups_1 \in \lV(\ups_0, \tups_0) \subseteq C_{U_-}^-(q)$, which is a smooth $(n-1)$-dimensional submanifold.
Thus, there exists small neighborhood $N_1 \subseteq \LUmi$ of $\ups_1$ such that $C \cap N_1 = C_{U_-}^-(q) \cap N_1$.
Now for any
distinct $\tups_j \in N_1$ with $q = \gamma_{\tups_j}(\varsm_j)$ for some $\varsm_j>0$, we have
$\tups_j \in C \cap N_1$.
This implies $(\tups_1, \tups_2, \tups_3) \in \Vmi$ and therefore $\ups_1 \in \Vmione(q)$.
\end{proof}
We are ready to prove the following results for the reconstruction.
\begin{prop}
    Suppose there are no cut points along any geodesic segments in $\overline{W}$.
    Suppose $\Vmi$ is sufficient for $W$.
    We consider the set
    \[
    \begin{split}
     \lV_- \coloneqq \{
     (\ups_1, \ups_2) \in (\LUmi)^2:
     &\text{ $(\ups_0', \ups_1, \ups_2, \ups_3') \in R$ is proper for some $\ups_0' \in \LUpl$ and $\ups_3' \in \LUmi$}\\
     &\quad \quad \quad \quad \quad \quad \quad \text{ and $\ups_1, \ups_2 \in \lV_1(\ups_0, \tups_0)$ for some $\ups_0, \tups_0 \in \lCreg(\ups_1, \ups_2)$}
     \}.
    \end{split}
    \]
%
    Then we have
    \[
    \begin{split}
    \{\lCreg(\ups_1, \ups_2): (\ups_1, \ups_2) \in \lV_-\}
    =\{\Creg_{U_+}(q): \text{$q \in \overline{W}$}\}.
    \end{split}
    \]
\end{prop}
\begin{proof}
    We denote the left-hand side by $S_1$ and the right-hand side by $S_2$.
    First, we prove $S_1 \subseteq S_2$.
    Let $(\ups_1, \ups_2) \in \lV_-$, which by definition implies there exist $\ups_0, \tups_0 \in \lCreg(\ups_1, \ups_2)$ such that $\ups_1, \ups_2 \in \lV_1(\ups_0, \tups_0)$.
    As $(\ups_0', \ups_1, \ups_2, \ups_3') \in R$ for some $\ups_0' \in \LUpl$ and $\ups_3' \in \LUmi$,
    we can find a unique $q$ such that
    $\gamma_{\ups_1}(\mR_+) \cap \gamma_{\ups_2}(\mR_+) \cap \overline{W} = \{q\}.$
    Then by Proposition \ref{pp_determineV}, we have
    $\ups_1, \ups_2 \in \Vmione(q)$.
    As $\Vmi$ is sufficient for $W$, by Definition \ref{def_sufficientV}, the lightlike vectors in $\Vmi$ are corresponding to null geodesics that do not intersect before $W$.
    Using Proposition \ref{pp_CCU}, we have $\lCreg(\ups_1, \ups_2) = \Creg_{U_+}(q)$ and therefore $\lCreg(\ups_1, \ups_2) \in S_2$.

    Next, we prove $S_2 \subseteq S_1$.
    Let $q \in \overline{W}$.
    As $\Vmi$ is sufficient for ${W}$, the set $\Vmione(q)$ is nonempty.
    One can choose $\ups_1\in \Vmione(q)$ and we write
    $(q,w_1) = (\gamma_{\ups_1}(\varsm_1), \dot{\gamma}_{\ups_1}(\varsm_1))$ for some $\varsm_1 > 0$.
    We want to prove $\Creg_{U_+}(q) = \lCreg(\ups_1, \ups_2)$ for some   $\ups_2 \in \Vmione(q)$.
    With $\Creg_{U_+}(q)$,  there exists $w_0 \in L_q^+M$ such that $\gamma_{q, w_0}(\varsm_0) \in U_+$ for some $0< \varsm_0 < \rho(q, w_0)$.
    In addition, we can perturb $w_0$ such that $w_0 \notin \mathrm{span}(w_1)$, since by Lemma \ref{lm_cutneighborhood} again we still have $\pi(w_0)$ before the first cut point of $q$ under small perturbation.
    Then we set $\ups_0 = (\gamma_{q, w_0}(\varsm_0), \dot{\gamma}_{q, w_0}(\varsm_0))$ and by Lemma \ref{lm_nulltransversal}, we have $\ups_0 \in \LUpl$.
    Using \cite[Lemma 6.23]{feizmohammadi2021inverse} again, we can find $w_j \in L_q^+M$ sufficiently close to $w_1$ such that
    \[
    w_0 \in \mathrm{span}(w_1, w_2, w_3) \quad \text{ and }\quad  w_0 \notin \mathrm{span}(w_j), \text{ for } j = 1,2,3.
    \]
    We set
    \[
    \text{$\ups_j = (\gamma_{q,w_j}(\varsm_j), \dot{\gamma}_{q,w_j}(\varsm_j))$  such that $\gamma_{q,w_j}(\varsm_j) \in U_-$},
    \]
    for some $\varsm_j < 0$, where $j = 2,3$.
    As before, by choosing $w_j$ close to $w_1$, we can assume $\ups_j$ are sufficiently close to $\ups_1$ and therefore $(\ups_1, \ups_2, \ups_3) \in \Vmi$, as $\ups_1 \in \Vmione$.
    Then by Lemma \ref{lm_nulltransversal} again $\ups_0 \in \LUpl$ and
    therefore $(\ups_0, \ups_1, \ups_2, \ups_3) \in R$.
    Note such $\gamma_{\ups_j}$ are not identical and cannot intersect at other points in $W$ by Lemma \ref{lm_finiteF}, for $j =0,1,2,3$.
    Then such $\ups_1, \ups_2 \in \Vmione(q)$ and
    by Proposition \ref{prop_CearU}, we have $\Creg_{U_+}(q) = \lCreg(\ups_1, \ups_2) \in S_1$.
\end{proof}
To conclude, we state the following theorem.
\begin{thm}\label{thm_3to1}
    Let $\Vmi$ be sufficient for $W \subseteq I(U_-, U_+)$, see Definition \ref{def_sufficientV}.
    Suppose there no cut points along any null geodesic segments in $\overline{W}$.
    Let $R \subset \LUpl \times \Vmi$ be a three-to-one scattering relation as in Definition \ref{def_3to1}.
    Then $\Vmi$ and $R$ determines the set
    \[
    \{\Creg_{U_+}(q): \text{$q \in \overline{W}$}\}.
    \]
\end{thm}

%% file: detection_condition.tex
\subsection{Detection of singularities}\label{sec_detection}
In this part, we explain how to extract a three-to-one scattering relation as in Definition \ref{def_3to1}, from observed new singularities of nonlinear wave interaction.
In the following, we focus on the simplest case, i.e., the detection of singularities in Step 1 of Section \ref{sec_layer}.
Recall in Section \ref{subsec_constructwaves},
for given lightlike vectors $\ups_j = (p_j, w_j) \in \LUmi$,
we construct conormal distributions
$\Ups_j \in {{I}^{\mu +1/4} (\Sigma({p_j, w_j, s_0}))}$ and let $v_j$ be the solution to
\[
\sq_g v_j = 0 \text{ in }M \quad \text{ with } v_j|_{U_-} = \Ups_j, \quad \text{ where } j = 1,2,3.
\]
Note that $v_j \in {I}^{\mu}(\Lambda_j({p_j, w_j, s_0}))$ are conormal distributions  before the first cut point, for $j = 1,2,3$.
Now to extract a three-to-one scattering relation using the scattering operator, we introduce the following singularities detection condition as in \cite{Kurylev2018}, see also a modified version for the boundary value problems in \cite{Uhlmann2021a}.

\begin{df}\label{def_D}
    Let $\Vmi \subset (\LUmi)^3$ be sufficient for $W = I(U_-, U_+)$.
    We say a vector $\ups_0 \in \LUpl$ satisfies the \textbf{condition (D)}  with light-like vectors $(\ups_1, \ups_2, \ups_3) \in \Vmi$ and $\hat{\ka} >0$,
    if
    \begin{itemize}
        \item[(a)] $\ups_0, \ups_1, \ups_2, \ups_3$ are proper as in Definition \ref{def_proper}, and
        \item[(b)] for any $0< \ka_0 < \hat{\ka}$ and $j = 1, 2, 3$,  there exists $\Ups_j \in { I^{\mu+1/4}(\Sigma(p_j,w_j,\ka_0))}$
        satisfying the support condition (\ref{assump_Upsj}),
        where $\ups_j = (p_j, w_j)$,
    \end{itemize}
    such that $(\yb, \etab)$
     is contained in  $\wfset(\partial_{\epsilon_1}\partial_{\epsilon_2}\partial_{\epsilon_3} \mcn(\ep_1 \Ups_1 + \ep_2 \Ups_2 + \ep_3 \Ups_3) |_{\epsilon_1 = \epsilon_2 = \epsilon_3=0})$, where $(\yb, \etab)$ is the restriction of $\ups_0^\sharp = (y, \eta)$ to $T^*\Spl$.
\end{df}
By this singularity detection condition, we can tell if singularities produced by the interaction of three conormal waves emanating from $(p_j, w_j)_{j=1}^3$ at the past null infinity can be detected at the corresponding covector of $(p_0, w_0)$ at the future infinity,
where $(p_j, w_j)_{j=1}^3$ are given by $(\ups_1, \ups_2, \ups_3) \in \Vmi$ and $(p_0, w_0)$ is given by $\ups_0 \in \LUpl$.
We introduce the following lemmas.
These lemmas may be slightly more general than Step 1 as we would like to use them for later steps.
\begin{lm}\label{lm_Nk_0}
    Let $\ups_1,\ups_2 \in \LUmi$ be distinct and let $F = \gamma_{\ups_1}(\mR_+) \cap \gamma_{\ups_2}(\mR_+) \cap J(i_-, i_+)$.
    Note that $F$ must be finite.
    Then we have the following statements.
    \begin{itemize}
        \item[(1)] Suppose $F =  \{q_1, \ldots, q_m \}$.
        Let $N_k$ be disjoint small open neighborhoods of $q_k \in F$, for $k = 1, \ldots, m$.
        Suppose $\xxiip$ is in the $\ka$-neighborhood of $\xxii$, for $i=1,2$.
        Then with sufficiently small $\ka > 0$, we have
        \[\gamma_{\ups'_1}(\mathbb{R_+}) \cap  \gamma_{\ups'_2}(\mathbb{R_+}) \cap J(\imi, \ipl)  \subseteq \cup_{k=1}^m N_k.\]
        \item[(2)] If $F = \emptyset$, then with sufficiently small $\ka > 0$, we have
        \[\gamma_{\ups'_1}(\mathbb{R_+}) \cap  \gamma_{\ups'_2}(\mathbb{R_+}) \cap J(\imi, \ipl)  = \emptyset.\]
    \end{itemize}
\end{lm}
\begin{proof}
    We prove this by contradiction.
    To prove (1), assume this is not true.
    Then we can find two sequences $\ups_{i,j} \in L^+\Omgin, i=1,2$ such that
    $\ups_{i,j}$ is in the $\ka_{1,j}$-neighborhood of $\ups_i$
    with
    \[
    q_j' \in \gamma_{\ups_{1, j}}(\mathbb{R}_+) \cap \gamma_{\ups_{2, j}}(\mathbb{R}_+)\cap J({\imi, \ipl}) \quad \text{ and }\quad  q_j' \in J({\imi, \ipl}) \setminus (\cup_{k=1}^m N_k).
    \]
    Since $J({\imi, \ipl})$ is compact, the sequence $\{q_j'\}$ has a subsequence converging to some point $q'$ and we abuse the notation to denote it still by $\{q_j'\}$.
    Since $J({\imi, \ipl}) \setminus (\cup_{k=1}^m N_k)$ are closed, we have $q' \in J({\imi, \ipl}) \setminus (\cup_{k=1}^m N_k)$.
    For each $j$, since $q_j'$ is the intersection point,
    we can write $q_j' = \gamma_{\ups_{i,j}}(\varsm_{i,j})$ for $i = 1,2$ for some $\varsm_{i,j} \geq 0$.
    Note that $q_j'  = \gamma_{\ups_j}(\varsm_{i,j}) \rightarrow q'$.
    By \cite[Lemma 6.4]{feizmohammadi2021inverse}, there exists $\varsm_i$ such that $\varsm_{i,j} \rightarrow \varsm_i$ for $i=1,2$.
    On the other hand, since the geodesic flow are continuous, we have $ \gamma_{\ups_{i,j}}(\varsm_{i,j}) 
    \rightarrow \gamma_{\ups_i}(\varsm_i)$ for $i=1,2$, as $j \rightarrow +\infty$.
    It follows that $q' = \gamma_{\xxii}(\varsm_i) \in F$, which contradicts with the fact that
    $q' \in  J({\imi, \ipl}) \setminus (\cup_{k=1}^m N_k)$.

    To prove (2), we can choose $N_k = \emptyset$ for $k = 1, \ldots, m$ and repeat the same analysis as above to get a point $ p = \gamma_{\xxii}(t_i)$ in $J({\imi, \ipl})$, $i=1,2$.
    This contradicts with $F = \emptyset$.
\end{proof}
\begin{lm}\label{lm_Nk}
    Let $\ups_1, \ups_2 \in \LUmi$ be distinct.
    Let $F = \gamma_{\ups_1}(\mathbb{R}_+) \cap \gamma_{\ups_2}(\mathbb{R_+}) \cap J(\imi, \ipl) =  \{q_1, \ldots, q_m \}$ and we write $\ups_j = (p_j, w_j)$, for $i=1,2$.
    Let $N_k$ be disjoint small open neighborhoods of each $q_k \in F$.
    Then with $\ka_0 > 0$ sufficiently small, we have
    \[
    K({p_1}, {w_1}, \ka_0) \cap  K(x_2, w_2, \ka_0)\cap J(\imi,\ipl) \subseteq \textstyle \bigcup_{k=1}^m N_k.
    \]
    If $F = \emptyset$, then $\ka_0 > 0$ sufficiently small $K({p_1}, {w_1}, \ka_0) \cap  K(x_2, w_2, \ka_0)\cap J(\imi, \ipl) = \emptyset$.
\end{lm}
\begin{proof}
    In the following, we denote 
    $K({p_i}, w_i, \ka_0)$ by $K_i$ and ${W}({p_i}, w_i, \ka_0)$ by ${W}_i$, for $i = 1,2$.
    We notice that if $q' \in K_1 \cap K_2$, then there exist
    $(p_i', {w}_i') \in {W}_i$ satisfying $q' \in \gamma_{{p_1'}, {w}_1'}(\mathbb{R}_+) \cap \gamma_{p_2', {w}_2'}(\mathbb{R}_+)$.
    Therefore, it suffices to show that with small enough $\ka_0$, we have
    \[\gamma_{{p_1'}, w'_1}(\mathbb{R}_+) \cap \gamma_{p_2', w_2'}(\mathbb{R}_+) \cap J(\imi, \ipl) \subseteq \textstyle \bigcup_{k=1}^m N_k,
    \]
    for any $(p_i', w_i') \in {W}_i$, for $i=1,2$.
   This is true according to Lemma \ref{lm_Nk_0}.
\end{proof}
\begin{lm}\label{lm_lightlikeprop}
    Let $(p_j, w_j)_{j=1}^3 \subseteq (\LUmi)^3$ be distinct.
    If $q' \in W$ such that $(q', \zj) \in \Lambda(p_j, w_j, s_0)$, for $j=1,2,3$,
    then with sufficiently small $\ka_0 > 0$
    we have
    \begin{itemize}
        \item[(1)] $\zi + \zj \neq 0$, for $ 1\leq i < j \leq 3$,
        \item[(2)] $\zeta^{1} + \zeta^{2} + \zeta^{3} \neq 0$.
    \end{itemize}
\end{lm}
\begin{proof}
    We claim that with $\ka_0 > 0$ small enough, we can find $q \in {\cap_{j=1}^{3}\gamma_{{{p_i}, w_i}}(\mathbb{R_+})}$ and $q'$ is in a small neighborhood of $q$.
    Indeed, let $F = {\cap_{j=1}^{3}\gamma_{{{p_j}, w_j}}(\mathbb{R_+})} \cap J({\imi, \ipl})$ and we can assume $F = \{q_1, \ldots, q_m \}$.
    Let $N_k$ be disjoint arbitrarily small open neighborhoods of each $q_k \in F$.
    Then Lemma \ref{lm_Nk} implies $q' \in N_k$ for some $k \in \{1, \ldots, m\}$, with $\ka_0$ sufficiently small.
    Now suppose $q_k = \gamma_{{{p_j}, w_j}}(\varsm_j)$ for each $j = 1,2,3$.
    We note that $\dot{\gamma}_{{{p_ij}, w_j}}(\varsm_j) \in L_{q_k}^+M$ for $j = 1,2,3$ are not multiples of each other,
    which implies
    \[
    c_i \dot{\gamma}_{{p_i}, w_i}(\varsm_i) + c_j \dot{\gamma}_{x_j, w_j}(\varsm_j) \neq 0,
    \quad \text{if }c_i, c_j \neq 0, \text{ for }i \neq j \in \{1,2,3\}.
    \]
    and by \cite[Lemma 6.21]{feizmohammadi2021inverse} one has
    \[
    c_1 \dot{\gamma}_{{p_1}, {w_1}}(\varsm_1) + c_2 \dot{\gamma}_{p_2, w_2}(\varsm_2) + c_3 \dot{\gamma}_{p_3, w_3}(\varsm_3) \neq 0,
    \quad \text{if }c_1, c_2, c_3 \neq 0.
    \]
    With $\ka_0$ small enough,  the lightlike covector $(q', \zj)$ are sufficiently small perturbations of lightlike covectors
    $(q_k, (c_j \dot{\gamma}_{p_j, w_j}(t_j))^b )$ for some nonzero constant $c_j$, where $j = 1,2,3$.
    Thus, we have the desired result.
\end{proof}
The following lemma is the result of the propagation of singularities, for example, see \cite[Theorem 2.1]{Taylor2017}.
\begin{lm}\label{lm_propagation}
Let $(p, \eta) \in L^{*, +}M$ and $v$ be the solution to
\[
(\sq_g + V) v = f \text{ in }M,  \quad \text{ with } \lR_-[v]  = 0.
\]
    Then $(p, \eta) \in \wfset(v)$ only if $(y, \eta) \in \wfset(f)$ or there exists $(q, \zeta)$ such that
    $(q, \zeta) \in \wfset(f)$ and $(p, \eta)$ is contained in the null bicharacteristics starting from $(q, \zeta)$.
\end{lm}
Assuming no cut points in $\overline{W}$, the following proposition is a direct result of the product of conormal distributions and the lemmas above.
Moreover, by our construction in Step 1, null geodesics starting from lightlike vectors in $\LUmi$ and ending at lightlike vectors in $\LUpl$ only intersect in $\overline{W}$.
\begin{prop}\label{pp_R1}
    Suppose there are no cut points along any geodesic segments in $\overline{W}$.
    Let $(\ups_0, \ups_1, \ups_2, \ups_3) \in \LUpl \times \Vmi$ be proper.
    If one has
    \[(\capgammathree) \cap \overline{W}= \emptyset,\] then
    $\ups_0$ does not satisfy the condition (D) with $(\ups_1, \ups_2, \ups_3)$ and small $\kappa_0 >0$.
\end{prop}
\begin{proof}
In the following, we write $\ups_j = (p_j, w_j) \in \LUpl$ or $\LUmi$, for $j = 0, 1,2, 3$.
To check the condition (D), recall $v_j \in I^\mu(\Lambda(p_j, w_j, \kappa_0)), j=1,2,3$ are Lagrangian distributions in $M$.
With no cut points in $\overline{W}$, they are actually conormal distributions.
Recall $\lU_3$ is the solution to
\[
\square_g \lU_3 -  6\alpha v_{1}v_{2}v_{3} = 0 \text{ in } M, \quad \text{ with }
\mathcal{R}_-[\lU_3] = 0.
\]
The goal is to show $(y, \eta) = (p_0, w_0^\flat) \notin \wfset(\lU_3)$.
Recall we assume $(\capgammathree) \cap \overline{W}= \emptyset$.
    First, we consider the case where locally there are exactly two null geodesics
    intersecting at one point.
    Without cut points, two null geodesics intersect at most one point in $\overline{W}$.
    Suppose $\gamma_i(\mR_+)$ and $\gamma_j(\mR_+)$ intersect at $q$ with $q \notin \gamma_k(\mR_+)$,
    where $\{i,j,k\} = \{1,2,3\}$.
    With $\kappa_0$ small enough, we may assume $q \notin K_k$.
    Since $v_k$ is smooth near $q$, the wave front set of $v_i v_j v_k$ at $p$ is in the span of $\Lambda_i$ and $\Lambda_j$, by the H\"{o}rmander-Sato Lemma, for example see \cite[Theorem 1.3.6]{Duistermaat2010}.
    Next, we consider the case that none of three geodesics intersect with each other.
    Then we have $\wfset(\alpha v_i v_j v_k) \subseteq \cup_{j=1}^3 \Lambda_j$.
    In both cases,
    we have
    \[
    \wfset(\alpha v_1v_2v_3) \in \Lambda^{(1)} \cap \Lambda^{(2)},
    \]
    where these notations are defined in Section \ref{subsec_constructwaves}.
    Note that if $(p, \eta) \in \wfset(\lU_3)$, then we must have either $(p, \eta) \in \wfset(\alpha v_1v_2v_3)$ or  $(q, \zeta) \in \wfset(\alpha v_1v_2v_3)$ such that $(p, \eta)$ is in the null bicharacteristics starting from ${(q, \zeta)}$ by Lemma \ref{lm_propagation}.

    For the first situation, we must have $(p,\eta)$ contained in $\Lambda_j$ for some $j = 1,2,3$, which is impossible for sufficiently small $\kappa_0$, as $(\ups_0, \ups_1, \ups_2, \ups_3)$ is proper.
    For the second situation, first we consider if there exists $(q, \zeta) \in \Lambda_j$ for some $j \in \{1,2,3\}$. 
    This is impossible, otherwise we have $(p, \eta) \in \Lambda_j$ again.
    Next, we consider if $(q, \zeta) \in \Lambda_{ij}$ for some $i,j \in \{1,2,3\}$.
    If so, then we must have $\zeta = \zi + \zj$ with $(q, \zi) \in \Lambda_i$ and $(q, \zj) \in \Lambda_j$.
    Recall that if a covector $\zeta = \zi + \zj$ is lightlike, then we must have $\zeta$ is propositional to either $\zi$ or $\zj$ by Lemma \ref{lm_lightlikeprop}.
    Then $(q, \zeta)$ is contained in either $\Lambda_i$ or $\Lambda_j$.
    By the analysis above, this is impossible.
    Thus, we can conclude that $(p, \eta) \notin \wfset(\lU_3)$ for small $\kappa_0$.
    Therefore, the condition (D) is not satisfied.
\end{proof}
Next, by Proposition \ref{pp_U3}, we have the following result.
\begin{prop}\label{pp_R2}
    Suppose there are no cut points along any geodesic segments in $\overline{W}$.
    Let $(\ups_0, \ups_1, \ups_2, \ups_3) \in \LUpl \times \Vmi$ be proper.
    Suppose there exists
    \[
    q \in \Gammaset \cap \overline{W}
    \]
    such that $q = \gamma_{\upsilon_j}(\varsm_j)$ with $\varsm_j >0$ for $j=1,2,3$ and
    \[
    q = \gamma_{\upsilon_0}(\varsm_0) \text{ with }
    \varsm_0 \in (-\rho(\upsilon_0^{-}), 0),
    \]
    where $\upsilon_0^- \coloneqq (y, -w) \in L^-_{U_+}M$ if we write $\upsilon_0 = (y, w)$.
    In addition, assume
    \[
    \dot{\gamma}_{\upsilon_0}(\varsm_0) \in \text{span}(\dot{\gamma}_{\upsilon_1}(\varsm_1), \dot{\gamma}_{\upsilon_2}(\varsm_2), \dot{\gamma}_{\upsilon_3}(\varsm_3)).
    \]
    Then $\ups_0$ satisfies the condition (D) with $(\ups_1, \ups_2, \ups_3)$ and small enough $\hat{\kappa} >0$.
\end{prop}

Now we define the following relation
\[\begin{split}
    R_s &= \{(\ups_0, \ups_1, \ups_2, \ups_3)
    \in \LUpl \times \Vmi:
    \text{ they are proper and there exists $\hat{\kappa} > 0$ }\\
    &\quad \quad \quad \quad \quad \quad\quad \quad \quad\quad \quad \quad\quad \quad \quad
    \text{ such that (D) is} \text{ valid for }
    \ups_0 \text{ with } (\ups_1, \ups_2, \ups_3) \text{ and } \kappa_0
    \}.
\end{split}\]
We verify that $R_s$ is a three-to-one scattering relation by the following proposition.
\begin{prop}\label{pp_RL}
    The relation $R_s$ defined above is a three-to-one scattering relation for $W$.
    Moreover,
    using the nonlinear scattering map, one can determine whether $(\ups_0, \ups_1, \ups_2, \ups_3) \in \LUpl \times \Vmi$ is in the relation $R_s$.
\end{prop}
\begin{proof}
    Let $(\ups_0, \ups_1, \ups_2, \ups_3) \in \LUpl \times \Vmi$.
    One can check if $(\ups_0, \ups_1, \ups_2, \ups_3)$ is proper by using the lens relation determined by the scattering operator.
    For proper $(\ups_0, \ups_1, \ups_2, \ups_3)$ and $R_s$ defined above,
    Proposition \ref{pp_R1} shows the first condition (R1) in Definition \ref{def_3to1} is satisfied
    and Proposition \ref{pp_R2} shows the second condition (R2) is satisfied.
    In particular, the condition (D) is determined by the nonlinear scattering operator.
\end{proof}

%% file: scheme2.tex
\subsubsection{Scheme 2}\label{sec_scheme2}
We follow a similar idea in the reconstruction of scattering light observation sets for Step 2.
Recall in Step 2, we construct
small open neighborhoods $U_\pm \subseteq S_\pm$ of null geodesic segments
and we write
\[
I(U_-, U_+) = W \cup P \cup W_0,
\]
where
\[
W = I(U_-, U_+) \setminus \Jmi(\Spl(T_1)),
\quad \tUmi = I(U_-, U_+) \cap \Pmi(T_1),
\quad W_0 = I(U_-, U_+) \cap I^-(\Spl(T_1)).
\]
Recall $\tUmi$ and $W_0$ are contained in the reconstructed region $(I(T_1), \hat{g}|_{I(T_1)})$.
Moreover, by choosing small $T_2 > 0$, the region of interest
\[
\overline{W}  \subseteq \lB(p_0, \delta)
\]
has no cut points along any null geodesic segments.
Now let $\phi_s: L^+M \rightarrow L^+M$ be the null geodesic flow for $s \in \mR$.
For each $(p, w) \in \LUmi$, there exists at most one $(\tp,\tw) \in  L^+M$ with $\tp \in \tUmi$ such that $(\tp, \tw) = \phi_s(p,w)$ for some $s \in \mR_+$,
where $\LUmi$ is defined (\ref{def_LUpl}).
This defines a restricted lens relation
\[
\lL: \LUmi \rightarrow L_{\tUmi}^+M,
\]
where we denote by $L_{\tUmi}^+M = \{\tups \in L^+M: \pi(\tups) \in \tUmi\}$.
We emphasize such $\lL$ is determined by the reconstructed $\hat{g}$ in $I(T_1)$.
We would like to derive a localized three-to-one scattering relation  $R \subseteq \LUpl \times \Vmi$ for a properly chosen subset $\Vmi\subseteq (L(\tUmi))^3$.
For this purpose,
we consider
\beq\label{def_setVmi}
\begin{split}
    \Vmi = \{(\ups_1, \ups_2, \ups_3) \in  (\LUmi)^3: &\text{ $\gamma_{\ups_i}(\mR_+) \cap \gamma_{\ups_j}(\mR_+) \cap \overline{W}_0 = \emptyset$ for $1 \leq i \leq j \leq 3$}\\
    &\quad \quad \quad \quad \quad \quad \quad \quad \text{and $\pi(\lL(\ups_j)) \in \tUmi$ for $j = 1,2,3$}
    \}.
\end{split}
\eeq
Indeed, as we have reconstructed $\hat{g}$ near $W_0$, we can determine if  $\ups_j(\mR_-)$ intersect there and if $\pi(\lL(\ups_j)) \in \tUmi$.
Then it remains to prove such $\Vmi$ is sufficient for $W$, see Definition \ref{def_sufficientV}.
This comes from the fact that we can construct null geodesics starting from $\tUmi$ and passing $q$ before the first cut point, for any $q \in W$.
We prove the following lemmas.
\begin{lm}\label{lm_nullnocut}
    Let $q \in W$.
    There exists  $\ups_1 \in \LUmi$ such that
    \[
    (q, w_1) = (\gamma_{\ups_1}({\varsm_1}), (\dot{\gamma}_{\ups_1}({\varsm_1}))) \quad \text{for some $0 <\varsm_1< \rho(\ups_1)$}
    \]
    and moreover $\gamma_{\ups_1}(\mR_+)$ hits $\tUmi$ exactly once.
\end{lm}
\begin{proof}
    Recall $U_- \subset \Spl$ is a small neighborhood of the null geodesic segment from $\pmi'$ to $p_0'$.
    We can find a timelike smooth curve in $\tM \setminus M$ sufficiently close to this segment and use \cite[Lemma 3.5]{Kurylev2018} to construct such null geodesic $\ups_1$.
    We can expect $\ups_1(\mR_+)$ is contained in $I(U_-, U_+)$ and by our construction it must intersects once with the achronal boundary $\tU_+$.
\end{proof}
\begin{lm}
    The set $\Vmi$ that we defined in (\ref{def_setVmi}) is sufficient for $W$ by Definition \ref{def_sufficientV}.
    In particular, for each $q \in W$, the set $\Vmione(q)$ is nonempty.
\end{lm}
\begin{proof}
    By (\ref{def_setVmi}), the condition (a) in Definition \ref{def_sufficientV} is satisfied.
    It suffices show $\Vmione(q)$ is nonempty for each $q \in W$.
    For each $q \in W$, we choose $\ups_1 \in L(U_-)$ as in Lemma \ref{lm_nullnocut}.  Note $\gamma_{\ups_1}(\mR_+)$ hits $\tUmi$ exactly once, which implies $\pi(\lL(\ups_j)) \in \tUmi$.
    We claim $\ups_1 \in \Vmione(q)$.
    Indeed, by Lemma \ref{lm_nullnocut}, such $\ups_1$ satisfies
    \[
    (q, w_1) = (\gamma_{\ups_1}({\varsm_1}), (\dot{\gamma}_{\ups_1}({\varsm_1}))) \quad \text{for some $0 <\varsm_1< \rho(\ups_1)$}.
    \]
    By the proof of Lemma \ref{lm_cutneighborhood}, there exists a small neighborhood $N \subseteq L_q^+M$ of $w_1$ such that for any $w\in N$,
    the null geodesic $\gamma_{q, w}(\varsm)$ hits $U_-$ at $\varsm_q < 0$ before the first cut point of $q$.
    We set $\ups = (\gamma_{q, w}(\varsm_q), \dot{\gamma}_{q, w}(\varsm_q))$ and this gives us a small neighborhood of $\ups_1$ in $\LUmi$.
    Thus, there exists $N_1 \subset \LUmi$ of $\ups_1$ such that for any $\ups \in N_1$ with $q \in \gamma_{\ups}(\mR_+)$, we have $q$ is before the first cut point of $\pi(\ups) \in U_-$.
    Now we pick arbitrary $\tups_j \in N_1$ that are distinct and satisfy $q = \gamma_{\tups_j}(\varsm_j)$ for some $\varsm_j > 0$, for $j = 1,2, 3$.
    Note that such $\gamma_{\ups_j}(\mR_+)$ cannot intersect before $U_-$, otherwise a shortcut argument shows $q$ is on or after the first cut point.
    Thus, we have $(\tups_1, \tups_2, \tups_3) \in \Vmi$.
\end{proof}
Next, we introduce the following  singularities detection condition as a modified version for Step 1.
Note for $(p_j, w_j) \in \LUmi, j = 1,2,3$,
instead of constructing $\Ups_j$ as conormal distributions on $U_-$ as before, we would like to choose $\Gamma_j$ such that $v_j$ solving
\[
\sq_g v_j = 0 \text{ in }M \quad \text{ with } v_j|_{U_-} = \Gamma_j
\]
are conormal waves in $\overline{W}$.
For this purpose, we consider the following construction in the known region $W_0$.
For fixed $(\ups_1, \ups_2, \ups_3) \in \Vmi$, we consider the null geodesics $\gamma_{\ups_j}(\mR_+)$ and set
\[
q_j = \gamma_{\ups_j}(\varsm_j^-) = \pi(\lL(\ups_j)) \in \tUmi.
\]
We choose small $\ep_0 > 0$ such that
\[
(x_j, w_j) = (\gamma_{\ups_j}(\varsm_j^--2\ep_0), \dot{\gamma}_{\ups_j}(\varsm_j^- -2\ep_0))
\]
are lightlike covectors contained in $\lB(p_0, \delta)$ given by Lemma \ref{lm_radius}.
In this case, we define
\[
\begin{split}
    \mathcal{W}({x_j, w_j, \kappa_0}) &= \{w \in L^+_{x_j} M: \|w - w_j\|_{g^+} < \kappa_0 \text{ with } \|w \|_{g^+} = \|w_j\|_{g^+}\}
\end{split}
\]
as a neighborhood of $w_j$ at the point $x_j$.
We define
\[
\begin{split}
    K({x_j, w_j, \kappa_0}) &= \{\gamma_{x_j, w}(\varsm) \in M: w \in \mathcal{W}({x_j, w_j, \kappa_0}), \  \varsm\in (0, \infty) \}
\end{split}
\]
as the subset of the light cone emanating from $x_j$ by light-like vectors in $\mathcal{W}({x_j, w_j, \kappa_0})$.
As $\kappa_0$ goes to zero, the surface $K({x_j, w_j, \kappa_0})$ tends to the geodesic $\gamma_{x_j, w_j}(\mathbb{R}_+)$.
Consider the Lagrangian submanifold
\[
\begin{split}
    \Sigma(x_j, w_j, \kappa_0) =\{(x_j, r w^\flat )\in T^*\tM: \eta \in \mathcal{W}({x_j, w_j, \kappa_0}), \ r\neq 0 \},
\end{split}
\]
which is a subset of the conormal bundle $N^*\{x_0\}$.
We define
\[
\begin{split}
\Lambda({x_j, w_j, \kappa_0})
= &\{(\gamma_{x_j, w}(s), r\dot{\gamma}_{x_j, w}(s)^\flat )\in T^*\tM: \\
& \quad \quad \quad \quad \quad
w \in \mathcal{W}({x_j, w_j, \kappa_0}), \ \varsm\in (0, \infty),\  r \in \mathbb{R}\setminus \{0 \} \}
\end{split}
\]
as the flow-out from $\Char(\sq_g)$ in the future direction.
Note that $\Lambda({x_j, w_j, \kappa_0})$ is the conormal bundle of $K({x_j, w_j, \kappa_0})$ near $\gamma_{x_j, w_j}(\mathbb{R}_+)$ before the first cut point of $x_j$.

Recall we have reconstructed the metric in $W_0$ up to conformal diffeomorphisms.
We can pick an arbitrary representative in the equivalent class, say $\hat{g} = \phi^*(\rho^2 g)$, where $\rho \in C^\infty(M)$ is a conformal factor.
According to \cite[Lemma 3.1]{Kurylev2018},
we can construct distorted plane waves
\[
\vjone \in \Ical^\mu(\Lambda(x_j, \xi_j, s_0)) \text{ satisfying } \square_{\hat{g}} u_j = f_j \text{ in }M, \quad j = 1, \ldots, 3,
\]
where $f_j \in I^{\mu+{3}/{2}}(\Sigma(x_j, w_j, \kappa_0))$ are sources singular near $(x_j, w_j^\sharp)$.
Such $\vjone$ are Lagrangian distributions with
nonzero principal symbol along $(\gamma_{x_j, w_j}(\mR_+), (\dot{\gamma}_{x_j, x_j}(\mR_+))^\flat )$.
In particular, as there are no cut point in $\overline{W}$, we have $\vjone$ are conormal distributions there.
With $\mu$ negative enough, one has $\vjone \in H^s(M)$ for some $s \geq 2$, for example see \cite[Section 4.1]{Wang2019}.
Now let
\[
T_j \coloneqq T(\gamma_{\ups_j}(\varsm_j^- - \ep)) > T(x_j),
\quad \text{ for } j = 1,2,3,
\]
and consider the reconstructed region
\[
\widetilde{W}_0 = \Jpl(\Smi(T_1)) \cap \{T < T(\gamma_{\ups_j}(\varsm_j^- - \ep))\}.
\]
We can restrict $\vjone$ to the Cauchy surface$\{T = T_j\}$ to get the Cauchy data.
In particular, by choosing small enough $\kappa_0$, we can assume these data are contained in the reconstructed region.
Then we solve the backward problem
\[
\begin{split}
\square_{\hat{g}} \vjtwo  &= 0 \quad \text{ in } \widetilde{W}_0,\\
\vjtwo(T_j) = \vjone(T_j),
&\quad  \partial_t\vjtwo (T_j)= \partial_t\vjone(T_j).
\end{split}
\]
Note such $\vjtwo$ have singularities propagating near $\gamma_{x_j, w_j}(\mR_-)$ and arriving $U_-$ near $\ups_j$.
For small enough $\kappa_0$, we set $\Gamma_j = \vjtwo|_{S_-} \in H^{s-\frac{1}{2}}(S_-)$.
Combing $\vjone$ for $T> T_j$ with $\vjtwo$ for $T < T_j$, we obtain a solution to
\[
\square_{\hat{g}} \tv_j = 0 \text{ in } M, \quad \text{ with } \tv_j|_{S_-} = \Gamma_j, \quad \text{ for } j = 1,2,3.
\]
Note we must have $\tv_j \in I^{\mu}(\Lambda(x_j, w_j, \kappa_0))$ and it is a conormal distribution in $\overline{W}$.
Now we consider the solution $v_j$ to
\[
\square_g v_j = 0 \text{ in } M, \quad \text{ with } v_j|_{S_-} = \Gamma_j, \quad \text{ for } j = 1,2,3.
\]
We emphasize that $\tv_j$ and $v_j$ have the same singularities structure
by (\ref{eq_conormalfactor}) and thus $v_j$ are conormal distributions in $\overline{W}$ as well.
With the construction above we introduce the following detection conditions for Step 2.
\begin{df}\label{def_D_step2}
    Let $\Vmi \subseteq (\LUmi)^3$ be sufficient for $W = I(U_-, U_+)$.
    We say a vector $\ups_0 \in \LUpl$ satisfies the \textbf{condition (D2)} with light-like vectors $(\ups_1, \ups_2, \ups_3) \in \Vmi$ and $\hat{\ka} >0$,
    if
    \begin{itemize}
        \item[(a)] $\ups_0, \ups_1, \ups_2, \ups_3$ are proper as in Definition \ref{def_proper}, and
        \item[(b)] for any $0< \ka_0 < \hat{\ka}$ and $j = 1, 2, 3$,
        {there exists $\Gamma_j$ constructed as above for each $\ups_j$,
        with disjoint the support,
    }
    \end{itemize}
    such that $(\yb, \etab)$
    is contained in  $\wfset(\partial_{\epsilon_1}\partial_{\epsilon_2}\partial_{\epsilon_3} \mcn(\ep_1 \Ups_1 + \ep_2 \Ups_2 + \ep_3 \Ups_3) |_{\epsilon_1 = \epsilon_2 = \epsilon_3=0})$, where $(\yb, \etab)$ is the restriction of $\ups_0^\sharp = (y, \eta)$ to $T^*\Spl$.
\end{df}
As before, we define a relation given by
\[\begin{split}
    R_s &= \{(\ups_0, \ups_1, \ups_2, \ups_3)
    \in \LUpl \times \Vmi:
    \text{ they are proper and there exists $\hat{\kappa} > 0$ }\\
    &\quad \quad \quad \quad \quad \quad\quad \quad \quad\quad \quad \quad\quad \quad \quad
    \text{ such that (D2) is} \text{ valid for }
    \ups_0 \text{ with } (\ups_1, \ups_2, \ups_3) \text{ and } \hat{\kappa}
    \}.
\end{split}\]
We can verify such $R_s$ is indeed a three-to-one scattering relation in Definition \ref{def_3to1}, by Proposition \ref{pp_R1} and \ref{pp_R2}.
Indeed, although the null geodesic $\gamma_{\ups_j}(\mR_+)$ may have conjugate points in $W$, yet with $\Ups_j$ we constructed as above, the linear waves $v_j$ are still conormal distributions in $W$.
Thus, the same analysis in Proposition \ref{pp_R1} and \ref{pp_R2} still work in this case.
After we reconstruct $R_s$,
Theorem \ref{thm_3to1} proves it determines the scattering light observation sets of any point in $W$ restricted to $U_+$.
Then by Theorem \ref{thm_SLOS} again, we reconstruct the metric in $W$ up to diffeomorphisms.

%% file: scheme3.tex
\subsubsection{Scheme 3}\label{sec_scheme3}
Recall in Step 3, we construct small open neighborhoods $U_\pm \subseteq S_\pm$ of null geodesic segments and we write
\[
I(U_-, U_+) = W \cup P \cup W_0,
\]
where
\[
W = I(U_-, U_+) \setminus \Jpl(\Smi(T_1)),
\quad \tUmi = I(U_-, U_+) \cap \Ppl(T_1),
\quad W_0 = I(U_-, U_+) \cap I^+(\Smi(T_1)).
\]
Recall $\tUpl$ and $W_0$ are contained in the reconstructed region $(I^+(\Smi(T_1)), \hat{g}|_{I^+(\Smi(T_1))})$.
Moreover, by choosing small $T_2 > 0$, the region of interest
\[
\overline{W}  \subseteq \lB(p_0, \delta)
\]
has no cut points along any null geodesic segments.
Recall $\tUpl$ is part of the boundary of the future set $\Jpl(\Smi(T_1))$ and thus is an achronal Lipschitz hypersurface.
It separates $W$ and $W_0$ in a chronological way.

For this step, compared to Step 1 and 2, the main difference is that the null geodesics starting from $U_-$ may intersect in $W_0$ after conjugate points.
Then the corresponding receding waves may interact there and produce new singularities, even though they may not interact in $W$.
To distinguish the new singularities produced in $W$ and those in $W_0$,
we would like to propose a new singularities detection condition.
Next, we use Section \ref{subsec_3to1} to reconstruct the regular scattering light observation sets for points in $W$ restricted to $U_+$.
Then Section \ref{sec_SLOS} shows the regular scattering light observation sets determined the metric in $W$.

As Section \ref{subsec_3to1} and Section \ref{sec_SLOS} are ready for this step, it remains for us to modify the new singularities detection condition in Section \ref{sec_detection}.
For this purpose,
first we observe any future pointing null geodesic starting from $\ups \in \LUmi$ either will intersect $\tUpl$, enter $\Jpl(\Smi(T_1))$, and stay there until arrive $S_+$, or will avoid $\tUpl$, enter $\Jpl(\Smi(T_1))$ later, and arrive $S_+$.
Thus, we consider the following restricted lens relation
\[
\lL_1: \LUmi \rightarrow L(S_+),
\]
which relates lightlike vectors on $U_-$ with those on the future null infinity.
Note that $\lL_1$ is an injective map due to the nontrapping property.
Then we consider the
restricted lens relation
\[
\lL_2: L(S_+)  \rightarrow L(\tUpl),
\]
which relates lightlike vectors on $S_+$ with those on the achronal hypersurface $\tUpl$ and may fail to be a map.
{Note that both $\lL_1$ and $\lL_2$ are determined by the conformal class the metric up to rescaling the lightlike vector.}
Next, for each $\ups \in \LUmi$, we consider the set
\[
\lA(\ups) = \{\ups' \in L(\tUpl): \text{ if $\ups' = \lL_2(\lL_1(\ups))$ exists}\},
\]
which contains at most one element.
If $\lA(\ups)$ is empty, it means $\gamma_{\ups}(\mR_+) \cap W_0  = \emptyset$ and thus the corresponding waves we send do not cause the wave interaction in $W_0$.
If $\lA(\ups)=\{\ups'\} \subseteq L(\tUpl)$, it means $\gamma_{\ups}(\mR_+)$ enters $W_0$ from the point $p' = \pi(\ups') \in \tUpl$.
To avoid singularities from the wave interaction in $W_0$, we would like to focus on new singularities  away from the causal future of $p'$.
As the goal is to reconstruct the regular scattering light observation sets, these part of new singularities are sufficient for the reconstruction, as they will always arrive $U_+$ earlier due to a shortcut argument.
More explicitly, we introduce the following  singularities detection condition as a modified version for Step 3 and later.
For Step 3, we actually choose $\Vmi = (\LUmi)^3$.
\begin{df}\label{def_D3}
    Let $\Vmi \subseteq (\LUmi)^3$ be sufficient for $W \subseteq I(U_-, U_+)$.
    We say a vector $\ups_0 \in \LUpl$ satisfies the \textbf{condition (D3)}  with lightlike vectors $(\ups_1, \ups_2, \ups_3) \in \Vmi$ and $\hat{\ka} >0$,
    if
    \begin{itemize}
        \item[(a)] $\ups_0, \ups_1, \ups_2, \ups_3$ are proper as in Definition \ref{def_proper}, and
        \item[(b)] for any $0< \ka_0 < \hat{\ka}$ and $j = 1, 2, 3$,  there exists $\Ups_j \in { I^{\mu+1/4}(\Sigma(p_j,w_j,\ka_0))}$ satisfying the support condition (\ref{assump_Upsj}), where $(p_j, w_j) = \ups_j$, and
        \item[(c)] $\pi(\ups_0) \notin \Jpl(\pi(\ups_j'))$, if there exists $\ups_j' \in \lA(\ups_j)$, for $j = 1,2,3$,
    \end{itemize}
    such that $(\yb, \etab)$
    is contained in  $\wfset(\partial_{\epsilon_1}\partial_{\epsilon_2}\partial_{\epsilon_3} \mcn(\ep_1 \Ups_1 + \ep_2 \Ups_2 + \ep_3 \Ups_3) |_{\epsilon_1 = \epsilon_2 = \epsilon_3=0})$, where $(\yb, \etab)$ is the restriction of $\ups_0^\sharp = (y, \eta)$ to $T^*\Spl$.
\end{df}
As before, we define a relation given by
\[\begin{split}
    R_s &= \{(\ups_0, \ups_1, \ups_2, \ups_3)
    \in \LUpl \times \Vmi:
    \text{ they are proper and there exists $\hat{\kappa} > 0$ }\\
    &\quad \quad \quad \quad \quad \quad\quad \quad \quad\quad \quad \quad\quad \quad \quad
    \text{ such that (D3) is} \text{ valid for }
    \ups_0 \text{ with } (\ups_1, \ups_2, \ups_3) \text{ and } \hat{\kappa}
    \}.
\end{split}\]
We verify such $R_s$ is indeed a three-to-one scattering relation in Definition \ref{def_3to1}, by Proposition \ref{pp_R2} and the following.
We emphasize in Definition \ref{def_3to1}, the condition (R1) requires the intersection point $q$ contained in $\overline{W}$,
which is guaranteed by the detection condition (c) above.
\begin{prop}\label{pp_R1new}
    Suppose there are no cut points along any geodesic segments in $\overline{W}$.
    Let $(\ups_0, \ups_1, \ups_2, \ups_3) \in \LUpl \times \Vmi$ be proper.
    If one has
    \[(\capgammathree) \cap \overline{W}= \emptyset,\] then
    $\ups_0$ does not satisfy the condition (D3) with $(\ups_1, \ups_2, \ups_3)$ and small $\kappa_0 >0$.
\end{prop}
\begin{proof}
    In the following, we write $\ups_j = (p_j, w_j) \in \LUpl$ or $\LUmi$, for $j = 0, 1,2, 3$.
    To check the condition (D3), recall $v_j \in I^\mu(\Lambda(p_j, w_j, \kappa_0))$, $j=1,2,3$ are Lagrangian distributions in $M$.
    With no cut points in $\overline{W}$, they are actually conormal distributions there.
    Recall $\lU_3$ is the solution to
    \[
    \square_g \lU_3 -  6\alpha v_{1}v_{2}v_{3} = 0 \text{ in } M, \quad \text{ with }
    \mathcal{R}_-[\lU_3] = 0.
    \]
    The goal is to show $(y, \eta) = (p_0, w_0^\flat) \notin \wfset(\lU_3)$.
    In particular, we assume $p_0 \notin \Jpl(\pi(\ups_j'))$ if there exists $\ups_j' \in \lA(\ups_j)$, for $j = 1,2,3$.
    Such $\ups_j' = (p_j', w_j')$ are lightlike vectors on the achronal boundary $\tUpl$.

%
    To have $(p, \eta) \in \wfset(\lU_3)$, we must have either $(p, \eta) \in \wfset(\alpha v_1v_2v_3)$ or $(q, \zeta) \in \wfset(\alpha v_1v_2v_3)$ such that $(p, \eta)$ is in the null bicharacteristics starting from ${(q, \zeta)}$ by Lemma \ref{lm_propagation}.
    If there do not exist $1 \leq i < j \leq 3$ satisfying $\gamma_{\ups_i}(\mR_+) \cap \gamma_{\ups_j}(\mR_+) \cap \overline{W}_0 \neq \emptyset$, then the proof of Proposition \ref{pp_R1} implies the desired result.
    Otherwise, the product $\alpha v_1v_2v_3$ may have new singularities produced in ${W}_0$.
    However, such new singularities must be contained in $J^+(p_j')$ for $\ups_j' \in \lA(\ups_j)$.
    With $p_0 \notin J^+(p_j')$, we have $(p, \eta) \notin \wfset(\lU_3)$.
\end{proof}

%% file: scheme4.tex
\subsubsection{Scheme 4}\label{sec_scheme4}
For Step 3, we combine the reconstruction scheme for Step 2 and 3.
Recall we construct
small open neighborhoods $U_\pm \subseteq S_\pm$ of null geodesic segments
and we write
\[
I(U_-, U_+) = W_0 \cup \tUmi \cup W \cup \tUpl \cup W_1.
\]
Recall $\tUmi$, $\tUpl$, $W_0$, and $W_1$ are contained in the reconstructed region.
Moreover, by choosing small $T_3 > 0$, the region of interest
\[
\overline{W}  \subseteq \lB(p_0, \delta)
\]
has no cut points along any null geodesic segments.
We consider the restricted lens relation
\[
\lL_0: \LUmi \rightarrow L_{\tUmi}^+M.
\]
As in Step 2, we define
\[
\begin{split}
    \Vmi = \{(\ups_1, \ups_2, \ups_3) \in  (\LUmi)^3: &\text{ $\gamma_{\ups_i}(\mR_+) \cap \gamma_{\ups_j}(\mR_+) \cap \overline{W}_0 = \emptyset$ for $1 \leq i \leq j \leq 3$}\\
    &\quad \quad \quad \quad \quad \quad \quad \quad \text{and $\pi(\lL_0(\ups_j)) \in \tUmi$ for $j = 1,2,3$}
    \}.
\end{split}
\]
Indeed, as we have reconstructed $\hat{g}$ near $W_0$, we can determine if  $\ups_j(\mR_-)$ intersect there and if $\pi(\lL_0(\ups_j)) \in \tUmi$.
The same proof shows such $\Vmi$ is sufficient for $W$, see Definition \ref{def_sufficientV}.

Next, we consider the following  singularities detection condition as a modified version for Step 2 and 3.
Note for $(p_j, w_j) \in \LUmi, j = 1,2,3$,
instead of constructing $\Ups_j$ as conormal distributions on $U_-$ as before, we could like to choose $\Gamma_j$ such that $v_j$ solving
\[
\sq_g v_j = 0 \text{ in }M \quad \text{ with } v_j|_{U_-} = \Gamma_j
\]
are conormal waves in $\overline{W}$.
Here we consider the same construction as in Scheme 2.
Now we introduce the following detection conditions for Step 4.
\begin{df}\label{def_D_step4}
    Let $\Vmi \subseteq (\LUmi)^3$ be sufficient for $W = I(U_-, U_+)$.
    We say a vector $\ups_0 \in \LUpl$ satisfies the \textbf{condition (D4)} with light-like vectors $(\ups_1, \ups_2, \ups_3) \in \Vmi$ and $\hat{\ka} >0$,
    if
    \begin{itemize}
        \item[(a)] $\ups_0, \ups_1, \ups_2, \ups_3$ are proper as in Definition \ref{def_proper}, and
        \item[(b)] for any $0< \ka_0 < \hat{\ka}$ and $j = 1, 2, 3$,  {there exists $\Gamma_j$ constructed as in Section \ref{sec_scheme2} for each $\ups_j$,
        with disjoint the support, and
        }
        \item[(c)] $\pi(\ups_0) \notin \Jpl(\pi(\ups_j'))$, if there exists $\ups_j' \in \lA(\ups_j)$, for $j = 1,2,3$,
    \end{itemize}
    such that $(\yb, \etab)$
    is contained in  $\wfset(\partial_{\epsilon_1}\partial_{\epsilon_2}\partial_{\epsilon_3} \mcn(\ep_1 \Ups_1 + \ep_2 \Ups_2 + \ep_3 \Ups_3) |_{\epsilon_1 = \epsilon_2 = \epsilon_3=0})$, where $(\yb, \etab)$ is the restriction of $\ups_0^\sharp = (y, \eta)$ to $T^*\Spl$.
\end{df}
As before, we define a relation given by
\[\begin{split}
    R_s &= \{(\ups_0, \ups_1, \ups_2, \ups_3)
    \in \LUpl \times \Vmi:
    \text{ they are proper and there exists $\hat{\kappa} > 0$ }\\
    &\quad \quad \quad \quad \quad \quad\quad \quad \quad\quad \quad \quad\quad \quad \quad
    \text{ such that (D4) is} \text{ valid for }
    \ups_0 \text{ with } (\ups_1, \ups_2, \ups_3) \text{ and } \hat{\kappa}
    \}.
\end{split}\]
We can verify such $R_s$ is indeed a three-to-one scattering relation in Definition \ref{def_3to1}, by Proposition \ref{pp_R1new} and \ref{pp_R2}.
After we reconstruct $R_s$,
Theorem \ref{thm_3to1} shows it determines the scattering light observation sets of any point in $W$ restricted to $U_+$.
Then by Theorem \ref{thm_SLOS} again, we reconstruct the metric in $W$ up to diffeomorphisms.

%% file: asymptotic.tex
\section{Higher-order linearization for general nonlinearity}
In this section, we consider the general nonlinearity and we explain how to use the ideas proposed above to recover the metric in this case.
\subsection{The asymptotic expansion for general nonlinearity}\label{subsec_assyp_higher}
In the following, let $\ep_j > 0$ be small parameters and
let $\Upsilon_j $ for $j = 1,2, 3$.
Let $u$ solve the nonlinear problem
\[
\sq_g u  + \sum_{m=2}^{+\infty}\beta_m(T,X) u^m = 0 \quad \text{ in }M, \quad \text{ with } \lR_-[u] = \ep_1\Upsilon_1 + \ep_2\Upsilon_2+ \ep_3\Upsilon_3,
\]
where $\lR_-[u]$ is the restriction of $u$ to $\Smi$.
We consider the nonlinear scattering map
\[
\mcn(\ep_1\Upsilon_1 + \ep_2\Upsilon_2+ \ep_3\Upsilon_3) = \lR_+[u],
\]
where $\lR_+[u]$ is the restriction of $u$ to $\Spl$.
Let $v_j$ solve the homogeneous linear problem (\ref{eq_vj}) with scattering data $\Ups_j$ on the past null infinity
and recall we write $w  = Q_s(h)$ if $w$ solves the in-homogeneous linear problem (\ref{def_Qs}).

We write $v = \sum_{j=1}^3 \ep_j v_j$ and plug it in the nonlinear equation we have
\[
\sq_g (u -v) = -\sum_{m=2}^{+\infty}\beta_m(T,X) u^m.
\]
We further write
\beq
\begin{split}\label{expand_u}
u = v - \sum_{m=1}{Q_s( \beta_{m} u^{m})}
 = v + {A_2 + A_3 + A_4 + \dots},
\end{split}
\eeq
where we rearrange these terms by the order of $\epsilon$, such that $A_2$ denotes the terms with $\epsilon_i \epsilon_j$,
$A_3$ denotes those with $\epsilon_i \epsilon_j \epsilon_k$, and $A_4$ denotes those with $\epsilon_i \epsilon_j \epsilon_k \epsilon_l$, for $1 \leq i,j,k,l \leq 3$.
By plugging (\ref{expand_u}) into itself, one can find
\[
\begin{split}
A_2 &= Q_s(\beta_2 v^2), \quad \\
A_3 &= Q_s(2\beta_2vA_2 +\beta_3 v^3),\\
A_4 &= Q_s(2 \beta_2 vA_3 + \beta_2A_2^2 + 3\beta_3v^2 A_2 + \beta_4v^4).
\end{split}
\]
For $N \geq 5$, we can write
\[
\begin{split}
A_N= Q_s(\beta_N v^N) + \mathcal{Q}_{N}(\beta_2, \beta_3, \ldots, \beta_{N-1}),
\end{split}
\]
where $\mathcal{Q}_{N}(\beta_2, \beta_3, \ldots, \beta_{N-1})$ contains the terms involved only with $\beta_2, \ldots, \beta_{N-1}$.
Note that $v$ appears $j$ times in each $A_j$, for $j = 2, 3, 4$.
Therefore, we introduce the notation $A_2^{ij}$ to denote the result if we replace $v$ by $v_i, v_j$ in $A_2$ in order, and similarly the notations
$A_3^{ijk}$, $A_4^{ijkl}$,
such that
\[
A_2 = \sum_{i,j} \ep_i \ep_j A_2^{ij}, \quad
A_3 = \sum_{i,j, k} \ep_i\ep_j\ep_k  A_3^{ijk},\quad
A_4 =\sum_{i,j, k,l} \ep_i\ep_j\ep_k \ep_l A_4^{ijkl}.
\]
More explicitly, for $1 \leq i,j,k,l \leq 3$, we have
\beq\label{eq_A}
\begin{split}
A_2^{ij} &= Q_s(\beta_2 v_iv_j), \\
A_3^{ijk} &= Q_s(2\beta_2 v_iA_2^{jk} +\beta_3 v_i v_j v_k),\\
A_4^{ijkl} &= Q_s(2 \beta_2 v_iA_3^{jkl} + \beta_2A_2^{ij}A_2^{kl} + 3\beta_3v_i v_j A_2^{kl} + \beta_4v_iv_jv_kv_l).
\end{split}
\eeq
In the following, for $N \geq 3$, we define
\[
\lU_N = \partial_{\epsilon_1}^{N-2}\partial_{\epsilon_2}\partial_{\epsilon_3} u |_{\epsilon_1 = \epsilon_2 = \epsilon_3=0}.
\]
Note that $\lU_N$ is not the $N$th order linearization of the scattering operator but they are  related by
\beq
\begin{split}\label{eq_LambdaU3}
    \partial_{\epsilon_1}^{N-2}\partial_{\epsilon_2}\partial_{\epsilon_3} \lN (\ep_1\Upsilon_1 + \ep_2\Upsilon_2+ \ep_3\Upsilon_3) |_{\epsilon_1 = \epsilon_2 = \epsilon_3=0} = \lR_+(\lU_N).
\end{split}
\eeq
\subsection{Singularities of nonlinear interaction when $N \geq 3$}
First by \cite[Lemma 3.3]{Lassas2018}, one has
\[
\beta_2 v_iv_j \in I^{\mu+1,\mu}(\Lambda_{ij}, \Lambda_i) + I^{\mu+1, \mu}(\Lambda_{ij}, \Lambda_j), \text{ when $i \neq j$}.
\]
In the following, let $(q, \zj) \in \Lambda_j$ for $j =1,2,3$.
With $K_i, K_j$ intersecting transversally, any $(q, \zeta) \in \Lambda_{ij}$ has a unique decomposition $\zeta = \zi + \zj$.
Away from
$\Lambda_i$ and $\Lambda_j$, the principal symbol of $\beta_2 v_iv_j$ equals to
\[
-(2\pi)^{-1}\beta_2 {\sigmp}(v_i) (q, \zeta^i) {\sigmp}(v_j) (q, \zeta^j)
\]
at $(q, \zeta) \in \Lambda_{ij}$.
In $\nxxi$, {with $v_iv_j$ vanishes on $S_-$}, we can identify $Q_s$ by $Q_g$ (the causal inverse of $\sq_g$ on $\tM$) to have
\[
A_2^{ij} \in I^{\mu, \mu-1}(\Lambda_{ij}, \Lambda_i) + I^{\mu, \mu-1}(\Lambda_{ij}, \Lambda_j),
\]
by \cite[Lemma 3.4]{Lassas2018}.
Additionally,
at $(q, \zeta) \in \Lambda_{ij}$ away from $\Lambda_i$ and $\Lambda_j$,
the principal symbol equals to
\[
\begin{split}
    {\sigmp}(A_2^{ij}) (q, \zeta) &=  \frac{- (2\pi)^{-1}}{|\zeta^i+ \zeta^j|_{g^*}^{2}} \beta_2(q)
    {\sigmp}(v_i) (q, \zeta^{i}) {\sigmp}(v_j) (q, \zeta^{j}).
\end{split}
\]
Next, recall the singularities of $Q_s(\beta_3 v_1 v_2v_3)$ in Proposition \ref{pp_v1v2v3} and \ref{pp_U3}.
Following a similar analysis, we have the following proposition describing the singularities of $\lU_3$, see also \cite{Lassas2018} for more details.

\begin{prop}\label{pp_A3ijk}
    Suppose the submanifolds $K_1, K_2, K_3$ intersect 3-transversally at $K_{123}$.
    Let $\Lambda_{123}^g$ and $\Lambda^{(1)}$ be defined as in Section \ref{subsec_constructwaves}.
    We decompose $\lU_3$ as
    \[
    \lU_3 = \lU_{3,0} + \lU_{3,1} = -2\sum_{(i,j,k) \in \Sigma(3)} Q_s(\beta_2 v_iA_2^{jk}) -6Q_s(\beta_3 v_1 v_2 v_3).
    \]
    In $\nxxi$ away from $\Lambda^{(1)}$, we have
    \[
    \lU_{3,0} \in I^{3\mu+ \frac{1}{2}, -\frac{1}{2}}(\Lambda_{123}, \Lambda_{123}^g), \quad \lU_{3,1} \in I^{3\mu - \frac{1}{2}, -\frac{3}{2}}(\Lambda_{123}, \Lambda_{123}^g).
    \]
    In particular, let $(y, \eta) \in L^{+,*}M$ lies along a future pointing null bicharactersitic of $\sq_g$ starting from $(q, \zeta) \in \Lambda_{123}$.
    Suppose $(y, \eta)$ is away from $\Lambda^{(1)}$ and before the first cut point of $q$.
    Then $(y, \eta) \in \nxxi$ and the principal symbol is given by
    \[
    \begin{split}
    \sigma_p(\lU_{3,0})(y, \eta) &=
    -6(2\pi)^{-2}
    \sigma_p(Q_g)(y, \eta, q, \zeta)\beta_3(q)
    \prod_{m=1}^3\sigma_p(v_m) (q, \zeta^m), \\
    \sigma_p(\lU_{3,1})(y, \eta) &=
    -2(2\pi)^{-2}
    \sigma_p(Q_g)(y, \eta, q, \zeta)\beta^2_2(q) \sum_{(i,j,k) \in \Sigma(3)}\frac{1}{|\zeta^i+ \zeta^j|_{g^*}^{2}}
    \prod_{m=1}^3\sigma_p(v_m) (q, \zeta^m),
    \end{split}
    \]
    where $\zeta = \zeta^1 + \zeta^2 + \zeta^3$ with $(q, \zeta^j) \in \Lambda_j$.
\end{prop}
Next, we consider a special type of conormal distributions first introduced by Rauch and Reed, and then discussed in \cite{Piriou1988}. This class of conormal distributions vanish to certain order at $K$, where $K$ is a $1$-codimensional submanifold of $M$.
Note the order of conormal distributions in \cite{Piriou1988} is the same as that of symbols and we shift it by $-\frac{n}{4} + \frac{N}{2}$ following the definition in \cite{MR2304165}.
The following analysis can also be found in \cite{UZ_acoustic}.
\begin{df}
    Let $ m < -1$ and  $k(m) \in \mathbb{N}$ such that $ -m -2 \leq k(m) < -m -1$.
    We say $u \in \Ir^{m - \frac{n}{4} + \frac{1}{2}}(K)$ if $u \in \Ical^{m - \frac{n}{4} + \frac{1}{2}}(K)$ vanishing to order $k(m) + 1$ at $K$.
\end{df}
By \cite[Propostion 2.3 and 2.4]{Piriou1988}, a distribution $u \in \Ir^{m}(K)$ if and only if there exists $h \in C^\infty$ vanishing to order $k(m)$ such that $u = hv$ with $v \in \Ir^{m + k(m)}(K)$, for $m <-1$.
Additionally, by \cite{Barreto2021a}, for any $u \in \Ical^{m - \frac{n}{4} + \frac{1}{2}}(K)$ with compact support,
by subtracting a compactly supported smooth function whose derivatives at $K$ up to order $k(m) +1$ coincide with those of $u$, one can modify $u$ such that $u \in  \Ir^{m - \frac{n}{4} + \frac{1}{2}}(K)$.
This can be done since we can show that $u$ is continuous up to order $k(m) +1$ at $K$, for $m<-1$.
In particular, the receding waves we construct in Section \ref{sec_receding}.


Moreover, we have the following lemma, see \cite[Theorem 2.1]{Piriou1988}
and also \cite[Proposition 2.5]{Barreto2021a}.
\begin{lm}\label{cl_Piriou}
    Let $m\in \mathbb{N}$ with $m \geq 2$.
    Suppose $v_1 \in \Ir^\mu(\Lambda_1)$.
    Then
    \[
    v_1^m \in \Ir^{\mu + (m -1)(\mu+\frac{3}{2})}(\Lambda_1),
    \]
    with the principal symbol given by $m$-fold fiberwise convolution
    \[
    \begin{split}
        \sigmp(v_1^m)
    = (2\pi)^{-(m-1)} \underbrace{{\sigmp}(v_1) \ast \ldots \ast {\sigmp}(v_1) }_{m}.
\end{split}
\]
\end{lm}
Now we state the following lemma about $v_1^m v_2 v_3$, see also \cite[Lemma 8]{UZ_acoustic} for four waves interaction.
\begin{lm}\label{lm_psmodel_v}
Let $m\in \mathbb{N}$ with $m \geq 2$.
Suppose $K_1,K_2,K_3$ intersect 3-transversally at a point $q \in \nxxi$.
Suppose $v_1 \in \Ir^{\mu} (\Lambda_1)$ and $v_j \in \Ical^{\mu} (\Lambda_j)$ for $j = 2, 3$.
Then in $\nxxi$
away from $\Lambda^{(1)}$, we have
\[
v_1^m v_2 v_3 \in \Ical^{3\mu + 2 +  (m-1)(\mu+ \frac{3}{2})}(\Lambda_{123}).
\]
Moreover, for $(q, \zeta) \in \Lambda_{123}$,  the principal symbol is given by
\beq\label{eq_psmodel_v}
    \sigmp(v_1^m v_2 v_3 )(q, \zeta) = (2\pi)^{-(m+1)} \sigmp(v_1^m)(q, \zeta^{1})\sigmp(v_2)(q, \zeta^{2})\sigmp(v_3)(q, \zeta^{3}),
\eeq
where the decomposition $\zeta = \sum_{j=1}^3 \zj$ with $(q, \zj) \in \Lambda_j$ is unique.
Note that $\sigmp(v_1^{m})$ is  given by the $m$-fold fiberwise convolution in Lemma \ref{cl_Piriou}.
\end{lm}

\begin{lm}\label{lm_psmodel}
Let $m\in \mathbb{N}$ with $m \geq 2$.
Suppose $K_1, K_2, K_3$ intersect 3-transversally at $K_{123}$.
Let $\Lambda_{123}^g$ and $\Lambda^{(1)}$ be defined as in Section \ref{subsec_constructwaves}.
Let $v_1 \in \Ir^{\mu} (\Lambda_1)$ and $v_j \in \Ical^{\mu} (\Lambda_j)$ for $j = 2, 3$.
Let $(y, \eta) \in L^{+,*}M$ lie along a future pointing null bicharactersitic of $\sq_g$ starting from $(q, \zeta) \in \Lambda_{123}$.
Suppose $(y, \eta)$ is away from $\Lambda^{(1)}$ and before the first cut point of $q$.
Then $(y, \eta) \in \nxxi$ and
\[
\begin{split}
&\sigma_p(Q_s(\partial_t^2({v}_1^{m} v_2 v_3))(y, \eta) \\
=&
-(2\pi)^{-(m+1)}
\sigma_p(Q_g)(y, \eta, q, \zeta)\beta_N(q)
\sigmp(v_1^m)(q, \zeta^{1})\sigmp(v_2)(q, \zeta^{2})\sigmp(v_3)(q, \zeta^{3}),
\end{split}
\]
where the decomposition $\zeta = \sum_{j=1}^3 \zj$ with $(q, \zj) \in \Lambda_j$ is unique and the homogeneous term $\sigmp(v_1^m)$ is given by the $m$-fold fiberwise convolution in  Lemma \ref{cl_Piriou}.
\end{lm}

\subsection{Proof of Theorem \ref{mainthm}}
To prove Theorem \ref{mainthm}, we use the following strategy in each step of the layer stripping method, based on the proof of Theorem \ref{thm_cubic}.
Indeed, the only difference lies in the detection condition of new singularities.
To determine the three-to-one scattering relation,
instead of just using
the third-order linearization,
in Definition \ref{def_D} - \ref{def_D_step4}
we consider whether
$(\yb, \etab)$
is contained in the wave front set of
\[
\partial_{\epsilon_1}^{m-2}\partial_{\epsilon_2}\partial_{\epsilon_3} \mcn(\ep_1 \Ups_1 + \ep_2 \Ups_2 + \ep_3 \Ups_3) |_{\epsilon_1 = \epsilon_2 = \epsilon_3=0}
\]
for some $m \geq 3$,
where $(\yb, \etab)$ is the restriction of $\ups_0^\sharp = (y, \eta)$ to $T^*\Spl$.

With such a detection condition in each step, we can prove Proposition \ref{pp_R1}
using the same idea as before and an inductive procedure.
It remains to prove Proposition \ref{pp_R2} for this setting.
Indeed, let $(\ups_0, \ups_1, \ups_2, \ups_3)$ and $q$ be as in Proposition \ref{pp_R2}.
In particular, we have
\[
q \in \Gammaset \cap \overline{W}.
\]
On the one hand, if there exists $m \geq 3$ such that $\beta_m(T,X) \neq 0 $ at $q$,
then we consider the smallest integer $m_0 \geq 3$ to make this happen.
With $\beta_m(q) = 0$ for $m < m_0$, we have
$\lU_{m_0} = Q_s(\beta_N v_1^{m_0-2}v_2 v_3)$ and therefore
\[
\partial_{\epsilon_1}^{m_0-2}\partial_{\epsilon_2}\partial_{\epsilon_3} \mcn(\ep_1 \Ups_1 + \ep_2 \Ups_2 + \ep_3 \Ups_3) |_{\epsilon_1 = \epsilon_2 = \epsilon_3=0}
= \lR_+[\lU_{m_0}].
\]
Then Lemma \ref{lm_psmodel} implies
\[
\sigma_p(\lU_{m_0})(y, \eta) \neq 0
\]
as $\beta_{m_0}(q) \neq 0$.
The same analysis as before shows $(\yb, \etab)$ is contained in the wave front set of $\lR_+[\lU_{m_0}]$ and thus the detection condition is satisfied.
On the other hand, if $\beta_m(T,X) = 0$ near $q$ for all $m \geq 3$.
We must have $\beta_2(q) \neq 0$.
This implies
\[
\partial_{\epsilon_1}\partial_{\epsilon_2}\partial_{\epsilon_3} \mcn(\ep_1 \Ups_1 + \ep_2 \Ups_2 + \ep_3 \Ups_3) |_{\epsilon_1 = \epsilon_2 = \epsilon_3=0}
=\lR_+[\lU_{3}] =\lR_+[\lU_{3,1}]
\]
and by Lemma \ref{lm_psmodel} we can show $(\yb, \etab)$ is contained in the desired wave front set and thus the detection condition is satisfied.
With Proposition \ref{pp_R2} proved, we determine the three-to-one scattering relation in each step as before and then reconstruct the scattering light observation sets as before.